\documentclass[11pt,reqno]{amsart}

\usepackage{lineno,hyperref}

\allowdisplaybreaks[4]

\usepackage{amsmath}
\usepackage{amssymb}
\usepackage{mathrsfs}
\usepackage{amsthm}
\usepackage{bm}
\usepackage{graphicx}
\usepackage{xcolor}   
\usepackage{tikz} 
\usetikzlibrary{arrows,shapes,chains}  
\usepackage{booktabs}
\usepackage{float}
\usepackage{subfig}
\usepackage{geometry}
\usepackage{color}

\newtheorem{Def}{Definition}[section]
\newtheorem{lem}[Def]{Lemma}
\newtheorem{theo}[Def]{Theorem}
\newtheorem{pro}[Def]{Proposition}

\newtheorem{ex}[Def]{Example}
\newtheorem{assum}{Assumption}
\newtheorem{cor}[Def]{Corollary}
\newtheorem{con}[Def]{Condition}
\usepackage{pgf}
\definecolor{Green}{RGB}{0,128,0}

\newcommand{\LL}{\langle}
\newcommand{\RR}{\rangle}

\newcommand{\tr}{\text{tr}}
\newcommand{\mcal}{\mathcal}

\newcommand{\mbb}{\mathbb}
\newcommand{\mbf}{\mathbf}

\newcommand{\ud}{\mathrm d}
\newcommand{\PD}{\partial}
\numberwithin{equation}{section}
\allowdisplaybreaks \allowdisplaybreaks[4]

\begin{document}

\title[Asymptotic error distribution for SPDE]{Asymptotic error distribution of  numerical methods for parabolic  SPDEs with multiplicative noise}

\author{Jialin Hong}
\address{SKLMS, ICMSEC, Academy of Mathematics and Systems Science, Chinese Academy of Sciences, Beijing 100190, China, and School of Mathematical Sciences, University of Chinese Academy of Sciences, Beijing 100049, China}
\email{hjl@lsec.cc.ac.cn}

\author{Diancong Jin}
\address{School of Mathematics and Statistics, Huazhong University of Science and Technology, Wuhan 430074, China;
	Hubei Key Laboratory of Engineering Modeling and Scientific Computing, Huazhong University of Science and Technology, Wuhan 430074, China}
\email{jindc@hust.edu.cn (Corresponding author)}

\author{Xu Wang}
\address{SKLMS, ICMSEC, Academy of Mathematics and Systems Science, Chinese Academy of Sciences, Beijing 100190, China, and School of Mathematical Sciences, University of Chinese Academy of Sciences, Beijing 100049, China}
\email{wangxu@lsec.cc.ac.cn}

\thanks{This work is supported by the National key R\&D Program of China (No. 2024YFA1015900), the National Natural Science Foundation of China (No. 12201228 and No. 12288201),  and the CAS Project for Young Scientists in Basic Research (YSBR-087).}


\begin{abstract}
This paper aims to investigate the asymptotic error distribution  of several numerical methods for stochastic partial differential equations (SPDEs) with multiplicative noise. Firstly, we give the limit distribution of the normalized error process of the exponential Euler method in  $\dot{H}^\eta$ for some $\eta>0$.  A key finding is that the asymptotic error in distribution of the exponential Euler method is governed by a linear SPDE driven by infinitely many independent 
$Q$-Wiener processes. This characteristic represents a significant difference from numerical methods for both stochastic ordinary differential equations  and SPDEs with additive noise. Secondly, as applications of the above result, we  derive the asymptotic error  distribution of a full discretization based on the temporal exponential Euler method and the spatial finite element method. As a concrete illustration, we provide the pointwise limit distribution of the normalized error process when the exponential Euler method is applied to a specific class of stochastic heat equations.
Finally, by studying the asymptotic error of the spatial semi-discrete spectral Galerkin method,  we demonstrate  that the actual strong convergence speed of spatial semi-discrete numerical methods may be highly problem-dependent, rather than universally predictable.
\end{abstract}

\maketitle


\section{Introduction}
The asymptotic error distribution refers to the limit distribution of the
normalized error process of a numerical method applied to a  stochastic system, where the normalization of the error process is conducted based on the strong convergence order of the numerical method.  Consequently, the existence of a nontrivial (non-zero) limit distribution implies that the strong convergence order is exact. The asymptotic error distribution also provides valuable insights in the optimal choice of tuning parameters
for the multilevel Monte Carlo method \cite{CLTMonte15} and the error structure \cite{Bouleau}. For finite-dimensional stochastic systems, the asymptotic error distribution of numerical methods has been extensively studied since the pioneering work by Kurtz and Protter \cite{Protter1991}. Concerning this topic, we refer the readers to \cite{Protter1998AOP,Jin25,Protter2020SPA} for stochastic ordinary differential equations (SODEs) driven by standard Brownian motions, to \cite{HuAAP,Neuenkirch,HuBIT}  for SODEs driven by fractional Brownian motions, to \cite{Fukasawa2023,Nualart2023} for stochastic integral equations, and to \cite{Wufuke} for Mckean-Vlasov SODEs.

In contrast to the extensive studies on finite-dimensional stochastic systems, the investigation into
the asymptotic error distribution of numerical methods for stochastic partial differential equations
(SPDEs) remains relatively nascent.  The recent work \cite{HJWY24} addressed this gap by establishing the asymptotic error distribution of the accelerated exponential Euler method for parabolic SPDEs with additive noise. This study developed a uniform approximation theorem for convergence in distribution to tackle the convergence in distribution of stochastic integrals with respect to $Q$-Wiener processes.  

A common feature observed in all aforementioned literature---covering both SODEs and SPDEs with additive noise---is that the asymptotic error of the numerical method is typically governed by a linear SODE or SPDE driven by only a finite number of additional independent Brownian motions or 
$Q$-Wiener processes.
This raises a critical question: Does this characteristic remain valid for numerical methods applied to SPDEs with multiplicative noise?

In this paper, we are devoted to  answering this question  by studying the asymptotic error distribution of  several numerical methods applied to the following parabolic SPDEs with multiplicative noise:
\begin{align}\label{SPDE}
	\begin{cases}
	\ud X(t)=AX(t)\ud t+F(X(t))\ud t+G(X(t))\ud W(t),\quad t\in(0,T],\\
	X(0)=X_0\in H,
	\end{cases}
\end{align}
where $H$ is a separable Hilbert space and $W$ is a $U$-valued $Q$-Wiener process with $U$ being another Hilbert space. The assumptions on the unbounded linear operator $A$, and the coefficients $F$ and $G$ will be specified in Section \ref{Sec2.1}.  Under Assumption \ref{assum1}, \eqref{SPDE} admits a unique mild solution given by 
\begin{align}\label{mildsoution}
	X(t)=E(t)X_0+\int_0^tE(t-s)F(X(s))\ud s+\int_0^tE(t-s)G(X(s))\ud W(s),\quad t\in[0,T],
\end{align}
where $\{E(t)\}_{t\ge 0}$ is the $C_0$-semigroup generated by $A$.  In order to leverage the smoothing property of $\{E(t)\}_{t\ge 0}$, we apply the exponential  Euler method to discretize \eqref{SPDE} and obtain the temporally continuous numerical solution $\{X^m(t)\}_{t\in[0,T]}$ (see \eqref{Xmmild}). Further, we show in Lemma \ref{strongorder} that $X^m(t)$  converges to
$X(t)$ with order $\frac{1}{2}$ in $\mbf L^p(\Omega;{H})$.  To verify that this strong convergence order is exact, we study the limit distribution of the normalized error process $U^m(t):=m^\frac{1}{2}(X^m(t)-X(t))$ in $H$.

Based on the uniform approximation theorem
for convergence in distribution (see \cite[Theorem 3.2]{HJWY24}), we prove in Theorem \ref{maintheo} that $U^m(t)$ converges in distribution to $U(t)$. We would like to point out that, different from the convergence in $H$ obtained in \cite{HJWY24}, we indeed  show that $U^m(t)$ converges in distribution to some process $U(t)$  in  $\dot{H}^\eta$ for some  $\eta> 0$, by sufficiently leveraging the smoothing property of $\{E(t)\}_{t\ge 0}$. It turns out that the limit distribution $U$ solves a linear SPDE driven by additional infinitely many independent $Q$-Wiener processes $\widetilde W^l$ with $l\in\mbb N^+$. In this way, we identify the characteristic of the asymptotic error distribution of the temporal semi-discrete exponential Euler method for SPDEs with multiplicative noise.  

As applications of Theorem \ref{maintheo}, we give the asymptotic error distribution of 
the exponential Euler method applied to  SODEs  and that of a full discretization applied to \eqref{SPDE} based on the temporal exponential Euler method and the spatial finite element method; see Corollaries \ref{cor-SODE} and \ref{fulldiscre}.  In addition, we consider a stochastic heat equation as a concrete example of \eqref{SPDE}, and obtain the asymptotic error distribution of its exponential Euler method. Especially, when the diffusion term is affine with respect to\ the state variable and the space is of one dimension, we establish the limit distribution of $U^m$ at any given $(t,x)$ by means of the convergence in distribution of $U^m$ in $\dot{H}^\eta$ for any $\eta\in(0,1)$ (see Theorem \ref{SHEasymerror'}). This kind of result on the pointwise convergence in distribution has not been reported  anywhere else to the best of our knowledge. 

Finally we investigate the asymptotic error of the spatial semi-discrete spectral Galerkin method applied to \eqref{SPDE}. Interestingly, the limit distribution of the corresponding  error process, weighed by the strong convergence  order, is zero according to Theorem \ref{YNerrorasym}. 
We further demonstrate  by a heuristic example (Example \ref{Example}) that the exact strong convergence speed of spatial semi-discrete numerical methods for SPDEs is highly problem-dependent. 

Let us state the main contributions of this work as follows.  
\begin{itemize}
	\item  We  establish the asymptotic error distribution of numerical methods applied to SPDEs with multiplicative noise for the first time, and identify the characteristic of  the asymptotic error of the  exponential Euler method.  
	
	\item It is shown that the convergence in distribution of the normalized error  process  still holds in $\dot{H}^\eta$ with some $\eta>0$, generalizing the existing convergence result in $H$. On basis of it,  we provide  the pointwise limit distribution of the normalized error process $U^m$ of numerical methods applied to stochastic heat equations. 
	
	\item We reveal that the asymptotic error or the exact strong convergence speed of the spatial spectral Galerkin method for SPDEs is highly problem-dependent.  
\end{itemize}

The remainder of this paper is organized as follows. Section \ref{Sec2} introduces some necessary notations and the assumptions imposed on the equation \eqref{SPDE}, and gives the strong convergence  of the exponential Euler method. In Section \ref{Sec3}, we establish the asymptotic error distribution for the exponential Euler method. Section \ref{Sec4} presents two key applications of the main theoretical result, including a concrete example of the equation \eqref{SPDE}. The asymptotic error of the spatial spectral Galerkin method is analyzed in Section \ref{Sec5}. Finally, Section \ref{Sec6} provides the conclusions of this study and outlines future research directions.

\section{Preliminaries}\label{Sec2}
In this section, we give the assumptions on \eqref{SPDE} and present the strong convergence order of the exponential Euler method. We begin with some notations.

For Banach spaces  $(\mcal X,\|\cdot\|_{\mcal X})$ and $(\mcal Y,\|\cdot\|_{\mcal Y})$, denote by $\mcal L(\mcal X,\mcal Y)$ the space of bounded linear operators from $\mcal X$ to $\mcal Y$ endowed with the usual operator norm $\|\cdot\|_{\mcal L(\mcal X,\mcal Y)}$, and denote $\mcal L(\mcal X):=\mcal L(\mcal X,\mcal X)$ for short. Denote by $Id_{\mcal X}$ the identity operator on $\mcal X$.
Denote by $\mbf C(\mcal X;\mcal Y)$ the space of $\mcal Y$-valued continuous functions defined on $\mcal X$ endowed with the norm $\|f\|_{\mbf C(\mcal X;\mcal Y)}:=\sup_{x\in\mcal X}\|f(x)\|_{\mcal Y}$, and by $\mbf C_b(\mcal X;\mcal Y)$ the space of bounded functions in $\mbf C(\mcal X;\mcal Y)$. Denote by $|\cdot|$ the $2$-norm of a vector or matrix.

Let $\big(\Omega,\mcal F,\mbf P\big)$ be a completed probability space and $\mbf E$ denote the expectation operator with respect to the probability measure $\mbf P$. For $p\ge1$, let $\mbf L^p(\Omega;\mcal X)$ be the space of $p$-fold integrable functions $f:\Omega\to \mcal X$ endowed with the norm $\|f\|_{\mbf L^p(\Omega;\mcal X)}:=\big(\mbf E\|f\|^p_{\mcal X}\big)^{1/p}$.  

Throughout the paper, we use $K(a_1,a_2,\ldots,a_l)$ to represent some generic constant depending on  parameters $a_1,a_2,\ldots,a_l$, which may vary for each appearance, and use the notation `$\overset{d}{\Longrightarrow}$' to stand for the convergence in distribution for random variables.

Let $\{X_n\}_{n=1}^\infty$ be a sequence of random variables defined on $(\Omega,\mathcal{F},\mathbf{P})$ taking values in a Polish space $\mcal E$. Let $(\tilde{\Omega},\tilde{\mathcal{F}},\tilde{\mathbf{P}})$  be an extension of $(\Omega,\mathcal{F},\mathbf{P})$ and $X$ be an $\mcal E$-valued random variable on this extension.
Then $X_n$ is said to stably converge in law to $X$ in $\mcal E$, denoted by `$X_n\overset{stably}{\Longrightarrow}X$ in $\mcal E$', if 
$$\lim_{n\to\infty}\mathbf{E}[Z f(X_n)]=\tilde{\mathbf{E}}[Z f(X)]$$
for all $f\in\mbf C_b(\mcal E;\mbb R)$ and all bounded random variable $Z$, where $\tilde{\mathbf E}$ denotes the expectation with respect to $\tilde{\mathbf{P}}$. From the above definition, we know that $X_n\overset{stably}{\Longrightarrow}X$ implies $X_n\overset{d}{\Longrightarrow}X$.
We refer the readers to \cite{Jacod1997} for more details of stable convergence in law.  
\subsection{Setting}\label{Sec2.1}
Throughout  this paper, let $(H, \LL\cdot,\cdot\RR, \|\cdot\|)$ and $(U,\LL\cdot,\cdot\RR_U,\|\cdot\|_U)$ be two  separable Hilbert spaces. Let  $\mcal L_2(U,H)$ stand for the space of Hilbert--Schmidt operators $\Gamma:U\to H$ equipped with the Hilbert--Schmidt norm $\|\Gamma\|_{\mcal L_2(U,H)}:=\big(\sum_{i=1}^{\infty}\|\Gamma \varphi_i\|^2\big)^{1/2}$, where $\{\varphi_i\}_{i\in\mbb N^+}$ is any orthonormal basis of $U$. It is well-known that the following properties hold for Hilbert--Schmidt operators. 

(1) It holds that $\|\Gamma\|_{\mcal L(U,H)}\le \|\Gamma\|_{\mcal L_2(U,H)}$ for any $\Gamma\in\mcal L_2(U,H)$. 

(2) Let $G_1$ and $G_2$ be another two separable Hilbert spaces and $S_1 \in \mcal L(G_1,U) $, $S_2 \in \mcal L(H,G_2)$, and $\Gamma \in \mcal L_2(U,H)$. Then $S_2 \Gamma S_1 \in \mcal L_2(G_1,G_2) $ and $\| S_2\Gamma S_1 \|_{\mcal L_2(G_1,G_2)}\le \|S_1\|_{\mcal L(G_1,U)}\|\Gamma\|_{\mcal L_2(U,H)}\|S_2\|_{\mcal L(H,G_2)}.$

Without extra statement, we always suppose that $\{W(t)\}_{t\in[0,T]}$ is a $U$-valued $Q$-Wiener process on $\big(\Omega,\mcal F,\mbf P\big)$ with respect to a normal filtration $\{\mcal F_t\}_{t\in[0,T]}$, where $Q\in\mcal L(U)$ is a positive definite and symmetric operator of finite trace $\tr(Q):=\sum_{i=1}^\infty q_i<\infty$ with $\{q_i\}_{i\in\mbb N^+}$ being its eigenvalues. Then, $W$ has the following  expansion:
\begin{align}\label{W}
W(t)=\sum_{i=1}^\infty Q^{\frac{1}{2}}h_i\beta_i(t)=\sum_{i=1}^\infty \sqrt{q_i}h_i\beta_i(t),\quad t\in[0,T],
\end{align}
where $\{h_i\}_{i\in\mbb N^+}$ is an orthonormal basis of $U$ consisting of eigenvectors of $Q$ such that $Qh_i=q_ih_i$ for $i\in\mbb N^+$, and $\{\beta_i\}_{i\in\mbb N^+}$ is a family of independent real-valued standard Brownian motions defined on $\big(\Omega,\mcal F,\{\mcal F_t\}_{t\in[0,T]},\mbf P\big)$. In addition, we can define the fractional power of $Q$. For any $r\in\mbb R$, define $Q^r:\text{Dom}(Q^r)\to U$ by $Q^ru:=\sum_{i=1}^{\infty}q_i^ru_i h_i$, where
$$u\in\text{Dom}(Q^r):=\Big\{u=\sum_{i=1}^{\infty}u_ih_i:u_i\in\mbb R,~\sum_{i=1}^{\infty}q_i^{2r}u_i^2<\infty\Big\}.$$

Further, we introduce the Cameron--Martin space $U_0=Q^{\frac{1}{2}}(U)$, which is a separable Hilbert space if equipped with the inner product $\LL u_0,v_0\RR_{U_0}:=\LL Q^{-\frac{1}{2}}u_0,Q^{-\frac{1}{2}}v_0\RR_U$ for all $u_0,v_0\in U_0$. It holds that $\{Q^{\frac{1}{2}}h_i\}_{i\in\mbb N^+}$ is an orthonormal basis of $U_0$.

Let $(-A):\text{Dom}(A)\subseteq H\to H$ be a linear, densely defined, self-adjoint, and positive definite operator, which is with compact inverse. In this setting, $A$ is the infinitesimal generator of a $C_0$-semigroup of contractions $\{E(t)=e^{tA}\}_{t\ge0}$ on $H$. In addition, there exists an increasing sequence of positive numbers $\{\lambda_i\}_{i\in\mbb N^+}$ and an orthonormal basis $\{e_i\}_{i\in\mbb N^+}$ of $H$ such that $-Ae_i=\lambda_ie_i$ with $0<\lambda_1\le\lambda_2\le\cdots\le\lambda_n(\to\infty)$. For any $r\in\mbb R$, define the operator $(-A)^{\frac{r}{2}}$ by $(-A)^{\frac{r}{2}}x:=\sum_{i=1}^{\infty}\lambda_i^{\frac{r}{2}}x_ie_i$ for all 
$$x\in\text{Dom}((-A)^{\frac{r}{2}}):=\Big\{x=\sum_{i=1}^\infty x_ie_i: x_i\in\mbb R,~\|x\|_r^2:=\|(-A)^{\frac{r}{2}}x\|^2=\sum_{i=1}^\infty \lambda_i^rx_i^2<\infty \Big\}.$$
Denote $\dot{H}^r:=\text{Dom}((-A)^{\frac{r}{2}})$, which is a Hilbert space equipped with the inner product $\LL u,v\RR_{r}:=\LL (-A)^{\frac{r}{2}}u,(-A)^{\frac{r}{2}}v\RR$  for $u,v\in \dot{H}^r$. Especially, it holds  $H=\dot{H}^0$.  It is easy to see that for  $\alpha\le\beta$,  
\begin{align}\label{norm}
\|x\|_\alpha\le \lambda_{1}^{\frac{\alpha-\beta}{2}}\|x\|_\beta,\quad x\in \dot{H}^\beta. 
\end{align}
In addition, it can be directly shown that the following interpolation inequality hold. 
\begin{pro}\label{interpolation}
	Let $p,q\in\mbb R$ with $p<q$, and $\gamma\in(p,q)$. Then it holds 
	\begin{align*}
		\|x\|_\gamma\le \|x\|_p^{1-\theta_\gamma}\|x\|_q^{\theta_\gamma},\quad\forall~x\in\dot{H}^q,
	\end{align*}
	where  $\theta_\gamma=\frac{\gamma-p}{q-p}\in(0,1)$.
\end{pro}
\begin{proof}
Noting that $\gamma=(1-\theta_\gamma)p+\theta_\gamma q$ and $
\frac1{1-\theta_\gamma},\frac1{\theta_\gamma}>1$, then the H\"older inequality yields for any $x=\sum_{i=1}^{\infty}x_ie_i\in\dot{H}^q$ that
\begin{align*}
	\|x\|_\gamma^2&=\sum_{i=1}^{\infty}\lambda_i^\gamma x_i^2=\sum_{i=1}^{\infty}\lambda_i^{(1-\theta_\gamma)p}x_i^{2(1-\theta_\gamma)}\lambda_i^{\theta_\gamma q}x_i^{2\theta_\gamma} \\ 
	&\le \Big(\sum_{i=1}^\infty \lambda_i^px_i^2\Big)^{1-\theta_\gamma}\Big(\sum_{i=1}^\infty \lambda_i^qx_i^2\Big)^{\theta_\gamma}=\|x\|_p^{2(1-\theta_\gamma)}\|x\|_q^{2\theta_\gamma},
\end{align*}
which finishes the proof.
\end{proof}

Let us also recall some frequently used properties with respect to\ the semigroup $\{E(t)\}_{t\ge 0}$ (cf.\ \cite[Lemma B.9]{Krusebook}):
\begin{align}
&\|(-A)^rE(t)\|_{\mcal L(H)}\le K(r)t^{-r}, \quad t>0, ~r\ge 0,\label{semigroup1}\\
&\|(-A)^{-\rho}(E(t)-Id_H)\|_{\mcal L(H)}\le K(\rho) t^\rho,\quad t>0,~\rho\in[0,1] \label{semigroup2}, \\
&\int_s^t\|(-A)^{\frac{\rho}{2}}E(t-r)u\|^2\ud r\le K(\rho)(t-s)^{1-\rho}\|u\|^2,\quad u\in H,~0\le s<t,~\rho\in[0,1], \label{semigroup3}
\end{align}
where both the constants $K(r)$ and $K(\rho)$ are independent of $t$.

Next, we give the assumptions on the initial value $X_0$,  and the coefficients $F$ and $G$ in \eqref{SPDE}.

\begin{assum}\label{assum1}
The initial value $X_0$ satisfies $\|X_0\|_{\mbf L^p(\Omega;\dot{H}^{1+\sigma})}<\infty$ for some $\sigma\in(0,1)$ and $p\ge 4$. 
The mappings $F:H\to H$ and $G:H\to \mcal L_2^0:=\mcal L_2(U_0,H)$ are globally Lipschitz continuous, i.e., there exists $L_1>0$ such that
\begin{align}
	\|F(u_1)-F(u_2)\|&\le L_1\|u_1-u_2\|,\quad\forall~u_1,u_2\in H, \label{FLip}\\
		\|G(u_1)-G(u_2)\|_{\mcal L_2^0}&\le L_1\|u_1-u_2\|,\quad\forall~u_1,u_2\in H. \label{GLip}
\end{align}
Also, there exists $L_2>0$ such that
\begin{align}
	\|G(u)\|_{\mcal L_2(U_0,\dot{H}^\sigma)}=\|(-A)^{\frac{\sigma}{2}}G(u)\|_{\mcal L_2^0}\le L_2(1+\|u\|_{\sigma}),\quad\forall~u\in\dot{H}^\sigma. \label{Ggrow}
\end{align} 
\end{assum}

Under Assumption \ref{assum1}, the equation \eqref{SPDE} has a unique $p$-fold integrable mild solution $X$, which has the following  spatial and temporal regularity (cf.\ Theorems 2.27 and 2.31 of \cite{Krusebook}):
\begin{align}
	\sup_{t\in[0,T]}\|X(t)\|_{\mbf L^p(\Omega;\dot{H}^{1+\sigma})}&\le K(T)\left(1+\|X_0\|_{\mbf L^p(\Omega;\dot{H}^{1+\sigma})}\right), \label{Xspatial}\\
	\|X(t)-X(s)\|_{\mbf L^p(\Omega;\dot{H}^{\gamma})}&\le K(T,\gamma)|t-s|^{1/2},\quad t,s\in[0,T],~\gamma\in[0,\sigma]. \label{Xtemporal}
\end{align}

In order to derive the asymptotic error distribution for numerical methods, the following assumptions on $F$ and $G$ are further required.

\begin{assum}\label{assum2}
The mappings $F:\dot{H}^\alpha\to H$ and $G:\dot{H}^\alpha\to\mcal L_2^0$ are twice continuously Fr\'echet differentiable for some $\alpha\in[0,\sigma+\frac{1}{2})$ with $\sigma$ being given in Assumption \ref{assum1}. Moreover, there exists a constant $L_3>0$ such that
\begin{align}
	&\|\mcal DF(v)u\|\le L_3\|u\|,\quad\forall~ v\in\dot{H}^\alpha,u\in H, \label{F'}\\
	&\|\mcal D^2F(v)(u_1,u_2)\|\le L_3\|u_1\|_\alpha\|u_2\|_\alpha,\quad\forall~v,u_1,u_2\in\dot{H}^\alpha,\label{F''}\\
	&\|\mcal DG(v)u\|_{\mcal L_2^0}\le L_3\|u\|,\quad\forall~ v\in\dot{H}^\alpha,u\in H,\label{G'}\\
	&\|\mcal D^2G(v)(u_1,u_2)\|_{\mcal L_2^0}\le  L_3\|u_1\|_\alpha\|u_2\|_\alpha,\quad\forall~v,u_1,u_2\in\dot{H}^\alpha. \label{G''}
\end{align}{\tiny }
\end{assum}

In the following, we use the notation $\mcal D^2F(v)u^2:=\mcal D^2F(v)(u,u)$ and similar for $\mcal D^2G$ if  no confusion occurs.

\begin{assum}\label{assum3}
There exist $\beta_1\in(0,1),$ $\beta_2>0,$ and $L_4>0$ such that
\begin{align}
	&\|(-A)^{-\frac{\beta_1}{2}}\mcal DG(v)uQ^{-\frac{\beta_2}{2}}\|_{\mcal L_2^0}\le L_4\|u\|,\quad\forall~v\in\dot{H}^\alpha,~u\in H, \label{assum3.1}\\
	&\|G(v)Q^{-\frac{\beta_2}{2}}\|_{\mcal L_2^0}\le L_4(1+\|v\|),\quad\forall~v\in H.\label{assum3.2}
	\end{align}
\end{assum}

\subsection{Exponential Euler method}
Let $m\in\mbb N^+$, $\tau=\frac{T}{m}$, and $\{t_n=n\tau,~n=0,1,\ldots,m\}$ be the uniform partition of $[0,T]$.
Consider the following exponential Euler method 
\begin{align*}
	\bar {X}_n^m=E(t_n-t_{n-1})\big(\bar{X}^m_{n-1}+\tau F(\bar X^m_{n-1})+G(\bar X^m_{n-1})\Delta W_{n-1}\big),\quad n=1,\ldots,m,
\end{align*}
starting from $\bar X^m_0=X_0$, or equivalently,
\begin{align*}
	\bar{X}^m_n=E(t_n)X_0+\tau\sum_{k=0}^{n-1}E(t_n-t_k)F(\bar X^m_k)+\sum_{k=0}^{n-1}E(t_n-t_k)G(\bar X^m_k)\Delta W_k,\quad n=1,\ldots,m,
\end{align*}
where $\Delta W_k=W(t_{k+1})-W(t_k)$ with $k=0,\ldots,m-1$.
For $t\in[0,T]$, we consider the continuous version of $\{\bar X^m_n,n=0,\ldots,m\}$:
\begin{align}\label{Xmmild}
	X^m(t)=E(t)X_0+\int_0^tE(t-\kappa_m(s))F(X^m(\kappa_m(s)))\ud s+\int_0^t E(t-\kappa_m(s))G(X^m(\kappa_m(s)))\ud W(s),
\end{align}
where $\kappa_m(s)=\lfloor\frac{s}{\tau}\rfloor \tau=\lfloor\frac{ms}{T}\rfloor\frac{T}{m}$. Then it is easily checked that $X^m(t_k)=\bar{X}^m_k$ for $k=0,\ldots,m$.

The following lemma gives the spatial and temporal regularity of $X^m$, whose proof is similar to that of \eqref{Xspatial}--\eqref{Xtemporal} and is given in the appendix. 
\begin{lem}\label{Xmregularity}
Let $\sigma\in(0,1)$  be given such that Assumption \ref{assum1} is fulfilled. Then the following estimates hold.
\begin{itemize}
\item[(i)]	$\sup\limits_{t\in[0,T]}\|X^m(t)\|_{\mbf L^p(\Omega;\dot{H}^{1+\sigma})}\le K(T)(1+\|X_0\|_{\mbf L^p(\Omega;\dot{H}^{1+\sigma})}).$
\item [(ii)] For any $\gamma\in[0,1+\sigma)$, there exists a constant $K(T,\gamma)>0$ independent of $m$ such that 
\[
\|X^m(t)-X^m(s)\|_{\mbf L^p(\Omega;\dot{H}^{\gamma})}\le K(T,\gamma)|t-s|^{\frac{1}{2}\big(1-\max(\gamma-\sigma,0)\big)},\quad t,s\in[0,T].\]
\end{itemize}
\end{lem}

Then we have the following strong convergence of $X^m$.
\begin{theo}\label{strongorder}
	Let $\sigma\in(0,1)$  be given such that Assumption \ref{assum1} is fulfilled. Then for any $\beta\in[0,\sigma]$, there is a constant $K(T,\beta)>0$ independent of $m$ such that
	\begin{align*}
		\sup_{t\in[0,T]}\|X^m(t)-X(t)\|_{\mbf L^p(\Omega;\dot{H}^\beta)}\le K(T,\beta)m^{-\frac{1}{2}}. 
	\end{align*}  
\end{theo}

The proof of Theorem \ref{strongorder} is also postponed to the appendix for the brevity of the paper.
Combining \eqref{Xspatial}, Lemma \ref{Xmregularity}(i), and Theorem \ref{strongorder}, and applying Proposition \ref{interpolation}, we can also obtain the convergence rate of $X^m(t)-X(t)$ in $\mbf L^p(\Omega;\dot{H}^\gamma)$ for any $\gamma\in[0,1+\sigma)$, whose proof is similar to that of Lemma \ref{Xmregularity}(ii) for the case $\gamma\ge\sigma$ and thus is omitted. 
\begin{cor}\label{cor1}
	Let $\sigma\in(0,1)$  be given such that Assumption \ref{assum1} is fulfilled. Then for any $\gamma\in[0,1+\sigma)$, there is a constant $K(T,\gamma)>0$ independent of $m$ such that
	\begin{align*}
		\sup_{t\in[0,T]}\|X^m(t)-X(t)\|_{\mbf L^p(\Omega;\dot{H}^\gamma)}\le K(T,\gamma)m^{-\frac{1}{2}\big(1-\max(\gamma-\sigma,0)\big)}. 
	\end{align*}  
\end{cor}

Theorem \ref{strongorder} indicates that the exponential Euler method has strong convergence order $\frac{1}{2}$ when approximating the equation \eqref{SPDE}. In what follows, we show that this convergence order is optimal by studying the asymptotic error distribution of the exponential Euler method \eqref{Xmmild}. 
To this end, we introduce the  normalized error process 
\begin{align}\label{Um}
U^m(t):=m^{\frac{1}{2}}(X^m(t)-X(t)),\quad t\in[0,T]
\end{align}
and present its limit distribution in $\dot{H}^\eta$ for a relatively small index $\eta\ge 0$ utilizing the following uniform approximation theorem for convergence in distribution.
\begin{theo}(\cite[Theorem 3.2]{HJWY24})\label{uniform approximation}
	Let $(\mcal X,\rho)$ be a metric space with the metric $\rho(\cdot,\cdot)$ and $Z^m, Z^{m,n}, Z^{\infty,n}$, $Z^{\infty,\infty}$ with $m,n\in\mbb N^+$ be  $\mcal X$-valued random variables defined on $(\Omega,{\mcal F},{\mbf P})$. Assume that the following conditions hold:
	\begin{itemize}
		\item [(A1)] For any bounded Lipschitz continuous function  $f:\mcal X\to\mbb R$,
		$$\lim_{n\to\infty}\sup_{m\ge 1}\big|{\mbf E}f(Z^m)-{\mbf E}f(Z^{m,n})\big|=0.$$
		
		\item [(A2)] There exists $n_0\in\mbb N^+$ such that for any $n\ge n_0$, $Z^{m,n}\overset{d}{\Rightarrow}Z^{\infty,n}$ in $\mcal X$ as $m\to\infty$.
		
		\item [(A3)]  $Z^{\infty,n}\overset{d}{\Rightarrow}Z^{\infty,\infty}$ in $\mcal X$ as $n\to\infty$.
	\end{itemize}
	Then it holds $Z^{m}\overset{d}{\Rightarrow}Z^{\infty,\infty}$ in $\mcal X$ as $m\to\infty$.
\end{theo}

\section{Asymptotic error distribution for  exponential Euler method}\label{Sec3}
In this section, we present our main result on the asymptotic error distribution of the temporal semi-discretization based on the exponential Euler method \eqref{Xmmild}, i.e., the limit distribution of the normalized error process $U^m$ defined in \eqref{Um}. 

To derive the limit distribution of $U^m$ in an infinite-dimensional space, based on Theorem \ref{uniform approximation}, a feasible approach is to consider its finite-dimensional approximation and study the  iterative limit distribution of the finite-dimensional approximation. We divide the proof into the following several steps.

\subsection{Auxiliary process $\widetilde U^m$}
In this part, we make a proper decomposition on $U^m$ and define an auxiliary process $\widetilde U^m$ which shares the same limit distribution as $U^m$.

\begin{lem}\label{errordecomposition}
Let $\sigma\in(0,1)$ and $\alpha\in[0,\sigma+\frac12)$ be given such that Assumptions \ref{assum1}--\ref{assum2} are fulfilled. Then  $\sup\limits_{m\ge 1}\sup\limits_{t\in[0,T]}\|U^m(t)\|_{\mbf L^p(\Omega;\dot{H}^{\sigma})}\le K(T)$  and it holds
\begin{align*}
	U^m(t)=&\;\int_0^tE(t-s)\mcal DF(X(s))U^m(s)\ud s+\int_0^tE(t-s)\mcal DG(X(s))U^m(s)\ud W(s) \\
	&\;-m^{\frac{1}{2}}\int_0^tE(t-s)\mcal DG(X^m(\kappa_m(s)))O^m(s)\ud W(s)+R^m(t),
\end{align*}
where 
$
O^m(s):=\int_{\kappa_m(s)}^sE(s-\kappa_m(r))G(X^m(\kappa_m(r)))\ud W(r)
$
and the residual term $R^m$ satisfies that for any  $\eta\in[0,\sigma)$ and sufficiently small $\epsilon>0$,
\[\sup\limits_{t\in[0,T]}\|R^m(t)\|_{\mbf L^2(\Omega;\dot{H}^\eta)}\le K(\eta,\epsilon)m^{-\min\big(\frac{1}{2}-\max(\alpha-\sigma,0),\frac{\sigma-\eta}{2}-\epsilon\big)}.
\]
\end{lem}

\begin{proof}
It follows from Corollary \ref{cor1} that $\sup\limits_{m\ge 1}\sup\limits_{t\in[0,T]}\|U^m(t)\|_{\mbf L^p(\Omega;\dot{H}^{\sigma})}\le K(T)$.
We decompose $U^m(t)$ as
\begin{align}\label{sec3eq1}
	U^m(t)=&\;m^{\frac{1}{2}}\int_0^t\big(E(t-\kappa_m(s))F(X^m(\kappa_m(s)))-E(t-s)F(X(s))\big)\ud s \notag \\
	&\;+m^{\frac{1}{2}}\int_0^t\big(E(t-\kappa_m(s))G(X^m(\kappa_m(s)))-E(t-s)G(X(s))\big)\ud W(s)\notag\\
	=:&\;I^m(t)+II^m(t),\quad t\in[0,T].
\end{align}

Next  we tackle  $I^m$ and $II^m$, respectively.

{\it Step 1.} We first decompose $I^m$ as
\begin{align*}
	I^m(t)=&\;m^{\frac{1}{2}}\int_0^tE(t-s)\big(F(X^m(s))-F(X(s))\big)\ud s\\
	&+m^{\frac{1}{2}}\int_0^tE(t-s)\big(E(s-\kappa_m(s))-Id_H\big)F(X^m(\kappa_m(s)))\ud s\\
	&-m^{\frac{1}{2}}\int_0^tE(t-s)\big(F(X^m(s))-F(X^m(\kappa_m(s)))\big)\ud s\\
	=:&\;A_1^m(t)+A_2^m(t)+A_3^m(t).
\end{align*}

The term $A_1^m$ can be further expanded as
\begin{align*}
	A_1^m(t)=\int_0^tE(t-s)\mcal DF(X(s))U^m(s)\ud s+A_{1,1}^m(t)
\end{align*}
with
\begin{align*}
	A_{1,1}^m(t):=m^{\frac{1}{2}}\int_0^tE(t-s)\int_0^1(1-\lambda)\mcal D^2F\big(X(s)+\lambda(X^m(s)-X(s))\big)(X^m(s)-X(s))^2\ud\lambda\ud s.
\end{align*}
Note that for any $s\in[0,T]$, $X(s),X^m(s)\in\dot{H}^\alpha$ almost surely with $\alpha\in[0,\sigma+\frac12)$ due to \eqref{Xspatial} and Lemma \ref{Xmregularity}(i). 
Then the property \eqref{semigroup1}, the condition \eqref{F''}, and Corollary \ref{cor1} yield that
\begin{align*}
\|A_{1,1}^m(t)\|_{\mbf L^2(\Omega;\dot{H}^\eta)}\le Km^{\frac{1}{2}}\int_0^t(t-s)^{-\frac{\eta}{2}}\|X^m(s)-X(s)\|_{\mbf L^4(\Omega;\dot{H}^\alpha)}^2\ud s\le Km^{-\frac{1}{2}+\max(\alpha-\sigma,0)}.
\end{align*}

For the term $A_2^m$, the linear growth property of $F$, together with properties \eqref{semigroup1}--\eqref{semigroup2} and Lemma \ref{Xmregularity}(i), yields for any $\gamma\in(0,\frac{1-\eta}{2})$ that
\begin{align*}
	\|A_2^m(t)\|_{\mbf L^2(\Omega;\dot H^\eta)}\le&\; Km^{\frac{1}{2}}\int_0^t\|(-A)^{\frac{\eta+1}{2}+\gamma}E(t-s)\|_{\mcal L(H)}\|(-A)^{-\frac{1}{2}-\gamma}\big(E(s-\kappa_m(s))-Id_H\big)\|_{\mcal L(H)}\\
	&\qquad\qquad\times\left(1+\|X^m(\kappa_m(s))\|_{\mbf L^2(\Omega;H)}\right)\ud s\\
	\le&\; K(\gamma)m^{-\gamma}\int_0^t(t-s)^{-\frac{\eta+1}{2}-\gamma}\ud s\le K(\gamma)m^{-\gamma}.
\end{align*}
Choosing $\gamma=\frac{1-\eta}2-\epsilon$ for any sufficiently small $\epsilon>0$, we then get
\[
\|A_2^m(t)\|_{\mbf L^2(\Omega;\dot H^\eta)}\le K(\epsilon)m^{-\frac{1-\eta}2+\epsilon}.
\]

For the term $A_3^m$, by noting that
\begin{align}
	X^m(s)-X^m(\kappa_m(s))=&\;\big(E(s-\kappa_m(s))-Id_H\big)X^m(\kappa_m(s))+\int_{\kappa_m(s)}^sE(s-\kappa_m(r))F(X^m(\kappa_m(r)))\ud r \notag\\
	&\;+\int_{\kappa_m(s)}^sE(s-\kappa_m(r))G(X^m(\kappa_m(r)))\ud W(r),\label{sec3eq2}
\end{align}
it can be further split as
$A_3^m(t)=\sum_{i=1}^4A_{3,i}^m(t)$
with
\begin{align*}
	A_{3,1}^m(t)&:=-m^{\frac{1}{2}}\int_0^tE(t-s)\mcal DF(X^m(\kappa_m(s)))\big(E(s-\kappa_m(s))-Id_H\big)X^m(\kappa_m(s))\ud s,\\
	A_{3,2}^m(t)&:=-m^{\frac{1}{2}}\int_0^tE(t-s)\mcal DF(X^m(\kappa_m(s)))\int_{\kappa_m(s)}^sE(s-\kappa_m(r))F(X^m(\kappa_m(r)))\ud r\ud s,\\
	A_{3,3}^m(t)&:=-m^{\frac{1}{2}}\int_0^tE(t-s)\mcal DF(X^m(\kappa_m(s)))\int_{\kappa_m(s)}^sE(s-\kappa_m(r))G(X^m(\kappa_m(r)))\ud W(r)\ud s,\\
	A_{3,4}^m(t)&:=-m^{\frac{1}{2}}\int_0^tE(t-s)\int_0^1(1-\lambda)\\
	&\qquad\qquad\cdot\mcal D^2F\big(X^m(\kappa_m(s))+\lambda(X^m(s)-X^m(\kappa_m(s)))\big)\big(X^m(s)-X^m(\kappa_m(s))\big)^2\ud\lambda\ud s.
\end{align*}
By properties \eqref{semigroup1}--\eqref{semigroup2}, the condition \eqref{F'}, and Lemma \ref{Xmregularity}(i), we get
\begin{align*}
	\|A^m_{3,1}(t)\|_{\mbf L^2(\Omega;\dot{H}^\eta)}\le&\;  Km^{\frac{1}{2}}\int_0^t\|(-A)^{\frac{\eta}{2}}E(t-s)\|_{\mcal L(H)}\|(-A)^{-\frac{1+\sigma}{2}}\big(E(s-\kappa_m(s))-Id_H\big)\|_{\mcal L(H)}\\
	&\qquad\qquad\times\|X^m(\kappa_m(s))\|_{\mbf L^2(\Omega;\dot{H}^{1+\sigma})}\ud s \\
	\le&\;  Km^{-\frac{\sigma}{2}},
\end{align*}
and similarly
\begin{align*}
	\|A^m_{3,2}(t)\|_{\mbf L^2(\Omega;\dot{H}^\eta)}\le& Km^{-\frac{1}{2}}\int_0^t\|(-A)^{\frac{\eta}{2}}E(t-s)\|_{\mcal L(H)}(1+\|X^m(\kappa_m(s))\|_{\mbf L^2(\Omega;H)})\ud s\le Km^{-\frac{1}{2}}.
\end{align*}
Applying the stochastic Fubini theorem, we rewrite $A_{3,3}^m$ as
\begin{align*}
	A_{3,3}^m(t)=-m^{\frac{1}{2}}\int_0^t\int_r^{(\kappa_m(r)+\frac{T}{m})\wedge t}E(t-s)\mcal DF(X^m(\kappa_m(s)))E(s-\kappa_m(r))G(X^m(\kappa_m(r)))\ud s\ud W(r).
\end{align*}
Then combining the It\^o isometry, the property \eqref{semigroup1}, the condition \eqref{F'}, the linear growth property of $G$, and Lemma \ref{Xmregularity}(i), one has
\begin{align*}
	&\|A^m_{3,3}(t)\|^2_{\mbf L^2(\Omega;\dot{H}^\eta)}\\
	=&\;m\mbf E\int_0^t\Big\|\int_r^{(\kappa_m(r)+\frac{T}{m})\wedge t}(-A)^{\frac{\eta}{2}}E(t-s)\mcal DF(X^m(\kappa_m(s)))E(s-\kappa_m(r))G(X^m(\kappa_m(r)))\ud s\Big\|_{\mcal L_2^0}^2\ud r \\
	\le &\; K \int_0^t\int_r^{(\kappa_m(r)+\frac{T}{m})\wedge t}(t-s)^{-\eta}\ud s\ud r 
	= K\int_0^t\int_{\kappa_m(s)}^s(t-s)^{-\eta}\ud r\ud s\\
	\le &\; Km^{-1}\int_0^t(t-s)^{-\eta}\ud s \le Km^{-1}.
\end{align*}
Similar to the estimate of $A_{1,1}^m$, the property \eqref{semigroup1}, the condition \eqref{F''}, and Lemma \ref{Xmregularity}(ii) yield that
\begin{align*}
	\|A_{3,4}^m(t)\|_{\mbf L^2(\Omega;\dot{H}^\eta)}\le Km^{\frac{1}{2}}\int_0^t(t-s)^{-\frac{\eta}{2}}\|X^m(s)-X^m(\kappa_m(s))\|^2_{\mbf L^4(\Omega;\dot{H}^\alpha)}\ud s\le Km^{-\frac{1}{2}+\max(\alpha-\sigma,0)}.
\end{align*}
It follows from the previous estimates for $A_{3,i}^m(t)$ with $i=1,\ldots,4$ that
\begin{align*}
	\|A_3^m(t)\|_{\mbf L^2(\Omega;\dot{H}^\eta)}\le Km^{-\min\big(\frac{\sigma}{2},\frac{1}{2}-\max(\alpha-\sigma,0)\big)}.
\end{align*}

Then the previous estimates for $A_{i}^m$ with $i=1,2,3$ lead to
\begin{align}
	I^m(t)=\int_0^tE(t-s)\mcal DF(X(s))U^m(s)\ud s+R^m_1(t),\quad t\in[0,T],\label{sec3eq2'}
\end{align}
where $R_1^m(t):=A_{1,1}^m(t)+A_2^m(t)+A_3^m(t)$ satisfies for any $\epsilon\ll 1$ that 
\[
\sup\limits_{t\in[0,T]}\|R_1^m(t)\|_{\mbf L^2(\Omega;\dot{H}^\eta)}\le K(\epsilon)m^{-\min\big(\frac{\sigma}{2},\frac{1}{2}-\max(\alpha-\sigma,0),\frac{1-\eta}{2}-\epsilon\big)}.
\]

{\it Step 2.} The estimate for $II^m$ is similar to that of $I^m$ utilizing in addition the It\^o isometry. We next decompose $II^m$ as
\begin{align*}
	II^m(t)=&\;m^{\frac{1}{2}}\int_0^tE(t-s)\big(G(X^m(s))-G(X(s))\big)\ud W(s)\\
	&+m^{\frac{1}{2}}\int_0^tE(t-s)\big(E(s-\kappa_m(s))-Id_H\big)G(X^m(\kappa_m(s)))\ud W(s)\\
	&-m^{\frac{1}{2}}\int_0^tE(t-s)\big(G(X^m(s))-G(X^m(\kappa_m(s)))\big)\ud W(s)\\
	=:&\;A_4^m(t)+A_5^m(t)+A_6^m(t).
\end{align*}

The term $A_4^m$ can also be further expanded as
\begin{align*}
	A_4^m(t)=\int_0^tE(t-s)\mcal DG(X(s))U^m(s)\ud W(s)+A_{4,1}^m(t)
\end{align*}
with
\begin{align*}
	A_{4,1}^m(t):=m^{\frac{1}{2}}\int_0^tE(t-s)\int_0^1(1-\lambda)\mcal D^2G\big(X(s)+\lambda(X^m(s)-X(s))\big)\big(X^m(s)-X(s)\big)^2\ud\lambda\ud W(s).
\end{align*}
The It\^o isometry, together with the property \eqref{semigroup1}, the condition \eqref{G''}, and Corollary \ref{cor1}, yields that
\begin{align*}
&\;\mbf E\|A_{4,1}^m(t)\|^2_\eta\\
=&\;m\mbf E\int_0^t\big\|(-A)^{\frac{\eta}{2}}E(t-s)\int_0^1(1-\lambda)\mcal D^2G(X(s)+\lambda(X^m(s)-X(s)))(X^m(s)-X(s))^2\ud\lambda\big\|_{\mcal L_2^0}^2\ud s\\
\le &\;Km\int_0^t(t-s)^{-\eta}\mbf E\|X^m(s)-X(s)\|_\alpha^4\ud s 
\le  Km^{-1+2\max(\alpha-\sigma,0)}.
\end{align*}

For the term $A_5^m$, we deduce for any $\gamma\in(0,\sigma-\eta)$ that
\begin{align*}
\mbf E\|A_5^m(t)\|^2_\eta
=&\;m\mbf E\int_0^t\|(-A)^{\frac{\eta}{2}}E(t-s)\big(E(s-\kappa_m(s))-Id_H\big)G(X^m(\kappa_m(s)))\|_{\mcal L_2^0}^2\ud s\\
\le &\; m\mbf E\int_0^t\|(-A)^{\frac{1+\eta+\gamma-\sigma}{2}}E(t-s)\|^2_{\mcal L(H)}\|(-A)^{-\frac{1+\gamma}{2}}\big(E(s-\kappa_m(s))-Id_H\big)\|_{\mcal L(H)}^2\\
&\qquad\quad\times\|(-A)^{\frac{\sigma}{2}}G(X^m(\kappa_m(s)))\|_{\mcal L_2^0}^2\ud s\\
\le &\; K(\gamma)m^{-\gamma}\int_0^t(t-s)^{-(1+\eta+\gamma-\sigma)}(1+\mbf E\|X^m(\kappa_m(s))\|^2_\sigma)\ud s
\le K(\gamma)m^{-\gamma}
\end{align*}
based on properties \eqref{semigroup1}--\eqref{semigroup2}, the condition \eqref{Ggrow}, and Lemma \ref{Xmregularity}(i). Choosing $\gamma=\sigma-\eta-2\epsilon$ for any $\epsilon\ll1$ leads to
\[
\mbf E\|A_5^m(t)\|^2_\eta\le K(\epsilon)m^{-(\sigma-\eta-2\epsilon)}.
\]

For the term $A_6^m$, we split it as $A^m_6(t)=\sum_{i=1}^4A_{6,i}^m(t)$ based on \eqref{sec3eq2} with
\begin{align*}
	A_{6,1}^m(t)&:=-m^{\frac{1}{2}}\int_0^tE(t-s)\mcal DG(X^m(\kappa_m(s)))\big(E(s-\kappa_m(s))-Id_H\big)X^m(\kappa_m(s))\ud W(s),\\
	A_{6,2}^m(t)&:=-m^{\frac{1}{2}}\int_0^tE(t-s)\mcal DG(X^m(\kappa_m(s)))\int_{\kappa_m(s)}^sE(s-\kappa_m(r))F(X^m(\kappa_m(r)))\ud r\ud W(s),\\
	A_{6,3}^m(t)&:=-m^{\frac{1}{2}}\int_0^tE(t-s)\mcal DG(X^m(\kappa_m(s)))O^m(s)\ud W(s),\\
	A_{6,4}^m(t)&:=-m^{\frac{1}{2}}\int_0^tE(t-s)\int_0^1(1-\lambda)\\&\qquad\qquad
	\cdot\mcal D^2G\big(X^m(\kappa_m(s))+\lambda(X^m(s)-X^m(\kappa_m(s)))\big)\big(X^m(s)-X^m(\kappa_m(s))\big)^2\ud\lambda\ud W(s).
\end{align*}
It follows from the It\^o isometry, properties \eqref{semigroup1}--\eqref{semigroup2}, the condition \eqref{G'}, and Lemma \ref{Xmregularity}(i) that
\begin{align*}
	\mbf E\|A_{6,1}^m(t)\|_\eta^2=&\;m\mbf E\int_0^t\|(-A)^{\frac{\eta}{2}}E(t-s)\mcal DG(X^m(\kappa_m(s)))\big(E(s-\kappa_m(s))-Id_H\big)X^m(\kappa_m(s))\|_{\mcal L_2^0}^2\ud s\\
	\le &\; Km\int_0^t(t-s)^{-\eta}\|(-A)^{-\frac{\sigma+1}{2}}\big(E(s-\kappa_m(s))-Id_H\big)\|_{\mcal L(H)}^2\mbf E\|X^m(\kappa_m(s))\|_{1+\sigma}^2\ud s \\
	\le &\; Km^{-\sigma}.
\end{align*}
Terms $A_{6,2}^m$ and $A_{6,4}^m$ can be estimated similarly based on Assumption \ref{assum2}, Lemma \ref{Xmregularity}, and the linear growth property of $F$:
\begin{align*}
	\mbf E\|A_{6,2}^m(t)\|_\eta^2=\;&m\mbf E\int_0^t\Big\|(-A)^{\frac{\eta}{2}}E(t-s)\mcal DG(X^m(\kappa_m(s)))\int_{\kappa_m(s)}^sE(s-\kappa_m(r))F(X^m(\kappa_m(r)))\ud r\Big\|_{\mcal L_2^0}^2\ud s \\
	\le &\; K\int_0^t(t-s)^{-\eta}\int_{\kappa_m(s)}^s(1+\mbf E\|X^m(\kappa_m(r))\|^2)\ud r\ud s
	\le Km^{-1},\\
	\mbf E\|A_{6,4}^m(t)\|_\eta^2\le&\; Km\int_0^t(t-s)^{-\eta}\mbf E\|X^m(s)-X^m(\kappa_m(s))\|_\alpha^4\ud s\le Km^{-1+2\max(\alpha-\sigma,0)}.
\end{align*}

Combining the previous estimates for $A_i^m$ with $i=4,5,6$, we have
\begin{align}
	II^m(t)=&\;\int_0^t E(t-s)\mcal DG(X(s))U^m(s)\ud W(s)\notag\\
	&\;-m^{\frac{1}{2}}\int_0^tE(t-s)\mcal DG(X^m(\kappa_m(s)))O^m(s)\ud W(s)+R_2^m(t),\quad t\in[0,T],\label{sec3eq3}
\end{align}
where $R_2^m(t):=A_{4,1}^m(t)+A_5^m(t)+A_{6,1}^m(t)+A_{6,2}^m(t)+A_{6,4}^m(t)$ satisfies for any $\epsilon\ll 1$ that 
\[
\sup\limits_{t\in[0,T]}\|R_2^m(t)\|_{\mbf L^2(\Omega;\dot{H}^\eta)}\le K(\epsilon)m^{-\min\big(\frac{1}{2}-\max(\alpha-\sigma,0),\frac{\sigma-\eta}{2}-\epsilon\big)}.
\]

Finally, the proof is finished as a result of \eqref{sec3eq1}, \eqref{sec3eq2'}, and \eqref{sec3eq3}.
\end{proof}

According to Lemma \ref{errordecomposition}, one can define the auxiliary process $\widetilde U^m$ by eliminating the residual term. 

\begin{lem}\label{auxiliary}
Let $\sigma\in(0,1)$ be given such that Assumptions \ref{assum1}--\ref{assum2} are fulfilled. Then for any $\eta\in[0,\sigma)$, $\sup\limits_{t\in[0,T]}\mbf E\|\widetilde{U}^m(t)\|^2_\eta\le K(T)$ and it holds
\begin{align*}
	\lim_{m\to\infty}\sup_{t\in[0,T]}\mbf E\|U^m(t)-\widetilde{U}^m(t)\|_\eta^2=0,
\end{align*}
where $\widetilde{U}^m(t)$ solves the following equation
\begin{align*}
	\widetilde{U}^m(t)=&\;\int_0^tE(t-s)\mcal DF(X(s))\widetilde{U}^m(s)\ud s+\int_0^tE(t-s)\mcal DG(X(s))\widetilde{U}^m(s)\ud W(s) \notag\\
	&\;-m^{\frac{1}{2}}\int_0^tE(t-s)\mcal DG(X^m(\kappa_m(s)))O^m(s)\ud W(s),~t\in[0,T]
\end{align*}
with $O^m(s)$ defined in Lemma \ref{errordecomposition}.
\end{lem}

\begin{proof}
Based on properties \eqref{semigroup1}--\eqref{semigroup2}, \eqref{Xspatial}, Lemma \ref{Xmregularity}, it can be shown that $\sup\limits_{t\in[0,T]}\mbf E\|\widetilde{U}^m(t)\|^2_\eta\le K(T)$. Then subtracting $\widetilde{U}^m$ from $U^m$, and using the It\^o isometry, the property \eqref{semigroup1}, Assumption \ref{assum2}, and Lemma \ref{errordecomposition}, we derive for $\epsilon\ll 1$ that 
\begin{align*}
	\mbf E\|U^m(t)-\widetilde{U}^m(t)\|_\eta^2
	\le&\; K\mbf E\int_0^t\|(-A)^{\frac{\eta}{2}}E(t-s)\mcal DF(X(s))(U^m(s)-\widetilde{U}^m(s))\|^2\ud s \\
	 &\;+K\mbf E\int_0^t\|(-A)^{\frac{\eta}{2}}E(t-s)\mcal DG(X(s))(U^m(s)-\widetilde{U}^m(s))\|_{\mcal L_2^0}^2\ud s\\
	 &\;+K(\eta,\epsilon)m^{-\min\big(1-2\max(\alpha-\sigma,0),\,\sigma-\eta-2\epsilon\big)} \\
	 \le &\; K\int_0^t(t-s)^{-\eta}\mbf E\|U^m(s)-\widetilde{U}^m(s)\|_\eta^2\ud s+K(\eta,\epsilon)m^{-\min\big(1-2\max(\alpha-\sigma,0),\,\sigma-\eta-2\epsilon\big)},
\end{align*}
where we used the fact $\|U^m(s)-\widetilde{U}^m(s)\|\le\lambda_1^{-\frac{\eta}2}\|U^m(s)-\widetilde{U}^m(s)\|_\eta$ in the last step according to \eqref{norm}. Then the proof is complete based on the Gronwall inequality. 
\end{proof}

\subsection{Finite-dimensional approximation $\widetilde U^{m,n}$ of $\widetilde{U}^m$}\label{Sec3.2}
In this part, we construct a finite-dimensional process $\widetilde U^{m,n}$ which uniformly approximates $\widetilde U^m(t)$ in the sense of Condition (A1) of Theorem \ref{uniform approximation}.

For $n\in\mbb N^+$, denote by $H_n:=\text{span}\{e_1,e_2,\ldots,e_n\}$ the $n$-dimensional subspace of $H$, and by $P_n:H\to H_n$ the projection operator defined by
$P_nv=\sum_{i=1}^n\LL v,e_i\RR e_i$ for any $v\in H$.  Define $A_n\in\mcal L(H_n)$ by $A_n:=AP_n$. Then $A_n$ generates a $C_0$-semigroup $\{E_n(t)=e^{tA_n}\}_{t\ge 0}$ on $H_n$. Further, 
define the operator $Q_n\in\mcal L(U)$ by $Q_nu=\sum_{k=1}^n\LL u,h_k\RR_U Qh_k$ and the truncated process
$W^n:=\sum_{k=1}^nQ^{\frac{1}{2}}h_k\beta_k$. It is easy to see that $Q_n$ is a symmetric and positive definite operator on $U$ with finite trace, and
$W^{n}$ is a $U$-valued $Q_n$-Wiener process.

Next, we define the finite-dimensional approximation process $\widetilde{U}^{m,n}$, which solves the following equation
\begin{align}\label{Utildemn}
	\widetilde U^{m,n}(t)=&\;\int_0^tE_n(t-s)P_n\mcal DF(X(s))\widetilde{U}^{m,n}(s)\ud s+\int_0^tE_n(t-s)P_n\mcal DG(X(s))\widetilde{U}^{m,n}(s)\ud W^n(s)\notag\\
	&\;-m^{\frac{1}{2}}\int_0^tE_n(t-s)P_n\mcal DG(X^m(\kappa_m(s)))O^{m,n}(s)\ud W^n(s),~t\in[0,T],
\end{align}
where $O^{m,n}(s):=\int_{\kappa_m(s)}^sE_n(s-\kappa_m(r))P_nG(X^m(\kappa_m(r)))\ud W^n(r)$.

The following fact shows that  Condition (A1) of Theorem \ref{uniform approximation} are fulfilled by $\widetilde{U}^{m,n}(t)$ and $\widetilde{U}^m(t)$.
\begin{lem}\label{Utildemn-Utildem}
	Let Assumptions \ref{assum1}--\ref{assum3} hold with $\beta_1\in(0,1)$. Then for any $\eta\in[0,1-\beta_1)$, it holds
	\begin{align*}
		\lim_{n\to\infty}\sup_{t\in[0,T]}\sup_{m\ge 1}\mbf E\|\widetilde{U}^{m,n}(t)-\widetilde{U}^m(t)\|_\eta^2=0. 
	\end{align*}
\end{lem}

\begin{proof}
Based on properties \eqref{semigroup1}--\eqref{semigroup2}, \eqref{Xspatial}, and Lemma \ref{Xmregularity},   one can show for any $\gamma\in[0,1)$ that 
\begin{align}
	\sup_{m\ge 1}\sup_{t\in[0,T]}\|\widetilde{U}^m(t)\|_{\mbf L^2(\Omega;\dot{H}^\gamma)}+\sup_{m,n\ge 1}\sup_{t\in[0,T]}\|\widetilde{U}^{m,n}(t)\|_{\mbf L^2(\Omega;\dot{H}^\gamma)}<\infty.\label{sec3eq4}
\end{align}	
	
Consider the decomposition $\widetilde{U}^{m,n}(t)-\widetilde{U}^m(t)=\sum_{i=1}^3S_i^{m,n}(t)$ with
\begin{align*}
	S_1^{m,n}(t)&:=\int_0^t\Big(E_n(t-s)P_n\mcal DF(X(s))\widetilde{U}^{m,n}(s)-E(t-s)\mcal DF(X(s))\widetilde{U}^m(s)\Big)\ud s, \\
	S_2^{m,n}(t)&:=\int_0^tE_n(t-s)P_n\mcal DG(X(s))\widetilde{U}^{m,n}(s)\ud W^n(s)-\int_0^tE(t-s)\mcal DG(X(s))\widetilde{U}^m(s)\ud W(s), \\
	S_3^{m,n}(t)&:=m^{\frac{1}{2}}\int_0^tE(t-s)\mcal DG(X^m(\kappa_m(s)))O^m(s)\ud W(s)\\
	&\qquad-m^{\frac{1}{2}}\int_0^tE_n(t-s)P_n\mcal DG(X^m(\kappa_m(s)))O^{m,n}(s)\ud W^n(s).
\end{align*}

Noting that $E_n(t)P_nu=E(t)P_nu$ for any $u\in H$, and
\begin{align}\label{Pn-Id}
	\|(-A)^{-\gamma}(P_n-Id_H)\|_{\mcal L(H)}= \lambda_{n+1}^{-\gamma},\quad\forall~\gamma\ge 0,
\end{align}
we have
\begin{align*}
	S_1^{m,n}(t)=&\;\int_0^tE(t-s)(P_n-Id_H)\mcal DF(X(s))\widetilde{U}^{m,n}(s)\ud s\\
	&\;+\int_0^tE(t-s)\mcal DF(X(s))(\widetilde{U}^{m,n}(s)-\widetilde{U}^m(s))\ud s \\
	=:&\; S_{1,1}^{m,n}(t)+S_{1,2}^{m,n}(t),
\end{align*}
where
\begin{align*}
	\|S_{1,1}^{m,n}(t)\|_{\mbf L^2(\Omega;\dot{H}^\eta)}\le&\; \int_0^t\|(-A)^{\frac{1+\eta}{2}}E(t-s)\|_{\mcal L(H)}\|(-A)^{-\frac{1}{2}}(P_n-Id_H)\|_{\mcal L(H)}\|\widetilde{U}^{m,n}(s)\|_{\mbf L^2(\Omega;H)}\ud s \\
	\le &\;K\lambda_{n+1}^{-\frac{1}{2}}\int_0^t(t-s)^{-\frac{1+\eta}{2}}\|\widetilde{U}^{m,n}(s)\|_{\mbf L^2(\Omega;H)}\ud s \le K\lambda_{n+1}^{-\frac{1}{2}}
\end{align*}
follows from \eqref{semigroup1},  \eqref{F'}, \eqref{sec3eq4}, and \eqref{Pn-Id}. 
By properties \eqref{norm}--\eqref{semigroup1}, and condition \eqref{F'}, it holds
\begin{align*}
	\|S_{1,2}^{m,n}(t)\|_{\mbf L^2(\Omega;\dot{H}^\eta)}\le K\int_0^t(t-s)^{-\frac{\eta}{2}}\|\widetilde{U}^{m,n}(s)-\widetilde{U}^m(s)\|_{\mbf L^2(\Omega;\dot{H}^\eta)}\ud s.
\end{align*}
Accordingly,
\begin{align}\label{S1mn}
	\sup_{m\ge 1}\mbf E\|S_1^{m,n}(t)\|^2_\eta\le K\int_0^t(t-s)^{-\eta}\sup_{m\ge 1}\mbf E\|\widetilde{U}^{m,n}(s)-\widetilde{U}^m(s)\|^2_\eta\ud s+K\lambda_{n+1}^{-1}.
\end{align}

For $S_2^{m,n}$, we decompose it into
\begin{align*}
	S_2^{m,n}(t)=&\;\int_0^tE(t-s)(P_n-Id_H)\mcal DG(X(s))\widetilde{U}^{m,n}(s)\ud W^n(s)\\
	&\;+\int_0E(t-s)\mcal DG(X(s))(\widetilde{U}^{m,n}(s)-\widetilde{U}^m(s))\ud W^n(s) \\
	&\;-\int_0^tE(t-s)\mcal DG(X(s))\widetilde U^m(s)\ud W^{Q-Q_n}(s) \\
	=:&\;S_{2,1}^{m,n}(t)+S_{2,2}^{m,n}(t)+S_{2,3}^{m,n}(t),
\end{align*}
where $W^{Q-Q_n}:=\sum_{k=n+1}^\infty Q^{\frac{1}{2}}h_k\beta_k$ is a $U$-valued $(Q-Q_n)$-Wiener process. By the It\^o isometry, \eqref{semigroup1}, \eqref{G'}, and \eqref{sec3eq4}, we derive
\begin{align*}
\mbf E\|S_{2,1}^{m,n}(t)\|_\eta^2=&\;\mbf E\int_0^t\|(-A)^{\frac{1+\eta}{4}}E(t-s)(-A)^{-\frac{1-\eta}{4}}(P_n-Id_H)\mcal DG(X(s))\widetilde{U}^{m,n}(s)Q_n^{\frac{1}{2}}\|_{\mcal L_2(U,H)}^2\ud s \\
\le &\; K\lambda_{n+1}^{-\frac{1-\eta}{2}}\int_0^t(t-s)^{-\frac{1+\eta}{2}}\mbf E\|\mcal DG(X(s))\widetilde{U}^{m,n}(s)\|_{\mcal L_2^0}^2\ud s
\le K\lambda_{n+1}^{-\frac{1-\eta}{2}}
\end{align*} 
and
\begin{align*}
	\mbf E\|S_{2,2}^{m,n}(t)\|_\eta^2&\le K\mbf E\int_0^t(t-s)^{-\eta}\|\mcal DG(X(s))(\widetilde{U}^{m,n}(s)-\widetilde{U}^m(s))\|_{\mcal L_2^0}^2\ud s \\
	&\le K\int_0^t(t-s)^{-\eta}\mbf E\|\widetilde{U}^{m,n}(s)-\widetilde{U}^m(s)\|_\eta^2\ud s.
\end{align*}
Denote by $P_{n,U}$ the orthogonal projection operator from $U$ to $\text{span}\{h_1,\ldots,h_n\}$. Then
\begin{align*}
	&\;\mbf E\|S_{2,3}^{m,n}(t)\|_\eta^2\\
	=&\;\mbf E\int_0^t\|(-A)^{\frac{\eta}{2}}E(t-s)\mcal DG(X(s))\widetilde U^m(s)Q^{\frac{1}{2}}(Id_U-P_{n,U})\|_{\mcal L_2(U,H)}^2\ud s \\
	\le &\; \mbf E\int_0^t\|(-A)^{\frac{\eta+\beta_1}{2}}E(t-s)\|_{\mcal L(H)}^2\|(-A)^{-\frac{\beta_1}{2}}\mcal DG(X(s))\widetilde{U}^m(s)Q^{-\frac{\beta_2}{2}}\|_{\mcal L_2^0}^2\|Q^{\frac{\beta_2}{2}}(Id_U-P_{n,U})\|_{\mcal L(U)}^2\ud s \\
	\le &\;K\big(\sup_{k\ge n+1}q_k\big)^{\beta_2}\int_0^t(t-s)^{-(\eta+\beta_1)}\mbf E\|\widetilde{U}^m(s)\|^2\ud s \\
	\le &\; K\big(\sup_{k\ge n+1}q_k\big)^{\beta_2}
\end{align*}
based on the It\^o isometry, \eqref{semigroup1}, \eqref{assum3.1}, and \eqref{sec3eq4}. We then obtain from the above estimates that
\begin{align}\label{S2mn}
	\sup_{m\ge 1}\mbf E\|S_2^{m,n}(t)\|^2_\eta \le K\int_0^t(t-s)^{-\eta}\sup_{m\ge 1}\mbf E\|\widetilde{U}^{m,n}(s)-\widetilde{U}^m(s)\|^2_\eta\ud s+K\lambda_{n+1}^{-\frac{1-\eta}{2}}+K\big(\sup_{k\ge n+1}q_k\big)^{\beta_2}.
\end{align}

Decompose $S_3^{m,n}$ similarly as
\begin{align*}
	S_3^{m,n}(t)=&\;m^{\frac{1}{2}}\int_0^tE(t-s)(Id_H-P_n)\mcal DG(X^m(\kappa_m(s)))O^m(s)\ud W(s) \\
	&\;+m^{\frac{1}{2}}\int_0^tE(t-s)P_n\mcal DG(X^m(\kappa_m(s)))\big(O^m(s)-O^{m,n}(s)\big)\ud W(s) \\
	&\;+m^{\frac{1}{2}}\int_0^tE(t-s)P_n\mcal DG(X^m(\kappa_m(s)))O^{m,n}(s)\ud W^{Q-Q_n}(s) \\
	=:&\;S_{3,1}^{m,n}(t)+S_{3,2}^{m,n}(t)+S_{3,3}^{m,n}(t).
\end{align*}
By noting that
\begin{align*}
	\mbf E\|O^m(s)\|^2=\int_{\kappa_m(s)}^s\mbf E\|E(s-\kappa_m(r))G(X^m(\kappa_m(r)))\|_{\mcal L_2^0}^2\ud r \le Km^{-1},
\end{align*}
one gets from the It\^o isometry, \eqref{semigroup1}, and \eqref{G'} that
\begin{align*}
	&\;\mbf E\|S_{3,1}^{m,n}(t)\|_\eta^2\\
	\le &\; m\mbf E\int_0^t\|(-A)^{\frac{\eta+1}{4}}E(t-s)\|_{\mcal L(H)}^2\|(-A)^{-\frac{1-\eta}{4}}(Id_H-P_n)\|_{\mcal L(H)}^2\|\mcal DG(X^m(\kappa_m(s)))O^m(s)\|_{\mcal L_2^0}^2\ud s \\
	\le &\; Km\lambda_{n+1}^{-\frac{1-\eta}{2}}\int_0^t(t-s)^{-\frac{1+\eta}{2}}\mbf E\|O^m(s)\|^2\ud s\le K\lambda_{n+1}^{-\frac{1-\eta}{2}}.
\end{align*}
The It\^o isometry, \eqref{Ggrow}, Lemma \ref{Xmregularity}(i), and \eqref{assum3.2} lead to
\begin{align}
	&\;\mbf E\|O^m(s)-O^{m,n}(s)\|^2 \notag\\
	\le &\; 2\mbf E\Big\|\int_{\kappa_m(s)}^sE(s-\kappa_m(r))(Id_H-P_n)G(X^m(\kappa_m(r)))\ud W(r)\Big\|^2 \notag\\
     &\;+2\mbf E\Big\|\int_{\kappa_m(s)}^sE(s-\kappa_m(r))P_nG(X^m(\kappa_m(r)))\ud W^{Q-Q_n}(r)\Big\|^2 \notag \\
     \le &\; 2\mbf E\int_{\kappa_m(s)}^s\|(-A)^{-\frac{\sigma}{2}}(Id_H-P_n)\|_{\mcal L(H)}^2\|(-A)^{\frac{\sigma}{2}}G(X^m(\kappa_m(r)))\|_{\mcal L_2^0}^2\ud r \notag\\
     &\;+2\mbf E\int_{\kappa_m(s)}^s\|E(s-\kappa_m(r))P_nG(X^m(\kappa_m(r)))Q^{\frac{1}{2}}(Id_U-P_{n,U})\|_{\mcal L_2(U,H)}^2\ud r \notag\\
     \le &\; K\lambda_{n+1}^{-\sigma}\int_{\kappa_m(s)}^s(1+\mbf E\|X^m(\kappa_m(r))\|_\sigma^2)\ud r \notag\\
     &\;+K\int_{\kappa_m(s)}^s\mbf E\|G(X^m(\kappa_m(r)))Q^{-\frac{\beta_2}{2}}\|_{\mcal L_2^0}^2\|Q^{\frac{\beta_2}{2}}(Id_U-P_{n,U})\|_{\mcal L(U)}^2\ud r \notag \\
     \le &\; Km^{-1}\big(\lambda_{n+1}^{-\sigma}+\big(\sup_{k\ge n+1}q_k\big)^{\beta_2}\big). \label{Om-Omn}
\end{align}
Applying the It\^o isometry, \eqref{semigroup1}, \eqref{G'}, and \eqref{Om-Omn}, one has
\begin{align*}
\mbf E\|S_{3,2}^{m,n}(t)\|_\eta^2=&\;m\mbf E\int_0^t\|(-A)^{\frac{\eta}{2}}E(t-s)P_n\mcal DG(X^m(\kappa_m(s)))\big(O^m(s)-O^{m,n}(s)\big)\|_{\mcal L_2^0}^2\ud s \\
\le&\; Km\int_0^t(t-s)^{-\eta}\mbf E\|O^m(s)-O^{m,n}(s)\|^2\ud s \\
\le&\; K\big(\lambda_{n+1}^{-\sigma}+\big(\sup_{k\ge n+1}q_k\big)^{\beta_2}\big).
\end{align*}
Similar to the estimate of  $\mbf E\|O^m(s)\|^2$, we obtain 
\begin{align}
	\mbf E\|O^{m,n}(s)\|^2\le Km^{-1}, \label{EOmn2}
\end{align}
which, together with the It\^o isometry, \eqref{semigroup1}, and \eqref{assum3.1}, gives
\begin{align*}
&\;\mbf E\|S_{3,3}^{m,n}(t)\|_\eta^2 \\
=&\;m\mbf E\int_0^t\|(-A)^{\frac{\eta+\beta_1}{2}}E(t-s)P_n(-A)^{-\frac{\beta_1}{2}}\mcal DG(X^m(\kappa_m(s)))O^{m,n}(s)Q^{\frac{1}{2}}\big(Id_U-P_{n,U}\big)\|_{\mbf L_2(U,H)}^2\ud s \\
\le&\; Km\mbf E \int_0^t(t-s)^{-(\eta+\beta_1)}\|(-A)^{-\frac{\beta_1}{2}}\mcal DG(X^m(\kappa_m(s)))O^{m,n}(s)Q^{-\frac{\beta_2}{2}}\|_{\mcal L_2^0}^2\|Q^{\frac{\beta_2}{2}}(Id_U-P_{n,U})\|_{\mcal L(U)}^2\ud s \\
\le&\; Km\big(\sup_{k\ge n+1}q_k\big)^{\beta_2}\int_0^t(t-s)^{-(\eta+\beta_1)}\mbf E\|O^{m,n}(s)\|^2\ud s \\
\le&\; K\big(\sup_{k\ge n+1}q_k\big)^{\beta_2}.
\end{align*}
As a result, 
\begin{align}\label{S3mn}
	\sup_{m\ge 1}\mbf E\|S^{m,n}_3(t)\|_\eta^2\le K\lambda_{n+1}^{-\frac{1-\eta}{2}} +K\lambda_{n+1}^{-\sigma}+K\big(\sup_{k\ge n+1}q_k\big)^{\beta_2}.
\end{align}

Combining \eqref{S1mn}, \eqref{S2mn}, and \eqref{S3mn}, we derive
\begin{align*}
	&\;\sup_{m\ge1}\mbf E\|\widetilde{U}^{m,n}(t)-\widetilde{U}^m(t)\|_\eta^2\le K\sum_{i=1}^3\sup_{m\ge1}\mbf E\|S_i^{m,n}(t)\|_\eta^2 \\
 \le 	&\; K\int_0^t(t-s)^{-\eta}\sup_{m\ge1}\mbf E\|\widetilde{U}^{m,n}(s)-\widetilde{U}^m(s)\|_\eta^2 \ud s+K\Big(\lambda_{n+1}^{-\min(\sigma,\frac{1-\eta}{2})}+\big(\sup_{k\ge n+1}q_k\big)^{\beta_2}\Big).
\end{align*}
Then applying the Gronwall inequality yields
\begin{align*}
	\sup_{t\in[0,T]}\sup_{m\ge 1}\mbf E\|\widetilde{U}^{m,n}(t)-\widetilde{U}^m(t)\|_\eta^2 \le K\Big(\lambda_{n+1}^{-\min(\sigma,\frac{1-\eta}{2})}+\big(\sup_{k\ge n+1}q_k\big)^{\beta_2}\Big),
\end{align*} 
which finises the proof due to the fact $\lim\limits_{n\to\infty}q_n=0$. 
\end{proof}

\subsection{Limit distribution of $\widetilde{U}^{m,n}$ as $m\to\infty$}
In this part,  we fix $n\in\mbb N^+$ and investigate the limit distribution of $\widetilde{U}^{m,n}$ in $H_n$ as $m\to\infty$. To this end, we rewrite \eqref{Utildemn} into the strong solution form
\begin{align*}
	\widetilde{U}^{m,n}(t)=\int_0^t\big(A_n\widetilde{U}^{m,n}(s)+P_n\mcal DF(X(s))\widetilde{U}^{m,n}(s)\big)\ud s+\int_0^tP_n\mcal DG(X(s))\widetilde{U}^{m,n}(s)\ud W^n(s) -\widetilde{V}^m(t),
\end{align*}
where
\begin{align}\label{Vtildem}
\widetilde{V}^m(t):=m^{\frac{1}{2}}\int_0^tP_n\mcal DG(X^m(\kappa_m(s)))O^{m,n}(s)\ud W^n(s).
\end{align}  
Here we drop the index $n$ in $\widetilde{V}^m$ for convenience. For a fixed $n$, the convergence in distribution of $\widetilde{U}^{m,n}(t)$ is a result of that of $\widetilde{V}^m$, and the latter is stated in the following lemma. 
\begin{lem}\label{VtildemCon}
	Let Assumptions \ref{assum1} and \ref{assum2} hold. Then for any fixed $n\in\mbb N^+$, $\widetilde{V}^m$ defined by \eqref{Vtildem} stably converges in law to $\widetilde{V}$ in $\mbf C([0,T];H_n)$ as $m\to\infty$, where
	\begin{align*}
		\widetilde{V}(t)=\sqrt{\frac{T}{2}}\sum_{l=1}^n\int_0^tP_n\mcal DG(X(s))\big(P_nG(X(s))Q^{\frac{1}{2}}h_l\big)\ud \widetilde{W}_l^n(s),~t\in[0,T].
	\end{align*}
	Here $\widetilde{W}_l^n(t):=\sum_{k=1}^nQ^{\frac{1}{2}}h_k\tilde\beta_{k,l}(t)$ with $\{\tilde\beta_{k,l}\}_{k,l=1,\ldots,n}$ being a family of independent real-valued standard Brownian motions which are independent of $\{\beta_k\}_{k\ge1}$.
\end{lem}

\begin{proof}
Denote $V^{m,i}(t)=\LL\widetilde{V}^m(t),e_i\RR$ with $i=1,\ldots,n$ and $V^m(t)=(V^{m,1}(t),\ldots,V^{m,n}(t))^\top$ for $t\in[0,T]$. By \eqref{Vtildem}, we have
\begin{align*}
V^{m,i}(t)=m^{\frac{1}{2}}\sum_{k=1}^n\int_0^t\LL P_n\mcal DG(X^m(\kappa_m(s)))O^{m,n}(s)Q^{\frac{1}{2}}h_k,e_i\RR\ud\beta_k(s). 
\end{align*}
Noting that $\mbf C([0,T];H_n)$ is isometric to $\mbf C([0,T];\mbb R^n)$, it suffices to prove the stable convergence in law of $V^m$ in   $\mbf C([0,T];\mbb R^n)$.	

Next, we will apply \cite[Theorem 4-1]{Jacod1997} to give the stable convergence in law of $V^m$.
Denote by $\LL Y_1,Y_2\RR_t$, $t\in[0,T]$ the cross variation process between real-valued semi-martingales $\{Y_1(t)\}_{t\in[0,T]}$ and $\{Y_2(t)\}_{t\in[0,T]}$. 

{\it Step 1. Convergence of $\LL V^{m,i},\beta_j\RR_t$.}
It holds that
\begin{align*}
	\LL V^{m,i},\beta_j\RR_t=m^{\frac{1}{2}}\int_0^t \LL P_n\mcal DG(X^m(\kappa_m(s)))O^{m,n}(s)Q^{\frac{1}{2}}h_j,e_i\RR \ud s=
 m^{\frac{1}{2}}\sum_{l=0}^{\lfloor\frac{t}{\tau}\rfloor}J_l(t),
\end{align*}
where $J_l(t):=\int_{t_l}^{t_{l+1}\wedge t}\LL P_n\mcal DG(X^m(t_l))O^{m,n}(s)Q^{\frac{1}{2}}h_j,e_i\RR\ud s$.
Noting that for any $l_2>l_1$, $$\mbf E(J_{l_1}(t)J_{l_2}(t))=\mbf E\big(J_{l_1}(t)\mbf E(J_{l_2}(t)|\mcal F_{t_{l_2}})\big)=0,$$
we then have $\mbf E(\LL V^{m,i},\beta_j\RR_t)^2=m\sum_{l=0}^{\lfloor\frac{t}{\tau}\rfloor}\mbf E|J_l(t)|^2$. 
It follows from \eqref{G'}, \eqref{EOmn2}, and the fact $\|BQ^{\frac{1}{2}}h_j\|\le \|B\|_{\mcal L_2^0}$ for any $B\in\mcal L_2^0$ that
\begin{align*}
	\mbf E|J_l(t)|^2\le&\; Km^{-1}\int_{t_l}^{t_{l+1}\wedge t}\mbf E\|\mcal DG(X^m(t_l))O^{m,n}(s)\|_{\mcal L_2^0}^2\ud s \\
	\le &\; Km^{-1}\int_{t_l}^{t_{l+1}\wedge t}\mbf E\|O^{m,n}(s)\|^2\ud s\le Km^{-3},
\end{align*}
which leads to
\begin{align}\label{converge1}
	 \mbf E(\LL V^{m,i},\beta_j\RR_t)^2\le Km^{-1}\to 0,\quad\forall~i,j=1,\ldots,n,~t\in[0,T].
\end{align}

{\it Step 2. Convergence of $\LL V^{m,i}, V^{m,j}\RR_t$.}
A direct computation shows
\begin{align*}
	&\;\LL V^{m,i},V^{m,j}\RR_t\\
	=&\;m\sum_{k=1}^n\int_0^t\LL P_n\mcal DG(X^m(\kappa_m(s)))O^{m,n}(s)Q^{\frac{1}{2}}h_k,e_i\RR\LL P_n\mcal DG(X^m(\kappa_m(s)))O^{m,n}(s)Q^{\frac{1}{2}}h_k,e_j\RR\ud s\\
	=&\;m\sum_{k,l_1,l_2=1}^n\int_0^tC^m_{k,i,l_1}(s)C^m_{k,j,l_2}(s)\big(\beta_{l_1}(s)-\beta_{l_1}(\kappa_m(s))\big)\big(\beta_{l_2}(s)-\beta_{l_2}(\kappa_m(s))\big)\ud s,
\end{align*}
where we used the fact $O^{m,n}(s)=\sum_{l=1}^nE_n(s-\kappa_m(s))P_nG(X^m(\kappa_m(s)))Q^{\frac{1}{2}}h_l\big(\beta_l(s)-\beta_l(\kappa_m(s))\big)$ and the notation $C^m_{k,i,l}(s):=\big\LL P_n\mcal DG(X^m(\kappa_m(s)))\big(E_n(s-\kappa_m(s))P_nG(X^m(\kappa_m(s)))Q^{\frac{1}{2}}h_l\big)Q^{\frac{1}{2}}h_k,e_i\big\RR$.
Note also that $\big(\beta_l(s)-\beta_l(\kappa_m(s))\big)^2=2\int_{\kappa_m(s)}^s\big(\beta_l(r)-\beta_l(\kappa_m(s))\big)\ud\beta_l(r)+(s-\kappa_m(s))$. It then follows
\begin{align*}
	\LL V^{m,i},V^{m,j}\RR_t
=&\;m\sum_{k=1}^n\sum_{\substack{l_1\neq l_2 \\l_1,l_2=1,\ldots,n}}\int_0^t	C^m_{k,i,l_1}(s)C^m_{k,j,l_2}(s)\big(\beta_{l_1}(s)-\beta_{l_1}(\kappa_m(s))\big)\big(\beta_{l_2}(s)-\beta_{l_2}(\kappa_m(s))\big)\ud s	\\
&+2m\sum_{k,l=1}^n\int_0^t	C^m_{k,i,l}(s)C^m_{k,j,l}(s)\int_{\kappa_m(s)}^s\big(\beta_l(r)-\beta_l(\kappa_m(s))\big)\ud\beta_l(r)\ud s \\
&+T\sum_{k,l=1}^n\int_0^t	C^m_{k,i,l}(s)C^m_{k,j,l}(s)(\frac{ms}{T}-\lfloor\frac{ms}{T}\rfloor)\ud s \\
=:&\; B^m_1(t)+B^m_2(t)+B^m_3(t).
\end{align*}

Denote $M_{k,i,j,l_1,l_2,p}(t):=\int_{t_p}^{t_{p+1}\wedge t}C^m_{k,i,l_1}(s)C^m_{k,j,l_2}(s)\big(\beta_{l_1}(s)-\beta_{l_1}(\kappa_m(s))\big)\big(\beta_{l_2}(s)-\beta_{l_2}(\kappa_m(s))\big)\ud s$. Then it holds
\begin{align}\label{sec3eq6}
	&\;\mbf E\Big(\int_0^tC^m_{k,i,l_1}(s)C^m_{k,j,l_2}(s)\big(\beta_{l_1}(s)-\beta_{l_1}(\kappa_m(s))\big)\big(\beta_{l_2}(s)-\beta_{l_2}(\kappa_m(s))\big)\ud s\Big)^2\notag \\
	= &\;\sum_{p=0}^{\lfloor\frac{t}{\tau}\rfloor}\mbf E|M_{k,i,j,l_1,l_2,p}(t)|^2+\sum_{p_1\neq p_2}^{\lfloor\frac{t}{\tau}\rfloor}\mbf E\big(M_{k,i,j,l_1,l_2,p_1}(t)M_{k,i,j,l_1,l_2,p_2}(t)\big). 
\end{align}
For $l_1\neq l_2$, 
\begin{align*}
	&\;\mbf E\big(M_{k,i,j,l_1,l_2,p_2}(t)|\mcal F_{t_{p_2}}\big)\\
	=&\;\int_{t_{p_2}}^{t_{p_2+1}\wedge t}C^m_{k,i,l_1}(s)C^m_{k,j,l_2}(s)\mbf E\big(\beta_{l_1}(s)-\beta_{l_1}(\kappa_m(s))\big)\mbf E\big(\beta_{l_2}(s)-\beta_{l_2}(\kappa_m(s))\big)\ud s=0.
\end{align*}
Accordingly, for $l_1\neq l_2$ and $p_2>p_1$, one has
\begin{align}\label{sec3eq7}
\mbf E\big(M_{k,i,j,l_1,l_2,p_1}(t)M_{k,i,j,l_1,l_2,p_2}(t)\big)=\mbf E\big(M_{k,i,j,l_1,l_2,p_1}(t)\mbf E\big(M_{k,i,j,l_1,l_2,p_2}(t)|\mcal F_{t_{p_2}}\big)\big)=0.
\end{align}
Applying \eqref{G'}, $\|BQ^{\frac{1}{2}}h_j\|\le \|B\|_{\mcal L_2^0}$ for $B\in\mcal L_2^0$, and Lemma \ref{Xmregularity}(i) yields
\begin{align*}
	\mbf E|C^m_{k,i,j}(s)|^4\le&\; \mbf E\|\mcal DG(X^m(\kappa_m(s)))\big(E_n(s-\kappa_m(s))P_nG(X^m(\kappa_m(s)))Q^{\frac{1}{2}}h_l\big)\|_{\mcal L_2^0}^4 \\
	\le&\; K\mbf E\|G(X^m(\kappa_m(s)))\|_{\mbf L_2^0}^4\le K(1+\mbf E\|X^m(\kappa_m(s))\|^4)\le K.
\end{align*}
Then we deduce that 
\begin{align}
	\mbf E|M_{k,i,j,l_1,l_2,p}(t)|^2 
	\le&\; Km^{-1} \int_{t_p}^{t_{p+1}\wedge t}(\mbf E|C^m_{k,i,l_1}(s)|^4)^{\frac{1}{2}}(\mbf E|C^m_{k,j,l_2}(s)|^4)^{\frac{1}{2}}\notag\\
	&\qquad\qquad\times\mbf E\big[\big(\beta_{l_1}(s)-\beta_{l_1}(\kappa_m(s))\big)^2\big(\beta_{l_2}(s)-\beta_{l_2}(\kappa_m(s))\big)^2\big]\ud s\notag\\
	\le &\; Km^{-4}. \label{sec3eq8}
\end{align}
Plugging \eqref{sec3eq7} and \eqref{sec3eq8} into \eqref{sec3eq6} produces
\begin{align*}
	\mbf E\Big(\int_0^tC^m_{k,i,l_1}(s)C^m_{k,j,l_2}(s)\big(\beta_{l_1}(s)-\beta_{l_1}(\kappa_m(s))\big)\big(\beta_{l_2}(s)-\beta_{l_2}(\kappa_m(s))\big)\ud s\Big)^2\le Km^{-3}.
\end{align*}
Therefore, we have
\begin{align*}
	\|B_1^m(t)\|_{\mbf L^2(\Omega;\mbb R)}\le Km\sum_{k=1}^n\sum_{l_1\neq l_2}^{n}m^{-\frac{3}{2}}\le K(n)m^{-\frac{1}{2}}\to 0.
\end{align*}
Similarly, one can prove $\|B_2^m(t)\|_{\mbf L^2(\Omega;\mbb R)}\le K(n)m^{-\frac{1}{2}}\to 0.$

For the convergence of $B_3^m$, denote 
$C_{k,i,l}(s):=\LL P_n\mcal DG(X(s))\big(P_nG(X(s)Q^{\frac{1}{2}}h_l)\big)Q^{\frac{1}{2}}h_k,e_i\RR$.
We claim for any $k,i,l=1,\ldots,n$ and $s\in[0,T]$ that
\begin{align}
	\|C^m_{k,i,l}(s)-C_{k,i,j}(s)\|_{\mbf L^2(\Omega;\mbb R)}\le K(n)m^{-\frac{1}{2}(1-\max(\alpha-\sigma,0))}.\label{claim}
\end{align}
In fact, 
\begin{align*}
	&\;C^m_{k,i,l}(s)-C_{k,i,l}(s)\\
	=&\;\big\LL P_n\big(\mcal DG(X^m(\kappa_m(s)))-\mcal DG(X(s))\big)\big(E_n(s-\kappa_m(s))P_nG(X^m(\kappa_m(s)))Q^{\frac{1}{2}}h_l\big)Q^{\frac{1}{2}}h_k,e_i\big\RR \\
	&\;+\big\LL P_n\mcal DG(X(s))\big((E_n(s-\kappa_m(s))-Id_{H_n})P_nG(X^m(\kappa_m(s)))Q^{\frac{1}{2}}h_l\big)Q^{\frac{1}{2}}h_k,e_i\big\RR \\
	&\;+\big\LL P_n\mcal DG(X(s))\big(P_n(G(X^m(\kappa_m(s)))-G(X(s)))Q^{\frac{1}{2}}h_l\big)Q^{\frac{1}{2}}h_k,e_i\big\RR\\
	=:&\; T_1^m(s)+T_2^m(s)+T_3^m(s).
\end{align*}
By Lemma \ref{Xmregularity}(ii) and Corollary \ref{cor1}, for any $\gamma\in[0,1+\sigma)$,
\begin{align}\label{sec3eq9}
	\sup_{s\in[0,T]}\|X^m(\kappa_m(s))-X(s)\|_{\mbf L^p(\Omega;\dot{H}^\gamma)}\le Km^{-\frac{1}{2}\big(1-\max(\gamma-\sigma,0)\big)}.
\end{align}
Applying the Taylor theorem for $\mcal DG$, and using $\|BQ^{\frac{1}{2}}h_k\|\le \|B\|_{\mcal L_2^0}$ for $B\in\mcal L_2^0$, $\|u\|_\alpha\le\lambda_n^{\frac{\alpha}{2}}\|u\|$ for $u\in H_n$, and \eqref{G''}, we have
\begin{align*}
	|T_1^m(s)|
	\le&\; \Big\|\int_0^1\mcal D^2G(X(s)+\lambda(X^m(\kappa_m(s))-X(s)))\\
	&\quad\big(X^m(\kappa_m(s))-X(s),E_n(s-\kappa_m(s))P_nG(X^m(\kappa_m(s)))Q^{\frac{1}{2}}h_l\big)\ud s\Big\|_{\mcal L_2^0} \\
	\le &\; K(n)\|X^m(\kappa_m(s))-X(s)\|_\alpha\|G(X^m(\kappa_m(s)))\|_{\mcal L_2^0} \\
	\le &\;K(n)\|X^m(\kappa_m(s))-X(s)\|_\alpha(1+\|X^m(\kappa_m(s))\|). 
\end{align*}
Then we deduce from the H\"older inequality, Lemma \ref{Xmregularity}(i), and \eqref{sec3eq9} that
\begin{align*}
	\|T_1^m(s)\|_{\mbf L^2(\Omega;\mbb R)}\le K(n)\|X^m(\kappa_m(s))-X(s)\|_{\mbf L^4(\Omega;\dot{H}^\alpha)}(1+\|X^m(\kappa_m(s))\|_{\mbf L^4(\Omega;H)})\le K(n)m^{-\frac{1}{2}}.
\end{align*}
Due to the condition \eqref{G'}, the facts $\|BQ^{\frac{1}{2}}h_k\|\le \|B\|_{\mcal L_2^0}$ for $B\in\mcal L_2^0$,  $\|E_n(t)-Id_{H_n}\|_{\mcal L(H_n)}\le K(n)t$ for $t\ge 0$, the linear growth property of $G$, and Lemma \ref{Xmregularity}(i), it follows that
\begin{align*}
	\|T_2^m(s)\|_{\mbf L^2(\Omega;\mbb R)}&\le K\|E_n(s-\kappa_m(s))-Id_{H_n}\|_{\mcal L(H_n)}\|P_nG(X^m(\kappa_m(s)))Q^{\frac{1}{2}}h_l\|_{\mbf L^2(\Omega;H)} \\
	&\le K(n)m^{-1}(1+\|X^m(\kappa_m(s))\|_{\mbf L^2(\Omega;H)})\le K(n)m^{-1}.
\end{align*}
Further, combining $\|BQ^{\frac{1}{2}}h_k\|\le \|B\|_{\mcal L_2^0}$ for $B\in\mcal L_2^0$, \eqref{GLip}, \eqref{G'}, and \eqref{sec3eq9} gives
\begin{align*}
	\|T_3^m(s)\|_{\mbf L^2(\Omega;\mbb R)}\le K\|G(X^m(\kappa_m(s)))-G(X(s))\|_{\mbf L^2(\Omega;\mcal L_2^0)}\le Km^{-\frac{1}{2}\big(1-\max(\alpha-\sigma,0)\big)}.
\end{align*}
According to the previous estimates for $T_i^m$, $i=1,2,3$, we prove the claim \eqref{claim}. Based on \eqref{claim}, 
\begin{align*}
	\lim_{m\to\infty}\mbf E\int_0^t|C^m_{k,i,l}(s)C^m_{k,j,l}(s)-C_{k,i,l}(s)C_{k,j,l}(s)|\ud s=0,\quad\forall~t\in[0,T].
\end{align*}
Then we can apply \cite[Proposition 4.2]{HJWY24} to conclude $\lim\limits_{m\to\infty}B_3^m(t)=\frac{T}{2}\sum_{k,l=1}^n\int_0^tC_{k,i,l}(s)C_{k,j,l}(s)\ud s$ in $\mbf L^1(\Omega;\mbb R)$ for any $t\in[0,T]$, which along with $\|B_i^m(t)\|_{\mbf L^2(\Omega;\mbb R)}\le K(n)m^{-\frac{1}{2}}$ for $i=1,2$ yields
\begin{align}\label{converge2}
\lim_{m\to\infty}\LL V^{m,i},V^{m,j}\RR_t=\frac{T}{2}\sum_{k,l=1}^n\int_0^tC_{k,i,l}(s)C_{k,j,l}(s)\ud s~\text{in}~\mbf L^1(\Omega;\mbb R)\quad\forall~i,j=1,\ldots,n,~t\in[0,T].
\end{align}

{\it Step 3. Stable convergence in law of $V^m$ and $\widetilde{V}^m$.}
According to \eqref{converge1} and \eqref{converge2} obtained in former steps, we use \cite[Theorem 4-1]{Jacod1997} to show that $V^m\overset{stably}{\Longrightarrow}V$ in $\mbf C([0,T];\mbb R^n)$, where $V$ is a $(\beta_1,\ldots,\beta_n)$-bias conditional Gaussian martingale on some extension of $(\Omega,\mcal F,\mbf P)$ (still denoted by $(\Omega,\mcal F,\mbf P)$) and satisfies
\begin{align}
	\LL V^i,\beta_j\RR_t=0,~\LL V^i,V^j\RR_t=\frac{T}{2}\sum_{k,l=1}^n\int_0^tC_{k,i,l}(s)C_{k,j,l}(s)\ud s,\quad i,j=1,2\ldots,n,~t\in[0,T]. \label{Vvariation}
\end{align} 
By the martingale representation theorem (cf. \cite[Proposition 1-4]{Jacod1997}), $V^i$ can be represented as
\begin{align*}
	V^i(t)=\sum_{l=1}^n\int_0^tu^{i,l}(s)\ud\beta_l(s)+\sum_{l=1}^p\int_0^tv^{i,l}(s)\ud \tilde{\beta}_l(s),\quad i=1,\ldots,n,~t\in[0,T],
\end{align*}
where $(\tilde{\beta}_1,\ldots,\tilde{\beta}_p)$ is a $p$-dimensional standard Brownian motion for some $p\in\mbb N^+$ and is independent of $(\beta_1,\ldots,\beta_n)$, and $u^{i,l}$ and $v^{i,p}$ are stochastically integrable processes to be determined. First, we have $u^{i,l}=0$ for $i,l=1,\ldots,n$ due to $\LL V^i,\beta_j\RR_t=0$. Further,  in order to give $v^{i,l}$, we take $p=n^2$ and rewrite $v^{i,l}$ and $\{\tilde{\beta}_l\}$, $l=1,\ldots,n^2$, as $v^{i,k,l}$ and $\{\tilde{\beta}_{k,l}\}$, $k,l=1,\ldots,n$, i.e.,
\begin{align*}
	V^i(t)=\sum_{k=1}^n\sum_{l=1}^n\int_0^t v^{i,k,l}(s)\ud \tilde{\beta}_{k,l}(s).
\end{align*}
Then, it follows from \eqref{Vvariation} that
\begin{align*}
	\LL V^i,V^j\RR_t=\sum_{k=1}^n\sum_{l=1}^n\int_0^tv^{i,k,l}(s)v^{j,k,l}(s)\ud s=\frac{T}{2}\sum_{k=1}^n\sum_{l=1}^n\int_0^t C_{k,i,l}(s)C_{k,j,l}(s)\ud s. 
\end{align*}
Thus, we have $v^{i,k,l}(s)=\sqrt{\frac{T}{2}}C_{k,i,l}(s)$,  $k,i,l=1,\ldots,n$, and further obtain
\begin{align*}
	V^i(t)=\sqrt{\frac{T}{2}}\sum_{k=1}^n\sum_{l=1}^n\int_0^t\big\LL P_n\mcal DG(X(s))\big(P_nG(X(s))Q^{\frac{1}{2}}h_l\big)Q^{\frac{1}{2}}h_k,e_i\big\RR\ud \tilde{\beta}_{k,l}(s).
\end{align*}
Define the operator $\Gamma:\mbf C([0,T];\mbb R^n)\to\mbf C([0,T];H_n)$ by
\begin{align*}
	\Gamma(f)(t)=\sum_{i=1}^nf_i(t)e_i,\quad\forall~f=(f_1,\ldots,f_n)^\top\in  \mbf C([0,T];\mbb R^n).
\end{align*}
It is not hard to see that $\Gamma$ is a continuous mapping and $\widetilde{V}^m=\Gamma(V^m)$. Since a continuous mapping can preserve the stable convergence in law of random variables, which can be verified directly by the definition of stable convergence in law,  we have $\widetilde{V}^m\overset{satbly}{\Longrightarrow}\Gamma(V)$ in $\mbf C([0,T];H_n)$ with
\begin{align*}
&\;\Gamma(V)(t)=\sum_{i=1}^nV^i(t)e_i\\
=&\;\sqrt{\frac{T}{2}}\sum_{k=1}^n\sum_{l=1}^n\int_0^t P_n\mcal DG(X(s))\big(P_nG(X(s))Q^{\frac{1}{2}}h_l\big)Q^{\frac{1}{2}}h_k\ud \tilde{\beta}_{k,l}(s),\quad t\in[0,T].	
\end{align*} 
By using $\widetilde{W}^n_l(t)=\sum_{k=1}^nQ^{\frac{1}{2}}h_k\tilde{\beta}_{k,l}(t)$, $t\in[0,T]$, $l=1,\ldots,n$, we have that $\Gamma(V)=\widetilde{V}$ and finally complete the proof.
\end{proof}

Based on Lemma \ref{VtildemCon}, we can establish the convergence in distribution of $\widetilde{U}^{m,n}$ as $m\to\infty$.
 
\begin{lem}\label{Utildemn-mtoinfty}
	Let Assumptions \ref{assum1} and \ref{assum2} hold. Then for any fixed $n\in\mbb N^+$, $\widetilde{U}^{m,n}\overset{d}{\Longrightarrow}\widetilde{U}^{\infty,n}$ in $\mbf C([0,T];H_n)$ as $m\to\infty$, where $\widetilde{U}^{\infty,n}$ satisfies
\begin{align*}
\widetilde{U}^{\infty,n}(t)=&\;\int_0^t\big(A_n\widetilde{U}^{\infty,n}(s)+P_n\mcal DF(X(s))\widetilde{U}^{\infty,n}(s)\big)\ud s+\int_0^tP_n\mcal DG(X(s))\widetilde{U}^{\infty,n}(s)\ud W^n(s) \notag\\
&\; -\sqrt{\frac{T}{2}}\sum_{l=1}^n\int_0^tP_n\mcal DG(X(s))\big(P_nG(X(s))Q^{\frac{1}{2}}h_l\big)\ud \widetilde{W}_l^n(s),\quad t\in[0,T]. 
\end{align*}
\end{lem}

\begin{proof}
This proof follows the same procedure as the one used in \cite[Theorem 2.3]{Jin25}. For convenience, we drop the index $n$ in $\widetilde{U}^{m,n}$ and denote $Z^m:=\widetilde{U}^{m,n}$, i.e.,
\begin{align*}
Z^m(t)=&\;\int_0^t\big(A_nZ^m(s)+P_n\mcal DF(X(s))Z^m(s)\big)\ud s+\int_0^tP_n\mcal DG(X(s))Z^m(s)\ud W^n(s) -\widetilde{V}^m(t).
\end{align*}
Let $Z^{m,M}=\{Z^{m,M}(t),t\in[0,T]\}$ be the solution of 
\begin{align*}
Z^{m,M}(t)=&\;\int_0^t\big(A_nZ^{m,M}(\kappa_M(s))+P_n\mcal DF(X(\kappa_M(s)))Z^{m,M}(\kappa_M(s))\big)\ud s\\
&\;+\int_0^tP_n\mcal DG(X(\kappa_M(s)))Z^{m,M}(\kappa_M(s))\ud W^n(s) -\widetilde{V}^m(t).
\end{align*}

Next, we show that $Z^{m,M}$ and $Z^m$ satisfy Conditions (A1)--(A3) of Theorem \ref{uniform approximation}.

Firstly, it can be shown that 
\begin{align*}
\sup_{t\in[0,T]}\sup_{m\ge 1}\mbf E\|Z^m(t)\|^4<\infty, ~ \sup_{m\ge 1}\mbf E\|Z^m(t)-Z^m(s)\|^2 \le K|t-s|,\quad t,s\in[0,T]. 
\end{align*}
Then based on a standard computation (cf. \cite[Theorem 2.3]{Jin25}), we have
$$\lim\limits_{M\to\infty}\sup\limits_{m\ge 1}\mbf E\|Z^{m,M}-Z^{m}\|^2_{\mbf C([0,T];H_n)}=0,$$
which implies Condition (A1) of  Theorem \ref{uniform approximation}. 

Secondly, following the argument of the proof of \cite[Theorem 2.3]{Jin25}, one can use $\widetilde{V}^m\overset{stably}{\Longrightarrow}\widetilde{V}$ to get for any $M\in\mbb N^+$, $Z^{m,M}\overset{stably}{\Longrightarrow}Z^{\infty,M}$ in $\mbf C([0,T];H_n)$ with $Z^{\infty,M}$ satisfying
\begin{align*}
	Z^{\infty,M}(t)=&\;\int_0^t\big(A_nZ^{\infty,M}(\kappa_M(s))+P_n\mcal DF(X(\kappa_M(s)))Z^{\infty,M}(\kappa_M(s))\big)\ud s\\
	&\;+\int_0^tP_n\mcal DG(X(\kappa_M(s)))Z^{\infty,M}(\kappa_M(s))\ud W^n(s) -\widetilde{V}(t).
\end{align*}
This verifies Condition (A2) of Theorem \ref{uniform approximation}.

Finally, one can prove that for any given $n\in\mbb N^+$, it holds
\begin{align*}
	\lim_{M\to\infty}\mbf E \|Z^{\infty,M}-\widetilde{U}^{\infty,n}\|_{\mbf C([0,T];H_n)}^2=0,
\end{align*}
which implies Condition (A3) of Theorem \ref{uniform approximation} and finishes the proof as a result of Theorem \ref{uniform approximation}.
\end{proof}

\subsection{Convergence of $\widetilde{U}^{\infty,n}(t)$ as $n\to\infty$}
In this part, we present the convergence of $\widetilde{U}^{\infty,n}(t)$ as $n\to\infty$ in $\dot{H}^\eta$ for $\eta\in[0,1-\beta_1)$.  For this purpose, we will need the following properties on stochastic integrals. Let $Q_1$ and $Q_2$ be two nonnegative symmetric operators on $U$ with finite traces. Let $W_i$ be a $U$-valued $Q_i$-Wiener process such that 
$W_i(t)=\sum_{k=1}^{\infty}Q^\frac{1}{2}_ih_k\beta^{(i)}_{k}(t)$ for $i=1,2$, where
$\big\{\beta^{(i)}_k\big\}_{k\in\mbb N^+}$ is a family of independent standard Brownian motions defined on $\big(\Omega,\mcal F,\{\mcal F_t\}_{t\in[0,T]},\mbf P\big)$ and $\big\{\beta^{(1)}_k\big\}_{k\in\mbb N^+}$ is independent of $\big\{\beta^{(2)}_k\big\}_{k\in\mbb N^+}$. 
Denote the sets
\begin{align*}
	\mcal N^2_{W_i}(0,T;\mcal L_2^0):=&\Big\{\Phi:[0,T]\times \Omega\to\mcal L_2(Q_i^\frac{1}{2}(U),H)\,\big|\,\Phi ~\text{is predicable} \\
	&\qquad\qquad\text{and} \mbf~ \mbf E\int_{0}^{T}\big\|\Phi(s)Q_i^{\frac{1}{2}}\big\|^2_{\mcal L_2(U,H)}\ud s<+\infty\Big\},\quad i=1,2.
\end{align*}

\begin{pro}\label{independence}
Let $\Phi_i\in\mcal N^2_{W_i}(0,T;\mcal L_2^0)$ for $i=1,2$. Then for any $s,t\in[0,T]$, it holds
\begin{align*}
	\mbf E\Big\LL\int_0^t\Phi_1(r)\ud W_1(r),\int_0^s\Phi_2(r)\ud W_2(r) \Big\RR=0. 
\end{align*}
\end{pro}

\begin{proof}
Following the argument in \cite[Proposition 5]{CHJS23} by replacing  $U_0$ and $Q^\frac{1}{2}_i(U_0)$ by $H$ and $Q^\frac{1}{2}_i(U)$ with $i=1,2$, respectively, we have that the correlation operator between $\int_0^t\Phi_1(r)\ud W_1(r)$ and $\int_0^s\Phi_2(r)\ud W_2(r)$ is $0$, more precisely, for any $a,b\in H$, it holds
\begin{align}\label{sec3eq10}
	\mbf E\Big[\Big\LL\int_0^t\Phi_1(r)\ud W_1(r),a\Big\RR\Big\LL\int_0^s\Phi_2(r)\ud W_2(r),b \Big\RR\Big]=0.
\end{align}

The conclusion immediately comes as a result of \eqref{sec3eq10} by noting
\begin{align*}
	\mbf E\Big\LL\int_0^t\Phi_1(r)\ud W_1(r),\int_0^s\Phi_2(r)\ud W_2(r) \Big\RR=\mbf E\sum_{k=1}^{\infty}\Big\LL\int_0^t\Phi_1(r)\ud W_1(r),e_k\Big\RR\Big\LL\int_0^s\Phi_2(r)\ud W_2(r),e_k \Big\RR=0,
\end{align*}
where the operations $\mbf E[\cdot]$ and $\sum_{k=1}^\infty$ can be exchanged due to the fact
\begin{align*}
	&\;\mbf E\sum_{k=1}^{\infty}\Big|\Big\LL\int_0^t\Phi_1(r)\ud W_1(r),e_k\Big\RR\Big\LL\int_0^s\Phi_2(r)\ud W_2(r),e_k \Big\RR\Big| \\
	\le &\;\frac{1}{2}\mbf E\sum_{k=1}^{\infty}\Big\LL\int_0^t\Phi_1(r)\ud W_1(r),e_k\Big\RR^2+\frac{1}{2}\mbf E\sum_{k=1}^{\infty} \Big\LL\int_0^s\Phi_2(r)\ud W_2(r),e_k \Big\RR^2 \\
	=&\;\frac{1}{2}\mbf E\Big\|\int_0^t\Phi_1(r)\ud W_1(r)\Big\|^2+\frac{1}{2}\mbf E\Big\|\int_0^s\Phi_2(r)\ud W_2(r)\Big\|^2<\infty
\end{align*} 
based on the It\^o isometry.
\end{proof}

Note that, by variation of constants formula, $\widetilde{U}^{\infty,n}$ solves the following equation
\begin{align}
\widetilde{U}^{\infty,n}(t)=&\,\int_0^tE_n(t-s)P_n\mcal DF(X(s))\widetilde{U}^{\infty,n}(s)\ud s+\int_0^tE_n(t-s)P_n\mcal DG(X(s))\widetilde{U}^{\infty,n}(s)\ud W^n(s) \notag\\
&\,-\sqrt{\frac{T}{2}}\sum_{l=1}^n\int_0^tE_n(t-s)P_n\mcal DG(X(s))\big(P_nG(X(s))Q^{\frac{1}{2}}h_l\big)\ud \widetilde{W}_l^n(s),\quad t\in[0,T]. \label{Utildeinfty-n}
\end{align} 
In order to identify the limit of $\widetilde{U}^{\infty,n}$, one needs to consider the limit of $\widetilde{W}_l^n$. We hence extend the family $\{\tilde{\beta}_{k,l}\}_{k,l=1}^n$ of standard Brownian motions to $\{\tilde{\beta}_{k,l}\}_{k,l=1}^\infty$, which is also a family of independent standard Brownian motions and is independent of $\{\beta_k\}_{k=1}^\infty$. 
Further, we define
\begin{align*}
	\widetilde{W}_l:=\sum_{k=1}^\infty Q^{\frac{1}{2}}h_k\tilde{\beta}_{k,l},\quad l\in\mbb N^+,
\end{align*}
which are independent $U$-valued $Q$-Wiener processes and are all independent of $W$.
With the above preparation, one observes that $\widetilde{U}^{\infty,n}$ converges formally to the solution $U$ of the following equation 
\begin{align}\label{U}
	U(t)=&\;\int_0^tE(t-s)\mcal DF(X(s))U(s)\ud s+\int_0^tE(t-s)\mcal DG(X(s))U(s)\ud W(s)\notag\\
	&-\sqrt{\frac{T}{2}}\sum_{l=1}^\infty \int_0^tE(t-s)\mcal DG(X(s))\big(G(X(s))Q^{\frac{1}{2}}h_l\big)\ud \widetilde{W}_l(s),\quad t\in[0,T].
\end{align}

The following lemma gives the well-posedness of \eqref{U}. 

\begin{lem}\label{Uwellposed}
	Let Assumptions \ref{assum1}--\ref{assum3} hold. Then the equation \eqref{U} admits a unique solution satisfying 
\[
\sup\limits_{t\in[0,T]}\mbf E\|U(t)\|_{\gamma}^2\le K(T,\gamma)
\]
	 for any $\gamma\in[0,1)$. 
\end{lem}
\begin{proof}
For any $l\in\mbb N^+$, we denote for simplicity $H_l(t):=\int_0^tE(t-s)\mcal DG(X(s))\big(G(X(s))Q^{\frac{1}{2}}h_l\big)\ud \widetilde{W}_l(s)$ for $t\in[0,T]$ and $\Phi_{l}^{\gamma,t}(s):=(-A)^{\frac{\gamma}{2}}E(t-s)\mcal DG(X(s))\big(G(X(s))Q^{\frac{1}{2}}h_{l}\big)$ for $s\in[0,t]$ and fixed $\gamma\in[0,1)$. 

We first verify the convergence of the series of $\sum_{l=1}^\infty H_l(t)$ in $\mbf L^2(\Omega;\dot{H}^\gamma)$. 
The It\^o isometry, \eqref{semigroup1}, \eqref{G'}, $\|BQ^{\frac{1}{2}}h_l\|\le \|B\|_{\mcal L_2^0}$ for $B\in\mcal L_2^0$, and \eqref{Xspatial} lead to
\begin{align}
	\mbf E\|H_l(t)\|_\gamma^2&=\mbf E\int_0^t\big\|\Phi_{l}^{\gamma,t}(s)\big\|_{\mcal L_2^0}^2\ud s 	\le \int_0^t(t-s)^{\gamma}\mbf E\|G(X(s))Q^{\frac{1}{2}}h_l\|^2\ud s\notag\\
	&\le \int_0^t(t-s)^{\gamma}(1+\mbf E\|X(s)\|^2)\ud s \le K(T).\label{sec3eq11}
\end{align}
For any $k_1,k_2\in\mbb N^+$ with $k_2\ge k_1$, it holds
\begin{align*}
	\mbf E\Big\|\sum_{l=k_1}^{k_2}H_l(t)\Big\|_\gamma^2=&\;\sum_{l=k_1}^{k_2}\mbf E\Big\|\int_0^t\Phi_{l}^{\gamma,t}(s)\ud \widetilde W_l(s)\Big\|^2
	+\sum_{\substack{l_1\neq l_2 \\k_1\le l_1,l_2\le k_2}}\mbf E\Big\LL \int_0^t\Phi^{\gamma,t}_{l_1}(s)\ud\widetilde{W}_{l_1}(s),\int_0^t\Phi^{\gamma,t}_{l_2}(s)\ud\widetilde{W}_{l_2}(s)\Big\RR\\
	=&\;\sum_{l=k_1}^{k_2}\mbf E\Big\|\int_0^t\Phi_{l}^{\gamma,t}(s)\ud \widetilde W_l(s)\Big\|^2,
\end{align*} 
where in the last step we used the fact $\Phi_l^{\gamma,t}\in\mcal N^2_{\widetilde{W}_l}(0,t;\mcal L_2^0)$ according to \eqref{sec3eq11} and Proposition \ref{independence}.
Then it follows from the It\^o isometry, \eqref{semigroup1}, \eqref{G'}, \eqref{assum3.2}, and \eqref{Xspatial} that
\begin{align}
\mbf E\Big\|\sum_{l=k_1}^{k_2}H_l(t)\Big\|_\gamma^2
\le&\; K\sum_{l=k_1}^{k_2}\mbf E\int_0^t(t-s)^{\gamma}\|G(X(s))Q^{\frac{1}{2}}h_l\|^2\ud s \notag\\
\le &\;K\sum_{l=k_1}^{k_2}\int_0^t(t-s)^{\gamma}\mbf E\|G(X(s))Q^{-\frac{\beta_2}{2}}Q^{\frac{1}{2}}h_l\|^2q_l^{\beta_2}\ud s \notag\\
\le &\; K\big(\sup_{l\ge k_1}q_l\big)^{\beta_2}\int_0^t(t-s)^{\gamma}\mbf E\|G(X(s))Q^{-\frac{\beta_2}{2}}\|_{\mcal L_2^0}^2\ud s \notag\\
\le &\; K\big(\sup_{l\ge k_1}q_l\big)^{\beta_2}\int_0^t(t-s)^{\gamma}(1+\mbf E\|X(s)\|^2)\ud s \notag\\
\le&\; K(T,\gamma)\big(\sup_{l\ge k_1}q_l\big)^{\beta_2},\quad\forall~t\in[0,T]. \label{sec3eq12}
\end{align}
Noting that $\lim_{k_1\to\infty}\big(\sup_{l\ge k_1}q_l\big)^{\beta_2}=0$ since $\tr(Q)=\sum_{i=1}^\infty q_i<\infty$, we have  that $\sum_{l=1}^\infty H_l(t)$ is a Cauchy sequence in $\mbf L^2(\Omega;\dot{H}^\gamma)$, and thus converges in $\mbf L^2(\Omega;\dot{H}^\gamma)$. Taking $k_1=1$ and passing to the limit $k_2\to\infty$ in \eqref{sec3eq12}, we obtain 
\begin{align*}
	\sup_{t\in[0,T]}\mbf E\Big\|\sum_{l=1}^\infty H_l(t)\Big\|_\gamma^2\le K(T,\gamma)\big(\sup_{l\ge 1}q_l\big)^{\beta_2}.
\end{align*}

The proof is then completed by proving that \eqref{U} admits a unique solution based on a standard argument utilizing the contraction mapping theorem. 
\end{proof}

Now we are in position to show that $\widetilde{U}^{\infty,n}(t)$ converges to $U(t)$.

\begin{lem}\label{UtildeinftyntoU}
	Let Assumptions \ref{assum1}--\ref{assum3} hold with $\beta_1\in(0,1)$. Then for any $t\in[0,T]$ and $\eta\in[0,1-\beta_1)$, it holds
	\[
	\lim\limits_{n\to\infty}\mbf E\|\widetilde{U}^{\infty,n}(t)-U(t)\|_\eta^2=0,
	\]
	where $U$ is given by \eqref{U}.
\end{lem}

\begin{proof}
By \eqref{Utildeinfty-n} and \eqref{U}, we write $U(t)-\widetilde{U}^{\infty,n}(t)=J_1^n(t)+J_2^n(t)-\sqrt{\frac{T}{2}}J_3^n(t)$, where
\begin{align*}
J_1^n(t):=&\;\int_0^tE(t-s)\mcal DF(X(s))U(s)\ud s-\int_0^tE_n(t-s)P_n\mcal DF(X(s))\widetilde{U}^{\infty,n}(s)\ud s,\\
J_2^n(t):=&\;\int_0^tE(t-s)\mcal DG(X(s))U(s)\ud W(s)-\int_0^tE_n(t-s)P_n\mcal DG(X(s))\widetilde{U}^{\infty,n}(s)\ud W^n(s), \\
J_3^n(t):=&\;\sum_{l=1}^\infty\int_0^tE(t-s)\mcal DG(X(s))\big(G(X(s))Q^{\frac{1}{2}}h_l\big)\ud \widetilde{W}_l(s) \\
&\;-\sum_{l=1}^n\int_0^tE_n(t-s)P_n\mcal DG(X(s))\big(P_nG(X(s))Q^{\frac{1}{2}}h_l\big)\ud \widetilde{W}_l^n(s).
\end{align*}

We decompose $J_1^n(t)$ into $J_1^n(t)=J_{1,1}^n(t)+J_{1,2}^n(t)$ with
\begin{align*}
	J_{1,1}^n(t)&:=\int_0^tE(t-s)\big(Id_H-P_n\big)\mcal DF(X(s))U(s)\ud s,\\
	J_{1,2}^n(t)&:=\int_0^tE(t-s)P_n\mcal DF(X(s))\big(U(s)-\widetilde{U}^{\infty,n}(s)\big)\ud s.
\end{align*}
Applying \eqref{semigroup1}, \eqref{Pn-Id}, \eqref{F'}, and Lemma \ref{Uwellposed} yields
\begin{align*}
	\|J_{1,1}^n(t)\|_{\mbf L^2(\Omega;\dot{H}^\eta)}\le&\; \int_0^t\|(-A)^{\frac{\eta+\beta_1}{2}}E(t-s)\|_{\mcal L(H)}\|(-A)^{-\frac{\beta_1}{2}}\big(Id_H-P_n\big)\|_{\mcal L(H)}\|U(s)\|_{\mbf L^2(\Omega;\dot{H}^\eta)}\ud s \\
	\le&\; K\lambda_{n+1}^{-\frac{\beta_1}{2}}.
\end{align*}
Moreover, by \eqref{norm}, \eqref{semigroup1}, and \eqref{F'}, it holds that
\begin{align*}
	\|J_{1,2}^n(t)\|_{\mbf L^2(\Omega;\dot{H}^\eta)}\le K \int_0^t(t-s)^{-\frac{\eta}{2}}\|U(s)-\widetilde{U}^{\infty,n}(s)\|_{\mbf L^2(\Omega;\dot{H}^\eta)}\ud s.
\end{align*}
Then applying the H\"older inequality yields
\begin{align}\label{J1n}
	\mbf E\|J_1^n(t)\|_\eta^2\le K\lambda_{n+1}^{-\beta_1}+K\int_0^t(t-s)^{-\eta}\mbf E\|U(s)-\widetilde{U}^{\infty,n}(s)\|_{\eta}^2\ud s.
\end{align}

Further, we decompose $J_2^n(t)$ into $J_2^n(t)=\sum_{i=1}^{3}J_{2,i}^n(t)$ with
\begin{align*}
	J_{2,1}^n(t)&:=\int_0^tE(t-s)\big(Id_H-P_n\big)\mcal DG(X(s))U(s)\ud W(s), \\
	J_{2,2}^n(t)&:=\int_0^tE(t-s)P_n\mcal DG(X(s))U(s)\ud (W(s)-W^n(s)),\\
	J_{2,3}^n(t)&:=\int_0^tE(t-s)P_n\mcal DG(X(s))\big(U(s)-\widetilde{U}^{\infty,n}(s)\big)\ud W^n(s).
\end{align*}
Using It\^o isometry, \eqref{semigroup1}, \eqref{Pn-Id}, and Lemma \ref{Uwellposed}, we obtain
\begin{align*}
	\mbf E\|J_{2,1}^n(t)\|_\eta^2 \le K\lambda_{n+1}^{-\beta_1}.
\end{align*}
In addition, it can be shown that
\begin{align*}
\mbf E\|J_{2,3}(t)\|_\eta^2\le K\int_0^t(t-s)^{-\eta}\mbf E\|U(s)-\widetilde{U}^{\infty,n}(s)\|^2\ud s.	
\end{align*}
For $J_{2,2}^n$, recall that $P_{n,U}$ is the projection operator from $U$ onto $\text{span}\{h_1,\ldots,h_n\}$. Then the It\^o isometry, \eqref{semigroup1}, \eqref{assum3.1}, and Lemma \ref{Uwellposed} yield
\begin{align*}
\mbf E\|J_{2,2}(t)\|_\eta^2
=&\;\mbf E\int_0^t\|(-A)^{\frac{\eta}{2}}E(t-s)P_n\mcal DG(X(s))U(s)Q^{\frac{1}{2}}\big(Id_U-P_{n,U}\big)\|_{\mbf L^2(U,H)}^2\ud s\\
\le &\;\mbf E\int_0^t\|(-A)^{\frac{\eta+\beta_1}{2}}E(t-s)P_n\|_{\mcal L(H)}^2\|(-A)^{-\frac{\beta_1}{2}}\mcal DG(X(s))U(s)Q^{-\frac{\beta_2}{2}}Q^{\frac{1}{2}}\|^2_{\mcal L_2(U,H)}\\
&\qquad\cdot\|Q^{\frac{\beta_2}{2}}\big(Id_U-P_{n,U}\big)\|_{\mcal L(U)}^2\ud s\\
\le&\; K\big(\sup_{k\ge n+1}q_k\big)^{\beta_2}\int_0^t(t-s)^{-(\eta+\beta_1)}\mbf E\|U(s)\|^2\ud s\\
\le&\; K\big(\sup_{k\ge n+1}q_k\big)^{\beta_2}.
\end{align*}
Accordingly, 
\begin{align}\label{J2n}
	\mbf E\|J_2^n(t)\|_\eta^2\le K\Big(\lambda_{n+1}^{-\beta_1}+\big(\sup_{k\ge n+1}q_k\big)^{\beta_2}\Big)+K\int_0^t(t-s)^{-\eta}\mbf E\|U(s)-\widetilde{U}^{\infty,n}(s)\|_\eta^2\ud s.
\end{align}

We proceed to tackle $J_3^n$, which can be decomposed into $J_3^n(t)=\sum_{i=1}^{4}J^n_{3,i}(t)$ with
\begin{align*}
	J_{3,1}^n(t)&:=\sum_{l=n+1}^\infty\int_0^tE(t-s)\mcal DG(X(s))\big(G(X(s))Q^{\frac{1}{2}}h_l\big)\ud\widetilde{W}_l(s),\\
	J_{3,2}^n(t)&:=\sum_{l=1}^n\int_0^tE(t-s)\mcal DG(X(s))\big(G(X(s))Q^{\frac{1}{2}}h_l\big)\ud \big(\widetilde{W}_l(s)-\widetilde{W}_l^n(s)\big),\\	J_{3,3}^n(t)&=\sum_{l=1}^n\int_0^tE(t-s)\mcal DG(X(s))\big((Id_H-P_n)G(X(s))Q^{\frac{1}{2}}h_l\big)\ud \widetilde{W}_l^n(s),\\
	J_{3,4}^n(t)&:=\sum_{l=1}^n\int_0^tE(t-s)\big(Id_H-P_n\big)\mcal DG(X(s))\big(P_nG(X(s))Q^{\frac{1}{2}}h_l\big)\ud \widetilde{W}_l^n(s).
\end{align*}
It follows from Proposition \ref{independence}, the It\^o isometry, \eqref{semigroup1}, \eqref{Xspatial}, \eqref{G'}, and \eqref{assum3.2} that
\begin{align*}
	\mbf E\|J_{3,1}^n(t)\|_\eta^2&=\sum_{l=n+1}^\infty\mbf E\Big\|\int_0^t(-A)^{\frac{\eta}{2}}E(t-s)\mcal DG(X(s))\big(G(X(s))Q^{\frac{1}{2}}h_l\big)\ud\widetilde{W}_l(s)\Big\|^2  \\
	&\le K\sum_{l=n+1}^\infty\int_0^t(t-s)^{-\eta}\mbf E\|G(X(s))Q^{\frac{1}{2}}h_l\|^2\ud s\\
	&\le K\sum_{l=n+1}^\infty\int_0^t(t-s)^{-\eta}\mbf E\|G(X(s))Q^{-\frac{\beta_2}{2}}Q^{\frac{1}{2}}h_l\|^2q_l^{\beta_2}\ud s\\
	&\le K\big(\sup_{l\ge n+1}q_l\big)^{\beta_2}\int_0^t(t-s)^{-\eta}\mbf E\|G(X(s))Q^{-\frac{\beta_2}{2}}\|_{\mcal L_2^0}^2\ud s\\
	&\le K\big(\sup_{l\ge n+1}q_l\big)^{\beta_2}.
\end{align*}
Applying Proposition \ref{independence}, \eqref{semigroup1}, \eqref{assum3.1}, and $\|Q^{\gamma}(Id_U-P_{n,U})\|_{\mcal L(U)}= \big(\sup\limits_{k\ge n+1}q_k\big)^{\gamma}$ for $\gamma\ge 0$, we deduce
\begin{align*}
\mbf E\|J_{3,2}(t)\|_\eta^2
=&\;\sum_{l=1}^n\mbf E\int_0^t\|(-A)^{\frac{\eta+\beta_1}{2}}E(t-s)(-A)^{-\frac{\beta_1}{2}}\mcal DG(X(s))\big(G(X(s))Q^{\frac{1}{2}}h_l\big)Q^{-\frac{\beta_2}{2}}Q^{\frac{1}{2}}\\
&\qquad\qquad\quad\circ Q^{\frac{\beta_2}{2}}\big(Id_U-P_{n,U}\big)\|_{\mcal L_2(U,H)}^2\ud s \\
\le&\; K\big(\sup\limits_{k\ge n+1}q_k\big)^{\beta_2}\sum_{l=1}^n\mbf E\int_0^t(t-s)^{-(\eta+\beta_1)}\|(-A)^{-\frac{\beta_1}{2}}\mcal DG(X(s))\big(G(X(s))Q^{\frac{1}{2}}h_l\big)Q^{-\frac{\beta_2}{2}}\|_{\mcal L_2^0}^2\ud s \\
\le &\; K\big(\sup\limits_{k\ge n+1}q_k\big)^{\beta_2}\mbf E\int_0^t(t-s)^{-(\eta+\beta_1)}\sum_{l=1}^n\|G(X(s))Q^{\frac{1}{2}}h_l\|^2\ud s \\
\le &\;K\big(\sup\limits_{k\ge n+1}q_k\big)^{\beta_2}\int_0^t(t-s)^{-(\eta+\beta_1)}\mbf E\|G(X(s))\|_{\mcal L_2^0}^2\ud s \\
\le &\; K\big(\sup\limits_{k\ge n+1}q_k\big)^{\beta_2}.
\end{align*}
By Proposition \ref{independence}, the It\^o isometry, \eqref{semigroup1}, \eqref{Ggrow}, \eqref{G'}, and \eqref{Xspatial}, 
\begin{align*}
	\mbf E\|J_{3,3}^n(t)\|^2_\eta&=\sum_{l=1}^{n}\mbf E\int_0^t\|(-A)^{\frac{\eta}{2}}E(t-s)\mcal DG(X(s))\big((Id_H-P_n)G(X(s))Q^{\frac{1}{2}}h_l\big)\|_{\mcal L_2^0}^2\ud s \\
	&\le K\sum_{l=1}^{n}\mbf E\int_0^t(t-s)^{-\eta}\|(-A)^{-\frac{\sigma}{2}}(Id_H-P_n)\|_{\mcal L(H)}^2\|(-A)^{\frac{\sigma}{2}}G(X(s))Q^{\frac{1}{2}}h_l\|^2\ud s\\
	&\le K \lambda_{n+1}^{-\sigma}\int_0^t(t-s)^{-\eta}\mbf E\|(-A)^{\frac{\sigma}{2}}G(X(s))\|_{\mcal L_2^0}^2\ud s	\le K \lambda_{n+1}^{-\sigma}.
\end{align*}
For $J_{3,4}^n$, one can validate that
$\mbf E\|J_{3,4}^n(t)\|_\eta^2\le K\lambda_{n+1}^{-\beta_1}$. Based on the previous estimates for $J_{3,i}^n$ with $i=1,2,3,4$, we arrive at
\begin{align}\label{J3n}
	\mbf E\|J_3^n(t)\|_\eta^2\le K\Big(\big(\sup\limits_{k\ge n+1}q_k\big)^{\beta_2}+\lambda_{n+1}^{-\min(\sigma,\beta_1)}\Big).
\end{align}

Then \eqref{J1n}--\eqref{J3n} yield
\begin{align*}
\mbf E\|U(t)-\widetilde{U}^{\infty,n}(t)\|_\eta^2\le K\int_0^t(t-s)^{-\eta}\mbf E\|U(s)-\widetilde{U}^{\infty,n}(s)\|_\eta^2\ud s+K\Big(\big(\sup\limits_{k\ge n+1}q_k\big)^{\beta_2}+\lambda_{n+1}^{-\min(\sigma,\beta_1)}\Big),
\end{align*}
which finishes the proof as a result of the Gronwall inequality. 
\end{proof}

\subsection{Convergence in distribution of $U^m(t)$}
Based on the results obtained in previous subsections, we are now able to state our main result on the convergence in distribution of $U^m(t)$. 
\begin{theo}\label{maintheo}
	Let Assumptions \ref{assum1}--\ref{assum3} hold with $\sigma,\beta_1\in(0,1)$. Then for any $t\in[0,T]$ and $\eta\in\big[0,\min(\sigma,1-\beta_1)\big)$, $U^m(t)\overset{d}{\Longrightarrow}U(t)$ in $\dot{H}^\eta$ as $m\to\infty$, where $U$ is given by \eqref{U}.
\end{theo}
\begin{proof}
Fix $t\in[0,T]$ and $\eta\in[0,\min(\sigma,1-\beta_1))$. 

By Lemma \ref{Utildemn-Utildem}, Condition (A1) of Theorem \ref{uniform approximation} is satisfied by $\widetilde{U}^{m,n}(t)$ and $\widetilde{U}^m(t)$ with $\mcal X=\dot{H}^\eta$. 
Further, according to Lemma \ref{Utildemn-mtoinfty} and the continuous mapping theorem (a continuous mapping preserves the convergence in distribution of random variables), $\widetilde{U}^{m,n}(t)\overset{d}{\Longrightarrow}\widetilde{U}^{\infty,n}(t)$ in $H_n$ for fixed $n\in \mbb N^+$. Noting that the $\|\cdot\|$-norm and the $\|\cdot\|_\eta$-norm are equivalent in $H_n$, it also holds that $\widetilde{U}^{m,n}(t)\overset{d}{\Longrightarrow}\widetilde{U}^{\infty,n}(t)$ in $\dot{H}^\eta$ for fixed $n\in \mbb N^+$, which verifies Condition (A2) of Theorem \ref{uniform approximation}. 
In addition, Lemma \ref{UtildeinftyntoU} implies $\widetilde{U}^{\infty,n}(t)\overset{d}{\Longrightarrow}U(t)$ in $\dot{H}^\eta$, which verifies
Condition (A3) of Theorem \ref{uniform approximation}. 
We then conclude that $\widetilde{U}^m(t)\overset{d}{\Longrightarrow}U(t)$ in $\dot{H}^\eta$ as $m\to\infty$ based on Theorem \ref{uniform approximation}.

By Lemma \ref{auxiliary}, $\|\widetilde{U}^m(t)-U^m(t)\|_\eta$ converges to $0$ in probability, which, combined with  $\widetilde{U}^m(t)\overset{d}{\Longrightarrow}U(t)$ in $\dot{H}^\eta$ and Slutzky's theorem (cf. \cite[Theorem
13.18]{Klenke}), yields $U^m(t)\overset{d}{\Longrightarrow}U(t)$ in $\dot{H}^\eta$. 
\end{proof}

\section{Applications of the main result}\label{Sec4}
In this section,  we present several applications of the main result, i.e., Theorem \ref{maintheo}, including the asymptotic error distribution of  the exponential Euler method for general SODEs, the asymptotic error distribution of a fully discrete exponential Euler method for general SPDEs, and a concrete example of a stochastic heat equation to which the main result can be applied.

\subsection{Asymptotic error distribution of the exponential Euler method for SODEs}
We consider the finite-dimensional counterpart of \eqref{SPDE} and the corresponding exponential Euler method by setting $H=\mbb R^d$ and $U=\mbb R^m$. In addition, we set $A=L\in\mbb R^{d\times d}$ as a negative definite matrix, $Q=I_m\in\mbb R^{m\times m}$ as the identity matrix with the classical orthonormal eigenbasis $\{h_i\in\mbb R^m: \text{the $i$th element is 1}\}_{i\in\mbb N^+}$, $B=\{(B^1(t),B^2(t),\ldots,B^m(t))^\top,t\in[0,T]\}$ as an $m$-dimensional standard Brownian motion defined on $(\Omega,\mcal F,\mbf P)$, and assume that $F=f:\mbb R^d\to\mbb R^d$ and $G=g=(g_1,\ldots,g_m):\mbb R^d\to\mbb R^{d\times m}$ are globally Lipschitz continuous. 

In this SODE setting, we have $\dot{H}^\gamma=\mbb R^d$ for any $\gamma\in\mbb R$, and the equation \eqref{SPDE} reduces to the following $d$-dimensional SODE
\begin{align*}
	\begin{cases}
	\ud Y(t)=LY(t)\ud t+f(Y(t))\ud t+g(Y(t))\ud B(t),\quad t\in[0,T], \\
	Y(0)=Y_0\in\mbb R^d.
\end{cases}
\end{align*}
Also, the continuous numerical solution $Y^m$ of the exponential Euler method satisfies
\begin{align*}
Y^m(t)=e^{tL}Y_0+\int_0^te^{(t-\kappa_m(s))L}f(Y^m(\kappa_m(s)))\ud s+\int_0^te^{(t-\kappa_m(s))L}g(Y^m(\kappa_m(s)))\ud B(s),\quad t\in[0,T].
\end{align*}

As an immediate result of Theorem \ref{maintheo}, we can obtain the asymptotic error distribution of $Y^m$. 
\begin{cor}\label{cor-SODE}
Assume that $f$ and $g$ are twice continuously differentiable with bounded first and second order derivatives. Then for any $t\in[0,T]$, $m^{\frac{1}{2}}\big(Y^m(t)-Y(t)\big)\overset{d}{\Longrightarrow}M(t)$ with $M$ solving the following SODE
\begin{align*}
	M(t)=&\;\int_0^te^{(t-s)L}\mcal Df(Y(s))M(s)\ud s+\int_0^te^{(t-s)L}\mcal Dg(Y(s))M(s)\ud B(s) \\
	&-\sqrt{\frac{T}{2}}\sum_{j=1}^{m}\int_0^te^{(t-s)L}\mcal Dg(Y(s))g_j(Y(s))\ud \widetilde{B}_j(s),\quad t\in[0,T],
\end{align*}
where $\widetilde{B}_1,\ldots,\widetilde{B}_m$ are independent $m$-dimensional standard Brownian motions and independent of $B$. 
\end{cor}

\subsection{Asymptotic error distribution of a fully discrete exponential Euler method}
In this part, we study the  asymptotic error distribution of a fully discrete numerical method applied to \eqref{SPDE}, based on the temporal exponential Euler method and spatial finite element method. 

Let $(S_h)_{h\in(0,1]}$ be a sequence of finite-dimensional subspaces of $\dot{H}^1$ and $R_h:\dot{H}^1\to S_h$ the Ritz projector onto $S_h$ with respect to the inner product $\LL\cdot,\cdot\RR_1=\LL(-A)^{\frac{1}{2}}\cdot,(-A)^{\frac{1}{2}}\cdot\RR$ in $\dot{H}^1$, i.e.,
\begin{align*}
	\LL R_hx,y_h\RR_1=\LL x,y_h\RR_1,\quad\forall~x\in\dot{H}^1,~y_h\in S_h.
\end{align*} 
We introduce the following assumption on the operator $R_h$.
\begin{assum}\label{assum4}
For $s=1,2$ and $h\in(0,1]$, there is a constant $K>0$ independent of $h$ such that
	\begin{align*}
		\|R_hx-x\|\le Kh^s\|x\|_s,\quad\forall~x\in\dot{H}^s.
	\end{align*}
\end{assum}  
Let the operator $\tilde A_h: S_h\to S_h$ be the discrete version of $-A$. More precisely, for $x_h\in S_h$, $\tilde A_hx_h$ is defined as the unique element satisfying
\begin{align*}
\LL \tilde A_hx_h,y_h\RR=\LL x_h,y_h\RR_1,\quad\forall~y_h\in S_h. 
\end{align*}   
Then the operator $\tilde A_h$ is self-adjoint and positive definite on $S_h$ and $-\tilde A_h$ generates an analytic semigroup of contractions on $S_h$, denoted by $\{\tilde E_h(t):=e^{-\tilde A_ht}\}_{t\ge0}$. Additionally, let $\tilde P_h:\dot{H}^{-1}\to S_h$ be the generalized orthogonal projector onto $S_h$ defined by
\begin{align*}
	\LL \tilde P_hx,y_h\RR=\LL (-A)^{-1}x,y_h\RR_1,\quad\forall~x\in\dot{H}^{-1},~y_h\in S_h.
\end{align*}
One can show that, when restricted to $H$, $\tilde P_h$ 
coincides with the usual orthogonal projector onto $S_h$ with respect to the inner product $\LL\cdot,\cdot\RR$.
Under Assumption \ref{assum4}, one has the following error estimate for $\tilde{E}_h(t)\tilde P_h-E(t)$ (cf. \cite[Lemma 3.8]{Krusebook}).
\begin{pro}\label{pro3}
	Let Assumption \ref{assum4} hold and $0\le\nu\le\mu\le 2$. Then it holds
	\begin{align*}
	\|\big(\tilde{E}_h(t)\tilde P_h-E(t)\big)x\|\le K(\mu,\nu)h^\mu t^{-\frac{\mu-\nu}{2}}\|x\|_\nu,\quad\forall~x\in\dot{H}^\nu,~t>0,~h\in(0,1].
	\end{align*}
\end{pro}

For the fully discrete method based on the temporal exponential Euler method and spatial finite element method, whose continuous numerical solution $X^m_h$ satisfies
\begin{align}
	X^m_h(t)=&\;\tilde{E}_h(t)\tilde{P}_hX_0+\int_0^t\tilde{E}_h(t-\kappa_m(s))\tilde{P}_hF(X^m_h(\kappa_m(s)))\ud s \notag\\
	&\;+\int_0^t\tilde{E}_h(t-\kappa_m(s))\tilde{P}_hG(X^m_h(\kappa_m(s)))\ud W(s),\quad t\in[0,T], \label{Xmh}
\end{align}
we next present its spatial strong convergence rate and asymptotic error distribution. 

\begin{lem}\label{Xmh-Xm}
	Let Assumptions \ref{assum1} and \ref{assum4} hold. Then for any $\epsilon\in(0,1)$, it holds for all $h\in(0,1]$ that
	\begin{align*}
		\sup_{t\in[0,T]}\|X^m_h(t)-X^m(t)\|_{\mbf L^p(\Omega;H)}\le K(\epsilon)h^{1+\sigma-\epsilon}. 
	\end{align*}
\end{lem}
\begin{proof}
By the expressions \eqref{Xmmild} and \eqref{Xmh}, we get $X^m_h(t)-X^m(t)=\sum_{i=1}^{3}M_i(t)$ for $t\in[0,T]$ with
\begin{align*}
M_1(t)&:=\big(\tilde{E}_h(t)\tilde P_h-E(t)\big)X_0, \\
M_2(t)&:=\int_0^t\big(\tilde{E}_h(t-\kappa_m(s))\tilde P_hF(X^m_h(\kappa_m(s)))-E(t-\kappa_m(s))F(X^m(\kappa_m(s)))\big)\ud s,\\
M_3(t)&:=\int_0^t\big(\tilde{E}_h(t-\kappa_m(s))\tilde P_hG(X^m_h(\kappa_m(s)))-E(t-\kappa_m(s))G(X^m(\kappa_m(s)))\big)\ud W(s).
\end{align*} 
	
Applying Proposition \ref{pro3} with $\mu=\nu=1+\sigma$ yields
\begin{align}
	\|M_1(t)\|_{\mbf L^p(\Omega;H)}\le Kh^{1+\sigma}\|X_0\|_{\mbf L^p(\Omega;\dot{H}^{1+\sigma})}.\label{M1}
\end{align}

Further, we decompose $M_2(t)$ into $M_{2}(t)=M_{2,1}(t)+M_{2,2}(t)$ with
\begin{align*}
	M_{2,1}(t)&:=\int_0^t\tilde E_h(t-\kappa_m(s))\tilde P_h\big(F(X^m_h(\kappa_m(s)))-F(X^m(\kappa_m(s)))\big)\ud s,\\
	M_{2,2}(t)&:=\int_0^t\Big(\tilde{E}_h(t-\kappa_m(s))\tilde P_h-E(t-\kappa_m(s))\Big)F(X^m(\kappa_m(s)))\ud s.
\end{align*}
It follows from the condition \eqref{FLip} and the contraction property of $\tilde{E}_h$ and $\tilde{P}_h$ that
\begin{align*}
	\|M_{2,1}(t)\|_{\mbf L^p(\Omega;H)}\le K\int_0^t\|X^m_h(\kappa_m(s))-X^m(\kappa_m(s))\|_{\mbf L^p(\Omega;H)}\ud s.
\end{align*}
Applying Proposition \ref{pro3} with $\mu=1+\sigma$ and $\nu=0$, and using the linear growth property of $F$ and Lemma \ref{Xmregularity}(i), we arrive at
\begin{align*}
	\|M_{2,2}(t)\|_{\mbf L^p(\Omega;H)}\le Kh^{1+\sigma}\int_0^t(t-\kappa_m(s))^{-\frac{1+\sigma}{2}}(1+\|X^m(\kappa_m(s))\|_{\mbf L^p(\Omega;H)})\ud s\le Kh^{1+\sigma}.
\end{align*}
This then leads to
\begin{align}\label{M2}
	\|M_2(t)\|_{\mbf L^p(\Omega;H)}\le Kh^{1+\sigma}+K\int_0^t\sup_{r\in[0,s]}\|X^m_h(r)-X^m(r)\|_{\mbf L^p(\Omega;H)}\ud s.
\end{align} 

For $M_3$, we split it as $M_3(t)=M_{3,1}(t)+M_{3,2}(t)$ with
\begin{align*}
	M_{3,1}(t)&:=\int_0^t\tilde E_h(t-\kappa_m(s))\tilde P_h\big(G(X^m_h(\kappa_m(s)))-G(X^m(\kappa_m(s)))\big)\ud W(s),\\
	M_{3,2}(t)&:=\int_0^t\Big(\tilde{E}_h(t-\kappa_m(s))\tilde P_h-E(t-\kappa_m(s))\Big)G(X^m(\kappa_m(s)))\ud W(s).
\end{align*}
The Burkholder--Davis--Gundy (BDG) inequality, the contraction property of $\tilde{E}_h$ and $\tilde P_h$, and \eqref{GLip} yield
\begin{align*}
	\|M_{3,1}(t)\|_{\mbf L^p(\Omega;H)}&\le K\Big\|\Big(\int_0^t\|\tilde E_h(t-\kappa_m(s))\tilde P_h\big(G(X^m_h(\kappa_m(s)))-G(X^m(\kappa_m(s)))\big)\|_{\mcal L_2^0}^2\ud s\Big)^{\frac{1}{2}}\Big\|_{\mbf L^p(\Omega;R)}\\
	&\le K\Big\|\Big(\int_0^t\|X^m_h(\kappa_m(s))-X^m(\kappa_m(s))\big)\|_{\mcal L_2^0}^2\ud s\Big)^{\frac{1}{2}}\Big\|_{\mbf L^p(\Omega;R)}\\
	&\le K\Big(\int_0^t\|X^m_h(\kappa_m(s))-X^m(\kappa_m(s))\|_{\mbf L^p(\Omega;H)}^2\ud s\Big)^{\frac{1}{2}}.
\end{align*}
From Proposition \ref{pro3}, we deduce for any $x\in H$ and $0\le \nu\le \mu\le 2$ that
\begin{align*}
	\|\big(\tilde{E}_h(t)\tilde P_h-E(t)\big)(-A)^{-\frac{\nu}{2}}x\|\le Kh^\mu t^{-\frac{\mu-\nu}{2}}\|(-A)^{-\frac{\nu}{2}}x\|_\nu\le Kh^\mu t^{-\frac{\mu-\nu}{2}}\|x\|,~t>0,
\end{align*}
which implies
\begin{align*}
	\|\big(\tilde{E}_h(t)\tilde P_h-E(t)\big)(-A)^{-\frac{\nu}{2}}\|_{\mcal L(H)}\le Kh^\mu t^{-\frac{\mu-\nu}{2}},~t>0.
\end{align*}
For any fixed $\epsilon\in(0,1)$, applying the above inequality with $\mu=1+\sigma-\epsilon$ and $\nu=\sigma$, together with \eqref{GLip}, Lemma \ref{Xmregularity}(i), and the BDG inequality, we obtain
\begin{align*}
&\;\|M_{3,2}(t)\|_{\mbf L^p(\Omega;H)}\\
\le&\; K\Big\|\Big(\int_0^t\|\big(\tilde{E}_h(t-\kappa_m(s))\tilde P_h-E(t-\kappa_m(s))\big)G(X^m(\kappa_m(s)))\|_{\mcal L_2^0}^2\ud s\Big)^{\frac{1}{2}}\Big\|_{\mbf L^p(\Omega;R)}\\
\le&\; K\Big\|\Big(\int_0^t\|\big(\tilde{E}_h(t-\kappa_m(s))\tilde P_h-E(t-\kappa_m(s))\big)(-A)^{-\frac{\sigma}{2}}\|_{\mcal L(H)}^2\|(-A)^{\frac{\sigma}{2}}G(X^m(\kappa_m(s)))\|_{\mcal L_2^0}^2\ud s\Big)^{\frac{1}{2}}\Big\|_{\mbf L^p(\Omega;R)}\\
\le &\; Kh^{1+\sigma-\epsilon}\Big\|\Big(\int_0^t\|(t-\kappa_m(s))^{-(1-\epsilon)}(1+\|X^m(\kappa_m(s))\|_\sigma^2)\ud s\Big)^{\frac{1}{2}}\Big\|_{\mbf L^p(\Omega;R)}\\
\le &\;Kh^{1+\sigma-\epsilon}\Big(\int_0^t(t-s)^{-(1-\epsilon)}(1+\|X^m(\kappa_m(s))\|_{\mbf L^p(\Omega;\dot{H}^\sigma)}^2)\ud s\Big)^{\frac{1}{2}}
\le Kh^{1+\sigma-\epsilon}.
\end{align*}
Accordingly, it follows that
\begin{align}\label{M3}
	\|M_3(t)\|^2_{\mbf L^p(\Omega;H)}\le Kh^{2(1+\sigma-\epsilon)}+K\int_0^t\sup_{r\in[0,s]}\|X^m_h(r)-X^m(r)\|_{\mbf L^p(\Omega;H)}^2\ud s.
\end{align}

A combination of \eqref{M1}--\eqref{M3} gives
\begin{align*}
\sup_{r\in[0,t]}\|X^m_h(r)-X^m(r)\|^2_{\mbf L^p(\Omega;H)}\le Kh^{2(1+\sigma-\epsilon)}+K\int_0^t\sup_{r\in[0,s]}\|X^m_h(r)-X^m(r)\|_{\mbf L^p(\Omega;H)}^2\ud s,
\end{align*}
which finishes the proof due to the Gronwall inequality.
\end{proof}

Based on the above spatial strong convergence rate, by choosing a proper spatial index $h$ related to the temporal index $m$, one can get the asymptotic error distribution for the fully discrete method \eqref{Xmh} as stated in the following corollary.

\begin{cor}\label{fulldiscre}
	Let Assumptions \ref{assum1}--\ref{assum4} hold. Then for any $\iota>\frac{1}{2(1+\sigma)}$ and $t\in[0,T]$, $m^{\frac{1}{2}}\big(X^m_{m^{-\iota}}(t)-X(t)\big)\overset{d}{\Longrightarrow}U(t)$  in $H$ as $m\to\infty$ with $U$ being given by \eqref{U}.
\end{cor}
\begin{proof}
Note that
\begin{align*}
	m^{\frac{1}{2}}\big(X^m_{m^{-\iota}}(t)-X(t)\big)=	m^{\frac{1}{2}}\big(X^m_{m^{-\iota}}(t)-X^m(t)\big)+U^m(t),
\end{align*}
where $U^m$ is defined in \eqref{Um}. 
Applying Lemma \ref{Xmh-Xm} yields that for any $\epsilon\in(0,1)$ and $t\in[0,T]$,
\begin{align}\label{sec4eq2}
	\|m^{\frac{1}{2}}\big(X^m_{m^{-\iota}}(t)-X^m(t)\big)\|_{\mbf L^p(\Omega;H)}\le K(\epsilon)m^{\frac{1}{2}-(1+\sigma-\epsilon)\iota}.
\end{align}
Since $\iota>\frac{1}{2(1+\sigma)}$, there is a sufficiently $\epsilon_0>0$ such that $(1+\sigma-\epsilon_0)\iota>\frac{1}{2}$. Taking $\epsilon=\epsilon_0$ in \eqref{sec4eq2} yields
\begin{align*}
		\|m^{\frac{1}{2}}\big(X^m_{m^{-\iota}}(t)-X^m(t)\big)\|_{\mbf L^p(\Omega;H)}\le K(\epsilon_0)m^{\frac{1}{2}-(1+\sigma-\epsilon_0)\iota}\to 0~\text{as}~m\to\infty.
\end{align*}
Thus, $\|m^{\frac{1}{2}}\big(X^m_{m^{-\iota}}(t)-X^m(t)\big)\|$ converges to $0$ in probability. Then the conclusion comes as a result of Theorem \ref{maintheo} and Slutzky's theorem (cf. \cite[Theorem
13.18]{Klenke}).
\end{proof}

\subsection{Asymptotic error distirbution for an example of SPDE}\label{Sec4.3}
In this subsection, we consider the stochastic heat equation  serving as a concrete example of \eqref{SPDE}. 

Let $\mcal O=(0,1)^d$ with $d\in\{1,2,3\}$ and $H=U=\mbf L^2(\mcal O;\mbb R)$. Consider the stochastic heat equation
\begin{align}\label{SHE}
	\frac{\PD}{\PD t} X(t,x)=\Delta X(t,x)+f(x,X(t,x))+g(x,X(t,x))	\frac{\PD}{\PD t} W(t,x),\quad(t,x)\in(0,T]\times\mcal O
\end{align}
with $X(t,x)=0$ for $(t,x)\in[0,T]\times\partial\mcal O$ and $X(0,x)=X_0(x)$ for $x\in\overline{\mcal O}$. 
Here, $\Delta=\sum_{i=1}^{d}\frac{\partial^2}{\partial x_i^2}$ is the Laplacian with homogeneous Dirichlet boundary condition, and hence admits the eigenfunctions $e_i(x)=2^{\frac{d}{2}}\sin(i_1\pi x_1)\cdots\sin(i_d\pi x_d)$ for $x=(x_1,\ldots,x_d)\in\mcal O$ and $i=(i_1,\ldots,i_d)\in\mbb (\mbb N^+)^d$.
Moreover, $W$ is a $Q$-Wiener process given by \eqref{W} with eigenbasis $\{h_i=e_i\}_{i\in(\mbb N^+)^d}$. 

Denote by $\mbf C^\delta(\mcal O;\mbb R)$ with $\delta\in(0,1]$ the space of $\delta$-H\"older continuous functions, equipped with the norm $\|v\|_{\mbf C^\delta(\mcal O;\mbb R)}:=\|v\|_{\mbf C(\mcal O;\mbb R)}+\sup\limits_{x,y\in\mcal O,x\neq y}\frac{|v(x)-v(y)|}{|x-y|^\delta}$, where $\|v\|_{\mbf C(\mcal O;\mbb R)}:=\sup\limits_{x\in\mcal O}|v(x)|$,
and by $W^{r,2}(\mcal O;\mbb R^d)$ with $r\ge 0$ the usual Sobolev space consisting of functions $v:\mcal O\to\mbb R$ with
\begin{align*}
	\|v\|_{W^{r,2}(\mcal O;\mbb R^d)}:=\Big(\int_\mcal O|v(x)|^2\ud x+\int_\mcal O\int_\mcal O\frac{|v(x)-v(y)|^2}{|x-y|^{(d+2r)}}\ud x\ud y\Big)^{\frac{1}{2}}<\infty.
\end{align*}

Define $F:H\to H$ and $G: H\to\mcal L_2^0$ by 
\begin{align*}
	(F(v))(x)&:=f(x,v(x)),\quad x\in\mcal O,~v\in H, \\
	(G(v)u)(x)&:=g(x,v(x))u(x),\quad x\in\mcal O,~v\in H,~u\in U_0.
\end{align*}

With the above preparation, \eqref{SHE} can be rewritten into the evolution form \eqref{SPDE} with $A=\Delta$.
Next we give the conditions on $X_0$, $f$, $g$, and $Q$.

\begin{con}\label{con1}
	The initial value $X_0$ satisfies $\|X_0\|_{\mbf L^4(\Omega;\dot{H}^{2})}<\infty$.
\end{con}

\begin{con}\label{con2}
	The function $f:\mcal O\times\mbb R\to\mbb R$ is twice continuously differentiable with
	\begin{align*}
		\int_\mcal O|f(x,0)|^2\ud x<\infty,\quad \sup_{x\in\mcal O}\sup_{y\in\mbb R}\Big|\frac{\PD^i}{\PD y^i}f(x,y)\Big|<\infty,\quad i=1,2.
	\end{align*}
\end{con}

\begin{con}\label{con3}
	The function $g:\mcal O\times\mbb R\to\mbb R$ is twice continuously differentiable with
	\begin{align*}
		\sup_{x\in\mcal O}|g(x,0)|+\sup_{x\in\mcal O}\sup_{y\in\mbb R}\Big(\Big|\frac{\PD}{\PD y}g(x,y)\Big|+\Big|\frac{\PD^2}{\PD y^2}g(x,y)\Big|+\Big|\frac{\PD}{\PD x}g(x,y)\Big|\Big)<\infty.
	\end{align*}
\end{con}
\begin{con}\label{con4}
The eigenvalues $q_i$ of $Q$ with $i\in(\mbb N^+)^d$ satisfy $q_i>0$, $\sum_{i\in(\mbb N^+)^d}q_i\|e_i\|^2_{\mbf C^1(\mcal O;\mbb R)}<\infty$, and $\sum_{i\in(\mbb N^+)^d}q_i^{(1-\gamma)}<\infty$ for any $\gamma\in(0,1)$.
\end{con}

\begin{lem}\label{verifyAssum}
	Under Conditions \ref{con1}--\ref{con4}, Assumptions \ref{assum1}--\ref{assum3} hold for $p=4$, $\sigma\in(\frac{1}{4},\frac{1}{2})$, $\alpha\in[\frac{d}{4},\sigma+\frac{1}{2})$, and $\beta_1,\beta_2\in(0,1)$.
\end{lem}

\begin{proof}
By \cite[(14)]{JentzenJDE}, $F:H\to H$ is well-defined and satisfies \eqref{FLip} under Condition \ref{con2}. According to the discussion of \cite[Page 121]{JentzenJDE}, $G:H\to\mcal L_2^0$ is well-defined and satisfies \eqref{GLip} under Condition \ref{con3}. In addition, it follows from \cite[(30)]{JentzenJDE} that
\begin{align*}
	\|(-A)^{r}G(v)\|_{\mcal L_2^0}\le K\big(\sup_{i\in(\mbb N^+)^d}q_i\|e_i\|_{\mbf C^1(\mcal O;\mbb R)}^2\big)(1+\|u\|_{2r})\le K(1+\|u\|_{2r}),\quad \forall~r\in(0,\frac{1}{4})
\end{align*}
under Condition \ref{con4}.
Consequently, \eqref{Ggrow} holds for any $\sigma\in(0,\frac{1}{2})$. The above facts combined with Condition \ref{con1} implies that Assumption \ref{assum1} hold for $p=4$ and all $\sigma\in(0,\frac{1}{2})$.

We deduce from Condition \ref{con2} that
\begin{align*}
	\|\mcal DF(v)u\|=\Big(\int_\mcal O\Big|\frac{\PD}{\PD y}f(x,v(x))u(x)\Big|^2\ud x\Big)^{\frac{1}{2}}\le K\|u\|,\quad\forall~ u\in H,~v\in H,
\end{align*}
which proves \eqref{F'}. In addition, For any $\alpha\ge\frac{d}{4}$,  due to the Sobolev embedding $\dot{H}^\alpha\hookrightarrow\mbf L^4(\mcal O;\mbb R)$ and Condition \ref{con2}, one has
\begin{align*}
	\|\mcal D^2F(v)(u_1,u_2)\|&=\Big(\int_\mcal O\Big|\frac{\PD^2}{\PD y^2}f(x,v(x))u_1(x)u_2(x)\Big|^2\ud x\Big)^{\frac12}\le K\|u_1\|_{\mbf L^4(\mcal O;\mbb R)}\|u_2\|_{\mbf L^4(\mcal O;\mbb R)}\\ 
	&\le K\|u_1\|_{\alpha}\|u_2\|_\alpha,\quad\forall~v,u_1,u_2\in\dot{H}^\alpha,
\end{align*}
which proves \eqref{F''}. 
Further, it has been shown in \cite[Section 4]{JentzenFCM} that $\|\mcal DG(v)u\|_{\mcal L_2^0}\le K\|u\|$ for any $u,v\in H$ under Condition \ref{con3}. Thus, \eqref{G'} holds true. By revisiting the proof of \cite[(38)]{JentzenFCM}, we have
\begin{align*}
	\|\mcal D^2G(v)(u_1,u_2)\|_{\mcal L_2^0}\le K\sqrt{\text{Tr}(Q)}\big(\sup_{i\in(\mbb N^+)^d}\|e_i\|_{\mbf C(\mcal O;\mbb R)}\big)\|u_1\|_{\mbf L^4(\mcal O;\mbb R)}\|u_2\|_{\mbf L^4(\mcal O;\mbb R)}.
\end{align*}
This combined with the Sobolev embedding $\dot{H}^\alpha\hookrightarrow\mbf L^4(\mcal O;\mbb R)$ for $\alpha\ge \frac{d}{4}$ yields 
\begin{align*}
	\|\mcal D^2G(v)(u_1,u_2)\|_{\mcal L_2^0}\le K\|u_1\|_\alpha\|u_2\|_\alpha,\quad \forall ~v,u_1,u_2\in\dot{H}^\alpha,
\end{align*}
which verifies \eqref{G''}. Thus, \eqref{F'}--\eqref{G''} hold for all $\alpha\ge \frac{d}{4}$. Accordingly, Assumption \ref{assum2} is fulfilled for all $\sigma\in(\frac{1}{4},\frac{1}{2})$ and $\alpha\in[\frac{d}{4},\sigma+\frac{1}{2})$. 

We proceed to verify Assumption \ref{assum3}. Note that under Condition \ref{con3}, $|g(x,y)|\le K(1+|y|)$ for any $x\in\mcal O,y\in\mbb R$.
Then for any $\beta_2\in(0,1)$, using Condition \ref{con4} gives
\begin{align*}
	\|G(v)Q^{-\frac{\beta_2}{2}}\|_{\mcal L_2^0}^2&=\sum_{i\in(\mbb N^+)^d}\|G(v)Q^{\frac{1-\beta_2}{2}}e_i\|^2=\sum_{i\in(\mbb N^+)^d}q_i^{1-\beta_2}\|G(v)e_i\|^2 \\
	&\le \sum_{i\in(\mbb N^+)^d}q_i^{1-\beta_2}\Big(\int_\mcal O|g(x,v(x))|^2\ud x\Big)\big(\sup_{i\in(\mbb N^+)^d}\|e_i\|_{\mbf C(\mcal O;\mbb R)}^2\big) \\
	&\le K(1+\|v\|^2)\sum_{i\in(\mbb N^+)^d}q_i^{1-\beta_2}
	\le K(1+\|v\|^2),\quad\forall~v\in H.
\end{align*} 
This implies that \eqref{assum3.2} holds for any $\beta_2\in(0,1)$. Finally, it follows from $\|(-A)^{-\eta}\|_{\mcal L(H)}\le K(\eta)$ for any $\eta\ge 0$ and Conditions \ref{con3}--\ref{con4} that for any $\beta_1,\beta_2\in(0,1)$,
\begin{align*}
	&\;\|(-A)^{-\frac{\beta_1}{2}}\mcal DG(v)uQ^{-\frac{\beta_2}{2}}\|_{\mcal L_2^0}^2=\sum_{i\in(\mbb N^+)^d}\|(-A)^{-\frac{\beta_1}{2}}\mcal DG(v)uQ^{\frac{1-\beta_2}{2}}e_i\|^2 \\
	\le &\;K\sum_{i\in(\mbb N^+)^d}\|DG(v)uQ^{\frac{1-\beta_2}{2}}e_i\|^2=K\sum_{i\in(\mbb N^+)^d}q_i^{(1-\beta_2)}\int_\mcal O\Big|\frac{\PD}{\PD y}g(x,v(x))u(x)e_i(x)\Big|^2\ud x\\
	\le&\; K\big(\sup_{i\in(\mbb N^+)^d}\|e_i\|_{\mbf C(\mcal O;\mbb R)}^2\big)\sum_{i\in(\mbb N^+)^d}q_i^{(1-\beta_2)}\|u\|^2
	\le K\|u\|^2,\quad\forall~u,v\in H,
\end{align*}
which verifies \eqref{assum3.1}. Thus, Assumption \ref{assum3} holds and the proof is complete. 
\end{proof}

As an immediate result of Theorem \ref{maintheo} and Lemma \ref{verifyAssum}, we obtain the asymptotic error distribution for the exponential Euler method, whose continuous solution is denoted by $\{X^m(t,\cdot)\}_{t\in[0,T]}$, applied to \eqref{SHE}.

\begin{theo}\label{SHEasymerror}
Consider the exponential Euler method \eqref{Xmmild} applied to the SPDE \eqref{SHE}. If Conditions \ref{con1}--\ref{con4} holds, then for any $\eta\in[0,\frac{1}{2})$ and $t\in[0,T]$, $m^{\frac{1}{2}}\big(X^m(t,\cdot)-X(t,\cdot)\big)\overset{d}{\Longrightarrow}U(t)$ in $\dot{H}^\eta$ as $m\to\infty$, where $U$ is given by \eqref{U}.
\end{theo}

Next, we show for \eqref{SHE} that when $d=1$ and the function $g$ is affine with respect to the second variable, the conclusion of Theorem \ref{SHEasymerror} can be strengthened.

\begin{lem}\label{verifyAssum'}
Assume that Conditions \ref{con1}, \ref{con2}, and \ref{con4} hold with $d=1$, and in addition $g(x,y)=a_1y+a_2$ with two constants $a_1,a_2\in\mbb R$. Then Assumptions \ref{assum1}--\ref{assum3} hold for $p=4$, $\sigma\in(\frac{1}{2},1)$, $\alpha\in[\frac{1}{4},\sigma+\frac{1}{2})$, and $\beta_1,\beta_2\in(0,1)$.
\end{lem}
\begin{proof}
Note that the current assumption on $g$ implies Condition \ref{con3}. In addition,  we indeed shows in the proof of Lemma \ref{verifyAssum} that \eqref{F'}--\eqref{assum3.2} hold for all $\alpha\ge\frac{d}{4}$ with $d\in\{1,2,3\}$ and for all $\beta_1,\beta_2\in(0,1)$. Thus Assumptions \ref{assum2} and \ref{assum3} hold for all $\sigma\in(\frac{1}{2},1)$, $\alpha\in[\frac{1}{4},\sigma+\frac{1}{2})$, and $\beta_1,\beta_2\in(0,1)$ provided $d=1$. 
It then suffices to prove that \eqref{Ggrow} holds for all $\sigma\in(\frac{1}{2},1)$.  

By \cite[(19)]{JentzenJDE}, for all $\gamma\in(\frac{1}{2},1)$,
\begin{align}\label{sec4eq3}
	\dot{H}^\gamma=\big\{v\in H:~\|v\|_{W^{\gamma,2}((0,1);\mbb R)}<\infty,~v(0)=v(1)=0\big\}.
\end{align}
Further, for any $\sigma\in(\frac{1}{2},1)$ and $v\in\dot{H}^\sigma$, \eqref{sec4eq3} implies $v\in W^{\sigma,2}((0,1);\mbb R)$ and thus $g(\cdot,v(\cdot))\in W^{\sigma,2}((0,1);\mbb R)$. It follows from \cite[(23)]{JentzenJDE}  that
\begin{align}\label{sec4eq4}
	\|g(\cdot,v(\cdot))e_i(\cdot)\|_{W^{\sigma,2}((0,1);\mbb R)}\le \frac{\sqrt{3}}{1-\sigma}\|v\|_{W^{\sigma,2}((0,1);\mbb R)}\|e_i\|_{\mbf C^1((0,1);\mbb R)}<\infty,\quad i\in\mbb N^+,
\end{align}
which implies $g(\cdot,v(\cdot))e_i(\cdot)\in\dot{H}^\sigma$ due to \eqref{sec4eq3} and $e_i(0)=e_i(1)=0$. Then \eqref{sec4eq4} and \cite[(20)]{JentzenJDE} yield
\begin{align*}
	\|g(\cdot,v(\cdot))e_i(\cdot)\|_\sigma\le K(\sigma)\|v\|_\sigma\|e_i\|_{\mbf C^1((0,1);\mbb R)},\quad v\in\dot{H}^\sigma.
\end{align*}
Thus, for any $\sigma\in(\frac{1}{2},1)$, using Condition \ref{con4} gives
\begin{align*}
\|(-A)^{\frac{\sigma}{2}}G(v)\|_{\mcal L_2^0}^2&=\sum_{i=1}^\infty \|(-A)^{\frac{\sigma}{2}}G(v)Q^{
\frac{1}{2}}e_i\|^2=\sum_{i=1}^\infty q_i\|g(\cdot,v(\cdot))e_i(\cdot)\|_\sigma^2\\
& \le K(\sigma)\Big(\sum_{i=1}^\infty q_i\|e_i\|_{\mbf C^1((0,1);\mbb R)}^2\Big)\|v\|_\sigma^2
\le K(\sigma)\|v\|_\sigma^2,\quad\forall~v\in\dot{H}^\sigma,
\end{align*}
which verifies \eqref{Ggrow} and completes the proof. 
\end{proof}

\begin{theo}\label{SHEasymerror'}
Consider the exponential Euler method \eqref{Xmmild} applied to the SPDE \eqref{SHE} with $d=1$. Under assumptions in Lemma \ref{verifyAssum'}, for any $\eta\in[0,1)$ and $t\in[0,T]$, $m^{\frac{1}{2}}\big(X^m(t,\cdot)-X(t,\cdot)\big)\overset{d}{\Longrightarrow}U(t)$ in $\dot{H}^\eta$ as $m\to\infty$, where $U$ is given by \eqref{U}.  

Accordingly, for any $(t,x)\in[0,T]\times(0,1)$, $m^{\frac{1}{2}}\big(X^m(t,x)-X(t,x)\big)\overset{d}{\Longrightarrow}U(t,x)$ in $\mbb R$ as $m\to\infty$. Here, $\{U(t,x),\,(t,x)\in[0,T]\times[0,1]\}$ is interpreted as the solution of 
\begin{align*}
	\frac{\PD}{\PD t}U(t,x)=&\;\frac{\PD^2}{\PD x^2}U(t,x)+\frac{\PD}{\PD y}f(x,X(t,x))U(t,x)+a_1U(t,x)\sum_{k=1}^{\infty}\sqrt{q_k}e_k(x)\frac{\ud}{\ud t}\beta_i(t) \\
	&\;-\sqrt{\frac{T}{2}}a_1(a_1X(t,x)+a_2)\sum_{l=1}^{\infty}\sum_{k=1}^{\infty}\sqrt{q_lq_k}e_l(x)e_k(x)\frac{\ud }{\ud t}\tilde{\beta}_{k,l}(t),\quad (t,x)\in(0,T]\times(0,1)
\end{align*}
with $U(t,x)=0$ for $(t,x)\in[0,T]\times\{0,1\}$ and $U(0,x)=0$ for $x\in[0,1]$.
\end{theo}
\begin{proof}
It follows from Lemma \ref{verifyAssum'} and Theorem \ref{maintheo} that for any $\eta\in[0,1)$ and $t\in[0,T]$, $m^{\frac{1}{2}}\big(X^m(t,\cdot)-X(t,\cdot)\big)\overset{d}{\Longrightarrow}U(t)$ in $\dot{H}^\eta$ as $m\to\infty$. 

For any given  $x\in(0,1)$, define the mapping $\xi_x: \dot{H}^\eta\to \mbb R$ by $\xi_x(\varphi)=\varphi(x)$ for any $\varphi\in\dot{H}^\eta$ with $\eta>\frac{1}{2}$. Then $\xi_x$ is a continuous mapping due to the Sobolev embedding $\dot{H}^\eta\hookrightarrow \mbf C((0,1);\mbb R)$ for $\eta>\frac{1}{2}$. The continuous mapping theorem and $m^{\frac{1}{2}}\big(X^m(t,\cdot)-X(t,\cdot)\big)\overset{d}{\Longrightarrow}U(t,\cdot)$ in $\dot{H}^\eta$ for $\eta>\frac{1}{2}$ yield that $m^{\frac{1}{2}}\big(X^m(t,x)-X(t,x)\big)\overset{d}{\Longrightarrow}U(t,x)$ in $\mbb R$ for any $(t,x)\in[0,T]\times(0,1)$.
\end{proof}

\section{Asymptotic error of a spatial semi-discrete method}\label{Sec5}

In this section, we turn to studying the asymptotic error of a spatial semi-discrete method--the spectral Galerkin method--applied to \eqref{SPDE}. 
Interestingly, we find that for general SPDEs, it is difficult to identify a nontrivial asymptotic error distribution using this spatial semi-discrete method, which is different from cases for temporal semi-discretizations. We subsequently provide an example to explain the reason.

Applying the spatial spectral Galerkin method to \eqref{SPDE}, we obtain the corresponding finite-dimensional numerical solution $Y^N$, $N\in\mbb N^+$, given by
\begin{align}\label{YN}
	Y^N(t)=E_N(t)P_NX_0+\int_0^t E_N(t-s)P_NF(Y^N(s))\ud s+\int_0^t E_N(t-s)P_NG(Y^N(s))\ud W(s)
\end{align}
for $t\in[0,T]$, where $E_N(t)$ and $P_N$ are defined as in the very beginning of Section \ref{Sec3.2}. 

Similar to the proof of \eqref{Xspatial} and Lemma \ref{Xmregularity}(i), one can establish the spatial regularity of $Y^N$.

\begin{lem}\label{YNspatial}
	Let Assumption \ref{assum1} hold with $\sigma\in[0,1)$. Then  there is a constant $K=K(T)>0$ independent of $N$ such that
$$\sup\limits_{t\in[0,T]}\|Y^N(t)\|_{\mbf L^p(\Omega;\dot{H}^{1+\sigma})}\le K(1+\|X_0\|_{\mbf L^p(\Omega;\dot{H}^{1+\sigma})}).$$
\end{lem}

The convergence order of $Y^N$ is given in the following lemma, whose proof is also postponed to the appendix.

\begin{lem}\label{YNconverge}
	Let Assumption \ref{assum1} hold with $\sigma\in[0,1)$. Then it holds 
	\begin{align*}
		\sup_{t\in[0,T]}\|Y^N(t)-X(t)\|_{\mbf L^p(\Omega;H)}\le K\lambda_{N+1}^{-\frac{1+\sigma}{2}}.
	\end{align*}
\end{lem}

As a direct result of \eqref{Xspatial}, Lemmas \ref{YNspatial}--\ref{YNconverge}, and Proposition \ref{interpolation}, we have the convergence order of $Y^N$ in $\dot{H}^\gamma$ with $\gamma\in[0,1+\sigma)$. 

\begin{cor}\label{cor5.3}
Let Assumption \ref{assum1} hold with $\sigma\in[0,1)$. Then for for any $\gamma\in[0,1+\sigma)$,  
\begin{align*}
	\sup_{t\in[0,T]}\|Y^N(t)-X(t)\|_{\mbf L^p(\Omega;\dot{H}^\gamma)}\le K\lambda_{N+1}^{-\frac{1+\sigma-\gamma}{2}}.
\end{align*}
\end{cor}

Next we give the asymptotic error of $Y^N$ based on the strong convergence rate given in Lemma \ref{YNconverge}.

\begin{theo}\label{YNerrorasym}
Let Assumptions \ref{assum1} and \ref{assum2} hold with $\sigma\in[0,1)$ and $\alpha\in[0,\frac{1+\sigma}{2})$. Then for any $t\in[0,T]$, it holds that $\lim\limits_{N\to\infty}\lambda_{N+1}^{\frac{1+\sigma}{2}}\big(Y^N(t)-X(t)\big)=0$  in $\mbf L^2(\Omega;H)$.
\end{theo}

\begin{proof}
Denote the normalized error process $V^N(t):=\lambda_{N+1}^{\frac{1+\sigma}{2}}\big(Y^N(t)-X(t)\big)$ for $t\in[0,T]$, and decompose it into $V^N=I_1^N+I_2^N+I_3^N$ with
\begin{align*}
I_1^N(t)&:=\lambda_{N+1}^{\frac{1+\sigma}{2}}E(t)(P_N-Id_H)X_0, \\
I_2^N(t)&:=\lambda_{N+1}^{\frac{1+\sigma}{2}}\int_0^t\big(E(t-s)P_NF(Y^N(s))-E(t-s)F(X(s))\big)\ud s, \\
I_3^N(t)&:=\lambda_{N+1}^{\frac{1+\sigma}{2}}\int_0^t\big(E(t-s)P_NG(Y^N(s))-E(t-s)G(X(s))\big)\ud W(s).
\end{align*}
 
Noting that $0<\lambda_1\le\cdots\le\lambda_i\le\cdots$, we have for almost sure (a.s.) $\omega\in\Omega$ that
\begin{align*}
	\sup_{t\in[0,T]}\|I_1^N(t)\|^2&=\lambda_{N+1}^{1+\sigma}\sup_{t\in[0,T]}\|E(t)(P_N-Id_H)(-A)^{-\frac{1+\sigma}{2}}(-A)^{\frac{1+\sigma}{2}}X_0\|^2 \\
	&=\lambda_{N+1}^{1+\sigma}\sup_{t\in[0,T]}\sum_{i=N+1}^\infty e^{-2\lambda_i t}\lambda_i^{-(1+\sigma)}\big\LL (-A)^{\frac{1+\sigma}{2}}X_0,e_i\big\RR^2 \\
	&\le \sum_{i=N+1}^\infty\big\LL (-A)^{\frac{1+\sigma}{2}}X_0,e_i\big\RR^2.
\end{align*} 
 Since $(-A)^{\frac{1+\sigma}{2}}X_0\in H$ for a.s. $\omega\in\Omega$, one has $\lim\limits_{N\to\infty}\sum_{i=N+1}^\infty\big\LL (-A)^{\frac{1+\sigma}{2}}X_0,e_i\big\RR^2=0$. Thus, for a.s. $\omega\in\Omega$, $\lim\limits_{N\to\infty}\sup\limits_{t\in[0,T]}\|I_1^N(t)\|^2=0$. 
 Further, using \eqref{Pn-Id} yields
 \begin{align*}
 	\sup_{t\in[0,T]}\|I_1^N(t)\|^2\le \|(-A)^{\frac{1+\sigma}{2}}X_0\|^2=\|X_0\|_{1+\sigma}^2.
 \end{align*}
It follows from $\|X_0\|_{\mbf L^p(\Omega;\dot{H}^{1+\sigma})}<\infty$ with $p\ge 4$ and the dominated convergence theorem that
\begin{align}
\lim_{N\to\infty}\mbf E\left[\sup_{t\in[0,T]}\|I_1^N(t)\|^2\right]=0, \label{I1N}
\end{align}
which further implies
\begin{align}\label{sec5eq6}
	\sup_{N\ge1}\mbf E\left[\sup_{t\in[0,T]}\|I_1^N(t)\|^2\right]\le K<\infty.
\end{align}

Further, we decompose $I_2^N=I_{2,1}^N+I_{2,2}^N$ with 
\begin{align*}
	I_{2,1}^N(t)&:=\lambda_{N+1}^{\frac{1+\sigma}{2}}\int_0^tE(t-s)P_N\big(F(Y^N(s))-F(X(s))\big)\ud s,\\
	I_{2,2}^N(t)&:=\lambda_{N+1}^{\frac{1+\sigma}{2}}\int_0^tE(t-s)\big(P_N-Id_H\big)F(X(s))\ud s.
\end{align*}
The Taylor theorem yields
\begin{align*}
	I_{2,1}^N(t)=\int_0^tE(t-s)P_N\mcal DF(X(s))V^N(s)\ud s+R_{I_{2,1}^N},
\end{align*}
where
\begin{align*}
R_{I_{2,1}^N}=\lambda_{N+1}^{\frac{1+\sigma}{2}}\int_0^tE(t-s)P_N\int_0^1(1-\lambda)\mcal D^2F(X(s)+\lambda(Y^N(s)-X(s)))\big(Y^N(s)-X(s)\big)^2\ud\lambda\ud s.	
\end{align*}
Applying \eqref{F''} and Corollary \ref{cor5.3} gives
\begin{align*}
\mbf E\|R_{I_{2,1}^N}\|^2\le K\lambda_{N+1}^{1+\sigma}\int_0^t\mbf E\|Y^N(s)-X(s)\|_\alpha^4\ud s\le K\lambda_{N+1}^{-(1+\sigma-2\alpha)}.
\end{align*}
This together with \eqref{F'} leads to
\begin{align*}
	\mbf E\|I_{2,1}^N(t)\|^2\le K\int_0^t\mbf E\|V^N(s)\|^2\ud s+K\lambda_{N+1}^{-(1+\sigma-2\alpha)}.
\end{align*}
It follows from \eqref{semigroup1}, \eqref{Xspatial}, the linear growth property of $F$, and \eqref{Pn-Id} that
\begin{align*}
	&\;\|I_{2,2}^N(t)\|_{\mbf L^2(\Omega;H)}\\
	\le&\; K\lambda_{N+1}^{\frac{1+\sigma}{2}}\int_0^t\|(-A)^{\frac{3+\sigma}{4}}E(t-s)\|_{\mcal L(H)}\|(-A)^{-\frac{3+\sigma}{4}}(P_N-Id_H)\|_{\mcal L(H)}(1+\|X(s)\|_{\mbf L^2(\Omega;H)})\ud s \\
	\le &\; K\lambda_{N+1}^{-\frac{1-\sigma}{4}}\int_0^t(t-s)^{-\frac{3+\sigma}{4}}\ud s \le K\lambda_{N+1}^{-\frac{1-\sigma}{4}}.
\end{align*}
In this way, it holds that for any $t\in[0,T]$, 
\begin{align}\label{I2N}
	\mbf E\|I_2^N(t)\|^2\le K\int_0^t\mbf E\|V^N(s)\|^2\ud s+K\big(\lambda_{N+1}^{-(1+\sigma-2\alpha)}+\lambda_{N+1}^{-\frac{1-\sigma}{2}}\big).
\end{align}

Next, we turn to tackling $I_3^N$, which is decomposed into $I_3^N=\sum_{i=1}^{3}I_{3,i}^N$ with
\begin{align*}
I_{3,1}^N(t)&:=\lambda_{N+1}^{\frac{1+\sigma}{2}}\int_0^tE(t-s)P_N\big(G(Y^N(s))-G(X(s))\big)\ud W(s),\\
I_{3,2}^N(t)&:=\lambda_{N+1}^{\frac{1+\sigma}{2}}\int_0^tE(t-s)\big(P_N-Id_H\big)\big(G(X(s))-G(X(t))\big)\ud W(s),\\
I_{3,3}^N(t)&:=\lambda_{N+1}^{\frac{1+\sigma}{2}}\int_0^tE(t-s)\big(P_N-Id_H\big)G(X(t))\ud W(s).
\end{align*}
By the Taylor theorem,  
\begin{align*}
	I_{3,1}^N(t)=\int_0^tE(t-s)P_N\mcal DG(X(s))V^N(s)\ud W(s)+R_{I_{3,1}^N},
\end{align*}
where
\begin{align*}
	R_{I_{3,1}^N}=\lambda_{N+1}^{\frac{1+\sigma}{2}}\int_0^tE(t-s)P_N\int_0^1(1-\lambda)\mcal D^2G(X(s)+\lambda(Y^N(s)-X(s)))\big(Y^N(s)-X(s)\big)^2\ud\lambda\ud W(s).	
\end{align*}
Using the It\^o isometry, \eqref{G'}, \eqref{G''}, and Corollary \ref{cor5.3}, we derive
\begin{align}
\mbf E\|I_{3,1}^N(t)\|^2&\le K\int_0^t \mbf E\|V^N(s)\|^2\ud s+K\lambda_{N+1}^{1+\sigma}\int_0^t\mbf E\|Y^N(s)-X(s)\|_\alpha^4\ud s\nonumber\\
&\le  K\int_0^t \mbf E\|V^N(s)\|^2\ud s+K\lambda_{N+1}^{-(1+\sigma-2\alpha)}. \label{sec5eq4}
\end{align}
Applying the It\^o isometry, \eqref{semigroup1}, \eqref{GLip}, \eqref{Xtemporal}, and \eqref{Pn-Id} yields
\begin{align}
\mbf E\|I_{3,2}^N(t)\|^2&=\lambda_{N+1}^{1+\sigma}\mbf E\int_0^t\|E(t-s)\big(P_N-Id_H\big)\big(G(X(s))-G(X(t))\big)\|_{\mcal L_2^0}^2\ud s \notag\\
&\le K	\lambda_{N+1}^{1+\sigma}\int_0^t\|(-A)^{\frac{3+\sigma}{4}}E(t-s)\|^2_{\mcal L(H)}\|(-A)^{-\frac{3+\sigma}{4}}\big(P_N-Id_H\big)\|_{\mcal L(H)}^2\mbf E\|X(t)-X(s)\|^2\ud s\notag\\
&\le K	\lambda_{N+1}^{-\frac{1-\sigma}{2}}\int_0^t(t-s)^{-\frac{1+\sigma}{2}}\ud s\le  K	\lambda_{N+1}^{-\frac{1-\sigma}{2}}. \label{sec5eq5}
\end{align}
For $I_{3,3}^N$, we deduce from the It\^o isometry and \eqref{semigroup3} that
\begin{align}
\mbf E\|I_{3,3}^N(t)\|^2&=\lambda_{N+1}^{1+\sigma}\mbf E\int_0^t\|E(t-s)\big(P_N-Id_H\big)G(X(t))\|_{\mcal L_2^0}^2\ud s \nonumber\\
&=\lambda_{N+1}^{1+\sigma}\mbf E\sum_{i=1}^{\infty}\int_0^t\|(-A)^{\frac{1}{2}}E(t-s)(-A)^{-\frac{1}{2}}(P_N-Id_H)G(X(t))Q^{\frac{1}{2}}h_i\|^2\ud s\nonumber\\
&\le K\lambda_{N+1}^{1+\sigma}\mbf E\sum_{i=1}^{\infty}\|(-A)^{-\frac{1}{2}}(P_N-Id_H)G(X(t))Q^{\frac{1}{2}}h_i\|^2 \label{sec5eq1}\\
&= K\lambda_{N+1}^{1+\sigma}\mbf E\sum_{i=1}^{\infty}\sum_{j=1}^{\infty}\big\LL (-A)^{-\frac{1}{2}}(P_N-Id_H)G(X(t))Q^{\frac{1}{2}}h_i,e_j\big\RR^2.\nonumber
\end{align} 
Since $(-A)^{-\gamma}$, $\gamma\ge 0$, and $P_N$ are self-disjoint from $H$ to itself, we have
\begin{align*}
\mbf E\|I_{3,3}^N(t)\|^2&=K\lambda_{N+1}^{1+\sigma}\mbf E\sum_{i=1}^{\infty}\sum_{j=1}^{\infty}\big\LL G(X(t))Q^{\frac{1}{2}}h_i,(-A)^{-\frac{1}{2}}(P_N-Id_H)e_j\big\RR^2 \\
&=K\lambda_{N+1}^{1+\sigma}\mbf E\sum_{i=1}^{\infty}\sum_{j=N+1}^{\infty}\big\LL (-A)^{\frac{\sigma}{2}}G(X(t))Q^{\frac{1}{2}}h_i,(-A)^{-\frac{1+\sigma}{2}}e_j\big\RR^2 \\
&=K\lambda_{N+1}^{1+\sigma}\mbf E\sum_{i=1}^{\infty}\sum_{j=N+1}^{\infty}\lambda_j^{-(1+\sigma)}\big\LL (-A)^{\frac{\sigma}{2}}G(X(t))Q^{\frac{1}{2}}h_i,e_j\big\RR^2\\
&=K\lambda_{N+1}^{1+\sigma}\mbf E\sum_{i=1}^{\infty}\sum_{j=N+1}^{\infty}\lambda_j^{-(1+\sigma)}\big\LL Q^{\frac{1}{2}}h_i,\big((-A)^{\frac{\sigma}{2}}G(X(t))\big)^*e_j\big\RR_{U_0}^2.
\end{align*}
Here, $\big((-A)^{\frac{\sigma}{2}}G(X(t))\big)^*$ denotes the disjoint operator of $(-A)^{\frac{\sigma}{2}}G(X(t))$. 
Noting that $\{Q^{\frac{1}{2}}h_i\}_{i\in\mbb N^+}$ is a complete orthonormal basis of $U_0$, we arrive at
\begin{align*}
\mbf E\|I_{3,3}^N(t)\|^2&=K\lambda_{N+1}^{1+\sigma}\mbf E\sum_{j=N+1}^\infty\lambda_j^{-(1+\sigma)}\big\|\big((-A)^{\frac{\sigma}{2}}G(X(t))\big)^*e_j\big\|_{U_0}^2 \\
&\le K\mbf E\left[\mcal G^N\right],
\end{align*}
where $\mcal G^N:=\sum_{j=N+1}^\infty\big\|\big((-A)^{\frac{\sigma}{2}}G(X(t))\big)^*e_j\big\|_{U_0}^2$. 
By  the fact $\|\Gamma^*\|_{\mcal L_2(H,U_0)}=\|\Gamma\|_{\mcal L_2^0}$ for any $\Gamma\in\mcal L_2^0$, $\mcal G^N\le \big\|\big((-A)^{\frac{\sigma}{2}}G(X(t))\big)^*\big\|^2_{\mcal L_2(H,U_0)}=\|(-A)^{\frac{\sigma}{2}}G(X(t))\|_{\mcal L_2^0}^2$. Moreover, using \eqref{Ggrow} and \eqref{Xspatial} gives
\begin{align*}
\mbf E\|(-A)^{\frac{\sigma}{2}}G(X(t))\|_{\mcal L_2^0}^2 \le K(1+\mbf E\|X(t)\|^2_\sigma)\le K(T).
\end{align*}
Additionally, $\sum_{j=1}^\infty\big\|\big((-A)^{\frac{\sigma}{2}}G(X(t))\big)^*e_j\big\|_{U_0}^2<\infty$ due to $\big((-A)^{\frac{\sigma}{2}}G(X(t))\big)^*\in\mcal L_2(H,U_0)$ for a.s. $\omega\in\Omega$, which indicates 
that  $\lim\limits_{N\to\infty}\mcal G^N=0$ for a.s. $\omega\in\Omega$. In this way, we can apply the dominated convergence theorem to deduce that for any $t\in[0,T]$, 
\begin{align}\label{sec5eq2}
	\lim_{N\to\infty}\mbf E\|I_{3,3}^N(t)\|^2\le K \lim_{N\to\infty}\mbf E\left[\mcal G^N\right]=0. 
\end{align}
In addition, it is easy to show on basis of \eqref{sec5eq1} that
\begin{align}\label{sec5eq3}
	\sup_{t\in[0,T]}\sup_{N\ge1}\mbf E\|I_{3,3}^N(t)\|^2\le K(T).
\end{align}

Combining  $I_3^N=\sum_{i=1}^{3}I_{3,i}^N$, \eqref{I2N}, \eqref{sec5eq4},  and \eqref{sec5eq5}, we have 
\begin{align*}
	\mbf E\|V^N(t)\|^2&\le K\mbf E\|I_1^N(t)\|^2+K\mbf E\|I_2^N(t)\|^2+K\sum_{i=1}^{3}\mbf E\|I_{3,i}^N(t)\|^2\\
	&\le K_1\int_0^t\mbf E\|V^N(s)\|^2\ud s+K_1\big(\lambda_{N+1}^{-(1+\sigma-2\alpha)}+\lambda_{N+1}^{-\frac{1-\sigma}{2}}\big)+K_1\mbf E\sup_{t\in[0,T]}\|I_1^N(t)\|^2+K_1\mbf E\|I_{3,3}^N(t)\|^2
\end{align*}
for some $K_1>0$. 
Then the Gronwall inequality yields
\begin{align}
	\mbf E\|V^N(t)\|^2 &\le a^N(t)+K_1\int_0^ta^N(s)e^{K_1(t-s)}\ud s \notag\\
	&\le a^N(t)+K_1e^{K_1T}\int_0^t a^N(s)\ud s,~t\in[0,T],\label{sec5eq9}
\end{align}
where $a^N(t):=K_1\Big(\lambda_{N+1}^{-(1+\sigma-2\alpha)}+\lambda_{N+1}^{-\frac{1-\sigma}{2}}+\mbf E\sup\limits_{t\in[0,T]}\|I_1^N(t)\|^2+\mbf E\|I_{3,3}^N(t)\|^2\Big)$. 
It follows from \eqref{I1N}, \eqref{sec5eq6}, \eqref{sec5eq2}, and \eqref{sec5eq3} that
\begin{align}
&\;\lim_{N\to\infty}a^N(t)=0,\quad\forall~t\in[0,T], \label{sec5eq7}\\
&\; \sup_{s\in[0,T]}\sup_{N\ge1}a^N(s)\le K(T)(\lambda_1^{-(1+\sigma-2\alpha)}+\lambda_{1}^{-\frac{1-\sigma}{2}}+1).\label{sec5eq8}
\end{align}
Based on \eqref{sec5eq7} and \eqref{sec5eq8}, the dominated convergence theorem gives 
\begin{align*}
	\lim_{N\to\infty}\int_0^ta^N(s)\ud s=0,
\end{align*}
which, together with \eqref{sec5eq9}, leads to
\begin{align*}
	\lim_{N\to\infty}\mbf E\|V^N(t)\|^2=0,\quad\forall~t\in[0,T]
\end{align*}
and completes the proof.
\end{proof}

Theorem \ref{YNerrorasym} indicates that the established strong convergence speed $\lambda_{N+1}^{-\frac{1+\sigma}{2}}$ of $Y^N$   is smaller than the exact one. We further demonstrate this via a heuristic example. 

\begin{ex}\label{Example}
Let $H=\mbf L^2([0,1];\mbb R)$ and $A$ be the Laplacian with the homogeneous Dirichlet boundary condition such that the eigenvalues and eigenfunctions of $-A$ admit forms
\begin{align*}
	\lambda_n=n^2\pi^2,\quad e_n(\xi)=\sqrt{2}\sin(n\pi\xi),\quad \xi\in[0,1]
\end{align*}
for all $n\ge 1$.
We set $x\in \dot{H}^1$ and consider the error between  $x$ and its spectral Galerkin approximation $P^Nx$, which is the error  $X(t)-Y^N(t)$ at $t=0$ provided $X_0=x$.  Generally,  since we do not know the exact function form of $x$, the error $\|P^Nx-x\|$ would be estimated	as
\begin{align*} 
	\|P^Nx-x\|\le \|(-A)^{-\frac{1}{2}}\big(P_N-Id_H\big)\|_{\mcal L(H)}\|x\|_1=\lambda_{N+1}^{-\frac{1}{2}}\|x\|.
\end{align*}
Then one infers that the ``optimal convergence order" of $\|P^Nx-x\|$ is one with respect to the spatial dimension $N$ since $\lambda_{N+1}^{-\frac{1}{2}}=\mcal O(N^{-1})$. Here, the convergence order one is optimal in the sense that it coincides with the spatial regularity of $x$.
 However, one indeed can show
\begin{align}
	\|P_Nx-x\|^2=\sum_{i=N+1}^{\infty}\LL x,e_i\RR^2 =\sum_{i=N+1}^{\infty}\lambda_i^{-1}\LL (-A)^{\frac{1}{2}}x,e_i\RR^2 
	\le \lambda_{N+1}^{-1}\sum_{i=N+1}^{\infty}\LL (-A)^{\frac{1}{2}}x,e_i\RR^2. \label{sec5eq10}
\end{align}
Thus, 
\begin{align*}
	\lambda_{N+1}^{\frac{1}{2}}\|P_Nx-x\|\le \Big(\sum_{i=N+1}^{\infty}\LL (-A)^{\frac{1}{2}}x,e_i\RR^2\Big)^{\frac{1}{2}}\to 0,\quad N\to\infty
\end{align*}
due to the fact $x\in\dot H^1$. It implies that the convergence speed of $\|P_Nx-x\|$ is indeed larger than $\lambda_{N+1}^{-\frac{1}{2}}.$ 

If we take $x=\sum_{n=2}^{\infty}\frac{1}{n^{\frac{3}{2}}(\ln n)^\gamma}e_n$ with $\gamma>\frac{1}{2}$, then for any $r\ge 0$,
\begin{align*}
	\|x\|_{1+r}^2=\sum_{n=2}^{\infty}\lambda_{n}^{1+r}\frac{1}{n^3(\ln n)^{2\gamma}}=\pi^{2r+2}\sum_{n=2}^{\infty}\frac{n^{2r}}{n(\ln n)^{2\gamma}}.
\end{align*}
This series converges if and only if $r=0$, which means that $x\in \dot{H}^1$ and $x\notin \dot{H}^{1+r}$ for any $r>0$. Then, using \eqref{sec5eq10} yields
\begin{align}\label{sec5eq11}
	\|P_Nx-x\|&\le \lambda_{N+1}^{-\frac{1}{2}}\Big(\sum_{i=N+1}^{\infty}\frac{\lambda_i}{i^3(\ln i)^{2\gamma}}\Big)^{\frac{1}{2}}=\pi\lambda_{N+1}^{-\frac{1}{2}}\Big(\sum_{i=N+1}^{\infty}\frac{1}{i(\ln i)^{2\gamma}}\Big)^{\frac{1}{2}} \notag\\
	&\le \pi\lambda_{N+1}^{-\frac{1}{2}}\Big(\int_N^{\infty}\frac{1}{x(\ln x)^{2\gamma}}\ud x\Big)^{\frac{1}{2}}\le \frac{\pi}{\sqrt{2\gamma-1}}\frac{1}{(\ln N)^{\frac{2\gamma-1}{2}}}\lambda_{N+1}^{-\frac{1}{2}},\quad\forall~\gamma>\frac{1}{2}.
\end{align}
Although the infinitesimal factor $\frac{1}{(\ln N)^{\frac{2\gamma-1}{2}}}$ is negligible compared with $\lambda_{N+1}^{-\frac{1}{2}}$, the converge speed of $\|P_Nx-x\|$ is definitely faster than  $\lambda_{N+1}^{-\frac{1}{2}}$. In fact, it is easy to show
\begin{align*}
	\lim_{N\to\infty}\lambda_{N+1}^{\frac{1}{2}}\|P_Nx-x\|=0,~\lim_{N\to\infty}\lambda_{N+1}^{\frac{1}{2}+r}\|P_Nx-x\|=\infty,\quad\forall~r>0.
\end{align*}
In addition, \eqref{sec5eq11} indicates that the exact convergence speed of  $\|P_Nx-x\|$ is problem-dependent due to the arbitrariness of $\gamma$. In other words, the asymptotic error distribution of $P_Nx$ is also problem-dependent.
\end{ex}

As is shown in Theorem \ref{YNerrorasym} and Example \ref{Example}, it seems that one can not obtain a sharp convergence  speed for the spatial spectral Galerkin method applied to a general SPDE. It is also interesting to study whether there is a spatial semi-discrete numerical method admitting a nontrivial limit distribution.

\section{Concluding remarks}\label{Sec6}
In this study, we investigate the asymptotic error distribution of the exponential Euler method when applied to parabolic SPDEs with multiplicative noise. Notably, the limit equation, in terms of distribution, is influenced by an infinite number of additional independent $Q$-Wiener processes.
Building on this finding, we further explore the asymptotic error distribution of a fully discrete method that employs the exponential Euler method for temporal discretization and the finite element method for spatial discretization. To illustrate our results, we provide a concrete example involving a class of stochastic heat equations, demonstrating the pointwise convergence in distribution of the normalized error process associated with the exponential Euler method.
Ultimately, for spatial semi-discretizations, we investigate the asymptotic error distributions of the spectral Galerkin method in Section \ref{Sec5}, which suggests that the asymptotic error distribution of spatially semi-discrete numerical methods for SPDEs may vary on a case-by-case basis. It raises the intriguing question of whether there exists a spatially semi-discrete numerical method, such as the finite element method, that admits a nontrivial limit distribution. We leave this as an open problem for future research.

\bibliographystyle{plain}
\bibliography{mybibfile}

\appendix
\section{Proofs of Lemma \ref{Xmregularity}, Theorem \ref{strongorder}, and Lemma \ref{YNconverge}}

\subsection{Proof of Lemma \ref{Xmregularity}}
	It follows from  $\|E(t)\|_{\mcal L(H)}\le 1$, Burkholder--Davis--Gundy (BDG) inequality,  and the linear growth property of $F$ and $G$ that
	\begin{align*}
		&\;\|X^m(t)\|_{\mbf L^p(\Omega;H)} \\
		\le&\; \|X_0\|_{\mbf L^p(\Omega;H)}+K\int_0^t(1+\|X^m(\kappa_m(s))\|_{\mbf L^p(\Omega;H)})\ud s+K\Big\|\Big(\int_0^t\|G(X^m(\kappa_m(s)))\|_{\mcal L_2^0}^2\ud s\Big)^{1/2}\Big\|_{\mbf L^p(\Omega;\mbb R)} \\
		\le &\; \|X_0\|_{\mbf L^p(\Omega;H)}+K\int_0^t\big(1+\sup_{r\in[0,s]}\|X^m(r)\|_{\mbf L^p(\Omega;H)}\big)\ud s+K\Big[\int_0^t\big(1+\sup_{r\in[0,s]}\|X^m(r)\|^2_{\mbf L^p(\Omega;H)}\big)\ud r\Big]^{\frac{1}{2}}.
	\end{align*}
Thus, we have
\begin{align*}
	\sup_{r\in[0,t]}\|X^m(r)\|_{\mbf L^p(\Omega;H)}^2\le K(T)(1+\|X_0\|_{\mbf L^p(\Omega;H)}^2)+K(T)\int_0^t\sup_{r\in[0,s]}\|X^m(r)\|_{\mbf L^p(\Omega;H)}^2\ud s,
\end{align*}	
which implies
\begin{align}\label{Appendix1}
	\sup_{t\in[0,T]}\|X^m(t)\|_{\mbf L^p(\Omega;H)}\le K(T)(1+\|X_0\|_{\mbf L^p(\Omega;H)})\le K(T,\sigma)(1+\|X_0\|_{\mbf L^p(\Omega;\dot{H}^{1+\sigma})})
\end{align}
due to the Gronwall inequality.

Using the BDG inequality, \eqref{semigroup1},  and \eqref{Appendix1} yields that for any $\beta\in[0,1)$,
\begin{align}
&\;	\|X^m(t)\|_{\mbf L^p(\Omega;\dot{H}^\beta)} \nonumber\\
	\le &\;\|X_0\|_{\mbf L^p(\Omega;\dot{H}^\beta)}+\int_0^t\|(-A)^{\frac{\beta}{2}}E(t-\kappa_m(s))F(X^m(\kappa_m(s)))\|_{\mbf L^p(\Omega;H)}\ud s \nonumber \\
	&\;+K\Big\|\Big(\int_0^t\|(-A)^{\frac{\beta}{2}}E(t-\kappa_m(s))G(X^m(\kappa_m(s)))\|^2_{\mcal L_2^0}\ud s\Big)^{1/2}\Big\|_{\mbf L^p(\Omega;\mbb R)} \nonumber\\
	\le&\; K\|X_0\|_{\mbf L^p(\Omega;\dot{H}^{1+\sigma})}+K\int_0^t(t-\kappa_m(s))^{-\frac{\beta}{2}}(1+\|X^m(\kappa_m(s))\|_{\mbf L^p(\Omega;H)})\ud s \nonumber\\
	&\;+K\Big(\int_0^t(t-\kappa_m(s))^{-\beta}(1+\|X^m(\kappa_m(s))\|^2_{\mbf L^p(\Omega;H)})\ud s\Big)^{1/2} \nonumber\\
	\le &\; K(1+\|X_0\|_{\mbf L^p(\Omega;\dot{H}^{1+\sigma})})\Big(1+\int_0^t(t-s)^{-\frac{\beta}{2}}\ud s+\Big(\int_0^t(t-s)^{-\beta}\ud s\Big)^{1/2}\Big) \nonumber\\
	\le &\; K(1+\|X_0\|_{\mbf L^p(\Omega;\dot{H}^{1+\sigma})}).  \label{Appendix2}
\end{align} 

Note that for $t\ge s$, $X^m(t)-X^m(s)=(E(t-s)-Id_H)X^m(s)+\int_s^tE(t-\kappa_m(r))F(X^m(\kappa_m(r)))\ud r+\int_s^tE(t-\kappa_m(r))G(X^m(\kappa_m(r)))\ud W(r)$. It follows from \eqref{semigroup2}, $\|E(t)\|_{\mcal L(H)}\le 1$, the BDG inequality, and \eqref{Appendix2} that for any $\delta\in(0,\frac{1}{2})$,
\begin{align}
&\;\|X^m(t)-X^m(s)\|_{\mbf L^p(\Omega;H)}\notag\\
\le&\; K(t-s)^{\delta}\|X^m(s)\|_{\mbf L^p(\Omega;\dot{H}^{2\delta})}+K\int_s^t(1+\|X^m(\kappa_m(r))\|_{\mbf L^p(\Omega;H)})\ud r \notag\\
&\;+K\Big\|\Big(\int_s^t\|G(X^m(\kappa_m(r)))\|_{\mcal L_2^0}^2\ud r\Big)^{1/2}\Big\|_{\mbf L^p(\Omega;\mbb R)}\notag \\
\le &\; K(t-s)^\delta+K(t-s)+K\Big(\int_s^t(1+\|X^m(\kappa_m(r))\|_{\mbf L^p(\Omega;H)}^2)\ud r\Big)^{1/2} \notag\\
\le &\; K(t-s)^\delta.\label{Appendix3}
\end{align}

Applying the BDG inequality, \eqref{semigroup1}, \eqref{semigroup3},  \eqref{GLip}, \eqref{Ggrow}, \eqref{Appendix2}, and \eqref{Appendix3} with $\delta=\frac{1+\sigma}{4}$, we obtain
\begin{align}
	&\;\|X^m(t)\|_{\mbf L^p(\Omega;\dot{H}^{1+\sigma})}\nonumber\\
	\le&\; \|X_0\|_{\mbf L^p(\Omega;\dot{H}^{1+\sigma})}+K\int_0^t(t-\kappa_m(s))^{-\frac{1+\sigma}{2}}(1+\|X^m(\kappa_m(s))\|_{\mbf L^p(\Omega;H)})\ud s \notag\\
	&\;+K\Big(\int_0^t\|(-A)^{\frac{1+\sigma}{2}}E(t-\kappa_m(s))\big(G(X^m(\kappa_m(s)))-G(X^m(t))\big)\|_{\mbf L^p(\Omega;\mcal L_2^0)}^2\ud s\Big)^{1/2} \notag \\
	&\;+K\Big\|\Big(\int_0^t\|(-A)^{\frac{1+\sigma}{2}}E(t-s)G(X^m(t))\|_{\mcal L_2^0}^2\|E(s-\kappa_m(s))\|_{\mcal L(H)}^2\ud s\Big)^{1/2} \Big\|_{\mbf L^p(\Omega;\mbb R)}\notag \\
	\le&\; K(1+\|X_0\|_{\mbf L^p(\Omega;\dot{H}^{1+\sigma})})+K\int_0^t(t-\kappa_m(s))^{-\frac{1+\sigma}{2}}\ud s \notag \\
	&\;+K\Big\|\Big(\sum_{i=1}^\infty\int_0^t\|(-A)^{\frac{1}{2}}E(t-s)(-A)^{\frac{\sigma}{2}}G(X^m(t))Q^{\frac{1}{2}}h_i\|^2\ud s\Big)^{1/2}\Big\|_{\mbf L^p(\Omega;\mbb R)} \notag\\
	\le & K(1+ \|X_0\|_{\mbf L^p(\Omega;\dot{H}^{1+\sigma})})+K\big\|\|(-A)^{\frac{\sigma}{2}}G(X^m(t))\|_{\mcal L_2^0}\big\|_{\mbf L^p(\Omega;\mbb R)} \notag\\
	\le &\;K(1+ \|X_0\|_{\mbf L^p(\Omega;\dot{H}^{1+\sigma})}). \label{Appendix4}
\end{align} 
This proves Lemma \ref{Xmregularity}(i).

Next, we prove the second conclusion. 
For the case $\gamma\in[0,\sigma]$, applying the BDG inequality, \eqref{semigroup1}-\eqref{semigroup2}, \eqref{Ggrow}, $\|(-A)^{-\rho}\|_{\mcal L(H)}\le K(\rho)$ for $\rho\ge 0$, and \eqref{Appendix4}, one has that for any $0\le s<t\le T$,
\begin{align}
	&\;\|X^m(t)-X^m(s)\|_{\mbf L^p(\Omega;\dot{H}^\gamma)}\notag\\
	\le&\; K(t-s)^{\frac{1+\sigma-\gamma}{2}}\|X^m(s)\|_{\mbf L^p(\Omega;\dot{H}^{1+\sigma})}+K\int_s^t(t-\kappa_m(r))^{-\frac{\gamma}{2}}\ud r \notag\\
	&\;+ K\Big(\int_s^t\|(-A)^{-\frac{\sigma-\gamma}{2}}E(t-\kappa_m(r))\|^2_{\mcal L(H)}\|(-A)^{\frac{\sigma}{2}}G(X^m(\kappa_m(r)))\|_{\mbf L^p(\Omega;\mcal L_2^0)}^2\ud s\Big)^{1/2} \notag\\
	\le &\; K(t-s)^{1/2}.\label{Appendix5}
\end{align}

For $\gamma\in(\sigma,1+\sigma)$, it follows from \eqref{Appendix4}-\eqref{Appendix5},  Proposition \ref{interpolation}, and the H\"older inequality that
 \begin{align*}
 	\|X^m(t)-X^m(s)\|_{\mbf L^p(\Omega;\dot{H}^\gamma)}&\le \Big[\mbf E\big(\|X^m(t)-X^m(s)\|_\sigma^{p(1+\sigma-\gamma)}\|X^m(t)-X^m(s)\|_{1+\sigma}^{p(\gamma-\sigma)}\big)\Big]^{1/p} \\
 	&\le \big\|X^m(t)-X^m(s)\big\|_{\mbf L^p(\Omega;\dot{H}^\sigma)} ^{1+\sigma-\gamma}	\big\|X^m(t)-X^m(s)\big\|_{\mbf L^p(\Omega;\dot{H}^{1+\sigma})} ^{\gamma-\sigma}	\\
 	&\le K(\gamma,T)|t-s|^{\frac{1+\sigma-\gamma}{2}}.	
 \end{align*} 
The above formula and \eqref{Appendix5} finish the proof.

\subsection{Proof of Theorem \ref{strongorder}}
Fix $\beta\in[0,\sigma]$.	By \eqref{mildsoution} and \eqref{Xmmild}, we have $X^m(t)-X(t)=\sum_{i=1}^6S_i^m(t)$, $t\in[0,T]$, where
	\begin{align*}
	S_1^m(t)&:=\int_0^tE(t-\kappa_m(s))\big(F(X^m(\kappa_m(s)))-F(X^m(s))\big)\ud s,\\
	S_2^m(t)&:=\int_0^tE(t-s)\big(E(s-\kappa_m(s))-Id_H)\big)F(X^m(s))\ud s,\\
	S_3^m(t)&:=\int_0^tE(t-s)\big(F(X^m(s))-F(X(s))\big)\ud s,\\
	S_4^m(t)&:=\int_0^tE(t-\kappa_m(s))\big(G(X^m(\kappa_m(s)))-G(X^m(s))\big)\ud W(s),\\
	S_5^m(t)&:=\int_0^tE(t-s)\big(E(s-\kappa_m(s))-Id_H\big)G(X^m(s))\ud W(s), \\
	S_6^m(t)&:=\int_0^t E(t-s)\big(G(X^m(s))-G(X(s))\big)\ud W(s). 
	\end{align*}
	
It follows from \eqref{semigroup1}, \eqref{FLip}, Lemma \ref{Xmregularity}(ii), and $\beta\le\sigma<1$ that
\begin{align*}
	\|S_1^m(t)\|_{\mbf L^p(\Omega;\dot{H}^\beta)}&\le K(\beta)\int_0^t (t-\kappa_m(s))^{-\frac{\beta}{2}}\|X^m(\kappa_m(s))-X^m(s)\|_{\mbf L^p(\Omega;H)}\ud s \\
	&\le K(\beta,T)m^{-\frac{1}{2}}\int_0^t(t-s)^{-\frac{\beta}{2}}\ud s\le Km^{-\frac{1}{2}}.
\end{align*}	
By \eqref{semigroup1}-\eqref{semigroup2}, the linear growth property of $F$, and Lemma \ref{Xmregularity}(i), 
\begin{align*}
		&\;\|S_2^m(t)\|_{\mbf L^p(\Omega;\dot{H}^\beta)}\\
		\le&\; K\int_0^t\|(-A)^{\frac{\beta+1}{2}}E(t-s)\|_{\mcal L(H)}\|(-A)^{-\frac{1}{2}}\big(E(s-\kappa_m(s))-Id_H\big)\|_{\mcal L(H)}(1+\|X^m(s)\|_{\mbf L^p(\Omega;H)})\ud s \\
		\le &\;K(\beta,T)m^{-\frac{1}{2}}\int_0^t(t-s)^{-\frac{\beta+1}{2}}\ud s\le K(\beta,T)m^{-\frac{1}{2}}.	
\end{align*}
Further, using \eqref{semigroup1} and \eqref{FLip}, we arrive at
\begin{align*}
\|S_3^m(t)\|_{\mbf L^p(\Omega;\dot{H}^\beta)}&\le K\int_0^t(t-s)^{-\frac{\beta}{2}}\|X^m(s)-X(s)\|_{\mbf L^p(\Omega;H)}\ud s\\
&\le K(\beta)\int_0^t(t-s)^{-\frac{\beta}{2}}\|X^m(s)-X(s)\|_{\mbf L^p(\Omega;\dot{H}^\beta)}\ud s.
\end{align*}

Applying the BDG inequality, \eqref{semigroup1}, \eqref{GLip}, and Lemma \ref{Xmregularity}(ii) yields
\begin{align*}
	\|S_4^m(t)\|_{\mbf L^p(\Omega;\dot{H}^\beta)}&\le K\Big\|\Big(\int_0^t\|(-A)^{\frac{\beta}{2}}E(t-\kappa_m(s))\big(G(X^m(\kappa_m(s)))-G(X^m(s))\big)\|_{\mcal L_2^0}^2\ud s\Big)^{1/2}\Big\|_{\mbf L^p(\Omega;\mbb R)} \\
	&\le K(\beta)\Big(\int_0^t(t-s)^{-\beta}\|X^m(\kappa_m(s))-X^m(s)\|^2_{\mbf L^p(\Omega;H)}\ud s\Big)^{1/2}\le K(\beta,T)m^{-\frac{1}{2}}. 
\end{align*}
Similarly, by the BDG inequality, \eqref{semigroup1} and \eqref{GLip}, we obtain
\begin{align*}
		\|S_6^m(t)\|_{\mbf L^p(\Omega;\dot{H}^\beta)}\le K(\beta)\Big(\int_0^t(t-s)^{-\beta}\|X^m(s)-X(s)\|^2_{\mbf L^p(\Omega;\dot{H}^\beta)}\ud s\Big)^{1/2}. 
\end{align*}
Next, let us estimate the $p$th moment of $S_5^m(t)$. It follows from the BDG inequality, \eqref{semigroup1}-\eqref{semigroup2}, and \eqref{GLip} that
\begin{align*}
	&\;\|S_5^m(t)\|_{\mbf L^p(\Omega;\dot{H}^\beta)}\\
	\le&\;K\Big\|\Big(\int_0^t\|(-A)^{\frac{\beta}{2}}E(t-s)\big(E(s-\kappa_m(s))-Id_H\big)G(X^m(s))\|_{\mcal L_2^0}^2\ud s\Big)^{1/2}\Big\|_{\mbf L^p(\Omega;\mbb R)}\\
	\le &\;K \Big\|\Big(\int_0^t\|(-A)^{\frac{\beta+1}{2}}E(t-s)\|^2_{\mcal L(H)}\|(-A)^{-\frac{1}{2}}\big(E(s-\kappa_m(s))-Id_H\big)\|^2_{\mcal L(H)}\\
	&\qquad\qquad\cdot\|G(X^m(s))-G(X^m(t))\|_{\mcal L_2^0}^2\ud s\Big)^{1/2}\Big\|_{\mbf L^p(\Omega;\mbb R)} \\
	&\;+K\Big\|\Big(\int_0^t\|(-A)^{\frac{\beta+1-\sigma}{2}}E(t-s)(-A)^{\frac{\sigma}{2}}G(X^m(t))\|^2_{\mcal L_2^0}\|(-A)^{-\frac{1}{2}}\big(E(s-\kappa_m(s))-Id_H\big)\|^2_{\mcal L(H)}\ud s\Big)^{1/2}\Big\|_{\mbf L^p(\Omega;\mbb R)} \\
	\le &\; K(\beta)m^{-\frac{1}{2}}\Big\|\Big(\int_0^t(t-s)^{-(\beta+1)}\|X^m(t)-X^m(s)\|^2\ud s\Big)^{1/2}\Big\|_{\mbf L^p(\Omega;\mbb R)} \\
	&\;+Km^{-\frac{1}{2}}\Big\|\Big(\sum_{i=1}^\infty\int_0^t\|(-A)^{\frac{\beta+1-\sigma}{2}}E(t-s)(-A)^{\frac{\sigma}{2}}G(X^m(t))Q^{\frac{1}{2}}h_i\|^2\ud s\Big)^{1/2}\Big\|_{\mbf L^p(\Omega;\mbb R)}.
\end{align*}
Further, using \eqref{semigroup3}, \eqref{Ggrow}, and Lemma \ref{Xmregularity}, we have
\begin{align*}
	\|S_5^m(t)\|_{\mbf L^p(\Omega;\dot{H}^\beta)}\le&\; K(\beta)m^{-\frac{1}{2}}\Big(\int_0^t(t-s)^{-(\beta+1)}\|X^m(t)-X^m(s)\|^2_{\mbf L^p(\Omega;H)}\ud s\Big)^{1/2} \\
	&\;+K(\beta)m^{-\frac{1}{2}}\|(-A)^{\frac{\sigma}{2}}G(X^m(t))\|_{\mbf L^p(\Omega;\mcal L_2^0)} \\
	\le &\; K(\beta,T)m^{-\frac{1}{2}}\Big[\big(\int_0^t(t-s)^{-\beta}\ud s\big)^{1/2}+1+\|X^m(t)\|_{\mbf L^p(\Omega;\dot{H}^\sigma)}\Big] \\
	\le &\; K(\beta,T)m^{-\frac{1}{2}}.
\end{align*}

Combining the previous estimates for $S_i^m(t)$, $i=1,\ldots,6$, and using the H\"older inequality, one has
\begin{align*}
	&\;\|X^m(t)-X(t)\|^2_{\mbf L^p(\Omega;\dot{H}^\beta)}\\
	\le&\; K(\beta,T)m^{-1}+K(\beta,T)\int_0^t(t-s)^{-\beta}\|X^m(s)-X(s)\|^2_{\mbf L^p(\Omega;\dot{H}^\beta)}\ud s,
\end{align*}
which gives $\sup\limits_{t\in[0,T]}\|X^m(t)-X(t)\|_{\mbf L^p(\Omega;\dot{H}^\beta)}\le K(\beta,T)m^{-\frac{1}{2}}$ due to the Gronwall inequality with singular kernel. This completes the proof.

\subsection{Proof of Lemma \ref{YNconverge}}
By \eqref{mildsoution} and \eqref{YN} , 
\begin{align*}
	&\;\|Y^N(t)-X(t)\|_{\mbf L^p(\Omega;H)}\\
	\le&\; \|E(t)(P_N-Id_H)X_0\|_{\mbf L^p(\Omega;H)} \\
	&\;+\Big\|\int_0^t\big(E_N(t-s)P_NF(Y^N(s))-E(t-s)F(X(s))\big)\ud s\Big\|_{\mbf L^p(\Omega;H)}\\
	&\;+\Big\|\int_0^t\big(E_N(t-s)P_NG(Y^N(s))-E(t-s)G(X(s))\big)\ud W(s)\Big\|_{\mbf L^p(\Omega;H)}\\
	=:&\; D_1+D_2+D_3.
\end{align*}
By \eqref{Pn-Id} and $\|E(t)\|_{\mcal L(H)}\le 1$, $t\ge 0$,
\begin{align*}
	D_1\le \lambda_{N+1}^{-\frac{1+\sigma}{2}}\|X_0\|_{\mbf L^p(\Omega;\dot{H}^{1+\sigma})}\le K\lambda_{N+1}^{-\frac{1+\sigma}{2}}.
\end{align*}
It follows from  \eqref{semigroup1}, \eqref{FLip}, \eqref{Pn-Id}, the linear growth property of $F$, and \eqref{Xspatial} that
\begin{align*}
D_2\le&\; \Big\|\int_0^t E(t-s)P_N\big(F(Y^N(s))-F(X(s))\big)\ud s\Big\|_{\mbf L^p(\Omega;H)} \\
&\; +\Big\|\int_0^t E(t-s)\big(P_N-Id_H\big)F(X(s))\ud s\Big\|_{\mbf L^p(\Omega;H)} \\
\le &\; K\int_0^t\|Y^N(s)-X(s)\|_{\mbf L^p(\Omega;H)}\ud s \\
&\;+K\int_0^t\|(-A)^{\frac{1+\sigma}{2}}E(t-s)\|_{\mcal L(H)}\|(-A)^{-\frac{1+\sigma}{2}}\big(P_N-Id_H\big)\|_{\mcal L(H)}(1+\|X(s)\|_{\mbf L^p(\Omega;H)})\ud s \\
\le &\;K\int_0^t\|Y^N(s)-X(s)\|_{\mbf L^p(\Omega;H)}\ud s+K\lambda_{N+1}^{-\frac{1+\sigma}{2}}.
\end{align*}
By the BDG inequality and Minkowski inequality,
\begin{align*}
	D_3\le &\; \Big\|\Big(\int_0^t \big\|E(t-s)P_N\big(G(Y^N(s))-G(X(s))\big)\big\|_{\mcal L_2^0}^2\ud s\Big)^{\frac{1}{2}}\Big\|_{\mbf L^p(\Omega;\mbb R)} \\
	&\;+\Big\|\Big(\int_0^t \big\|E(t-s)(P_N-Id_H)\big(G(X(s))-G(X(t))\big)\big\|_{\mcal L_2^0}^2\ud s\Big)^{\frac{1}{2}}\Big\|_{\mbf L^p(\Omega;\mbb R)}\\
	&\;+\Big\|\Big(\int_0^t \big\|E(t-s)(P_N-Id_H)G(X(t))\big\|_{\mcal L_2^0}^2\ud s\Big)^{\frac{1}{2}}\Big\|_{\mbf L^p(\Omega;\mbb R)} \\
	=: D_{3,1}+D_{3,2}+D_{3,3}.
\end{align*}
Using \eqref{GLip} and the Minkowski inequality yields
\begin{align*}
D_{3,1}\le K\Big(\int_0^t\|Y^N(s)-X(s)\|_{\mbf L^p(\Omega;H)}^2\ud s\Big)^{\frac{1}{2}}.
\end{align*}
Applying \eqref{semigroup1}, \eqref{Pn-Id}, \eqref{GLip}, and \eqref{Xtemporal}, we arrive at
\begin{align*}
	D_{3,2}\le&\; K\|(-A)^{-\frac{1+\sigma}{2}}(P_N-Id_H)\|_{\mcal L(H)}\Big\|\Big(\int_0^t (t-s)^{-1-\sigma}\|X(t)-X(s)\|^2\ud s\Big)^{\frac{1}{2}}\Big\|_{\mbf L^p(\Omega;\mbb R)} \\
	\le &\; K\lambda_{N+1}^{-\frac{1+\sigma}{2}}\Big(\int_0^t(t-s)^{-1-\sigma}\|X(t)-X(s)\|_{\mbf L^p(\Omega;H)}^2\ud s\Big)^{\frac{1}{2}}\le K\lambda_{N+1}^{-\frac{1+\sigma}{2}}.
\end{align*} 
We infer from \eqref{semigroup3} and \eqref{Pn-Id} that
\begin{align*}
	&\;\int_0^t\|E(t-s)\big(P_N-Id_H\big)G(X(t))\|_{\mcal L_2^0}^2\ud s\\
	=&\;\sum_{i=1}^\infty \int_0^t\|(-A)^{\frac{1}{2}}E(t-s)(-A)^{-\frac{1+\sigma}{2}}\big(P_N-Id_H\big)(-A)^{\frac{\sigma}{2}}G(X(t))Q^{\frac{1}{2}}h_i\|^2\ud s \\
	\le&\; K  \sum_{i=1}^\infty \|(-A)^{-\frac{1+\sigma}{2}}\big(P_N-Id_H\big)\|_{\mcal L(H)}^2\|(-A)^{\frac{\sigma}{2}}G(X(t))Q^{\frac{1}{2}}h_i\|^2  \\
	\le&\; K\lambda_{N+1}^{-(1+\sigma)}\|(-A)^{\frac{\sigma}{2}}G(X(t))\|_{\mcal L_2^0}^2
	\le  K\lambda_{N+1}^{-(1+\sigma)}(1+\|X(t)\|_\sigma^2),
\end{align*}
which together with \eqref{Xspatial} immediately yields $D_{3,3}\le K\lambda_{N+1}^{-\frac{1+\sigma}{2}}$. In this way, we have
\begin{align*}
	D_3\le K\lambda_{N+1}^{-\frac{1+\sigma}{2}}+K\Big(\int_0^t\|Y^N(s)-X(s)\|_{\mbf L^p(\Omega;H)}^2\ud s\Big)^{\frac{1}{2}}.
\end{align*}
Combining the previous estimates for $D_i$, $i=1,2,3$, we obtain
\begin{align*}
	\|Y^N(t)-X(t)\|_{\mbf L^p(\Omega;H)}^2\le K\lambda_{N+1}^{-(1+\sigma)}+K\int_0^t\|Y^N(s)-X(s)\|_{\mbf L^p(\Omega;H)}^2\ud s.
\end{align*} 
Finally, the proof is complete based on the above formula and the Gronwall inequality.


\end{document}